\RequirePackage[l2tabu, orthodox]{nag}

\documentclass[oneside,article,a4paper]{memoir}

  \usepackage{amsmath}

  \usepackage{amsthm,thmtools}

  \usepackage{enumitem}

  \usepackage{tikz}

  \usepackage[colorlinks,linkcolor=blue,citecolor=blue,draft=false]{hyperref}

  \usepackage[capitalize,nosort]{cleveref}

  \usepackage[nobysame,shortalphabetic,initials]{amsrefs}

  \usepackage{mathtools}
  \usepackage{stmaryrd}

  \usepackage{relsize}
  \usepackage{stackengine}

  \usepackage[charter,cal=cmcal]{mathdesign}

  \usepackage[mathscr]{eucal}

  \usepackage{indentfirst}

  \usepackage{xspace}

  \usepackage[inline,nomargin]{fixme}
    \fxsetup{targetlayout=color}

  \isopage[12]
  \checkandfixthelayout

  \setsecnumdepth{subsection}
  \counterwithout{section}{chapter}

\thanksmarkseries{alph}

  \newcommand{\from}{\colon}

  \newcommand{\fto}{\twoheadrightarrow}

  \newcommand{\ito}{\hookrightarrow}

  \newcommand{\sto}{\rightarrowtriangle}

  \newcommand{\weto}{\mathrel{\ensurestackMath{\stackon[-2pt]{\xrightarrow{\makebox[.8em]{}}}{\mathsmaller{\mathsmaller\weq}}}}}

  \newcommand{\cto}{\rightarrowtail}

  \newcommand{\ato}{\mathrel{\ensurestackMath{\stackon[-5pt]{\cto}{\shortmid}}}}

  \newcommand{\tfto}{\mathrel{\ensurestackMath{\stackinset{l}{0.6pt}{c}{1pt}{\shortmid}{\to\mathrel{\mkern-16mu}\to}}}}

  \renewcommand{\iff}{if and only if}
  \newcommand{\st}{such that}
  
  \newcommand{\wrt}{with respect to}

  \newcommand{\llp}{left lifting property}
  \newcommand{\rlp}{right lifting property}
  
  \newcommand{\wfs}{weak factorisation system}

  \newcommand{\he}{homotopy equivalence}
  \newcommand{\fhe}{fiberwise homotopy equivalence}
  \newcommand{\whe}{weak homotopy equivalence}

  \newcommand{\He}{Homotopy equivalence}
  \newcommand{\Fhe}{Fiberwise homotopy equivalence}
  \newcommand{\Whe}{Weak homotopy equivalence}
  \newcommand{\Wfs}{Weak factorisation system}

  \newcommand{\cat}[1]{\mathcal{#1}}

  \newcommand{\ncat}[1]{\mathsf{#1}}

  \newcommand{\Pos}{\ncat{Pos}}

  \newcommand{\Set}{\ncat{Set}}

  \newcommand{\sSet}{\ncat{sSet}}

  \newcommand{\Simp}{\Delta}

  \newcommand{\cof}{\mathsf{cof}}

  \renewcommand{\emptyset}{{\mathord\varnothing}}

  \newcommand{\nat}{{\mathord\mathbb{N}}}

  \newcommand{\set}[2]{\left\{#1\mathrel{}\middle|\mathrel{}#2\right\}}

  \newcommand{\inter}{\cap}

  \newcommand{\union}{\cup}

  \newcommand{\weq}{\mathrel\sim}

  \newcommand{\heq}{\mathrel\simeq}

  \newcommand{\iso}{\mathrel{\cong}}

  \newcommand{\htp}{\mathrel\sim}

  \newcommand{\adj}{\dashv}

  \newcommand{\colim}{\operatorname{colim}}

  \let\bigcoprod\coprod

  \renewcommand{\coprod}{\sqcup}

  \newcommand{\tensor}{\odot}

  \newcommand{\cotensor}{\mathop\pitchfork}

  \newcommand{\pull}{\times}

  \newcommand{\push}{\sqcup}

  \newcommand{\etimes}{\mathop{\ensurestackMath{\stackunder[-3pt]{\times}{-}}}}

  \newcommand{\coend}{\int}

  \newcommand{\diag}{\operatorname{diag}}

  \newcommand{\op}{{\mathord\mathrm{op}}}

  \newcommand{\slice}{\mathbin\downarrow}

  \newcommand{\fslice}{\mathbin\twoheaddownarrow}

  \newcommand{\ob}{\operatorname{ob}}
  \newcommand{\mor}{\operatorname{mor}}

  \newcommand{\face}{\delta}

  \newcommand{\dgn}{\sigma}

  \newcommand{\bd}{\partial}

  \newcommand{\simp}[1]{\mathord\Delta[#1]}

  \newcommand{\bdsimp}[1]{\mathord{\bd\simp{#1}}}

  \makeatletter
  \def\horn#1{\expandafter\horn@i#1,,\@nil}
  \def\horn@i#1,#2,#3\@nil{\mathord\Lambda^{#2}[#1]}
  \makeatother

  \newcommand{\ssimp}[1]{\mathord\Delta_\sharp[#1]}

  \newcommand{\bdssimp}[1]{\mathord{\bd\ssimp{#1}}}

  \newcommand{\nerve}{\operatorname{N}}

  \newcommand{\sd}{\operatorname{sd}}
  \newcommand{\Sd}{\operatorname{Sd}}

  \newcommand{\Ex}{\operatorname{Ex}}

  \newcommand{\Sk}{\operatorname{Sk}}

  \newcommand{\id}{\operatorname{id}}

  \DeclareMathOperator{\cod}{cod}

  \newcommand{\uvar}{\mathord{\relbar}}

  \renewcommand{\epsilon}{\varepsilon}

  \renewcommand{\phi}{\varphi}

  \renewcommand{\bar}{\widebar}

  \renewcommand{\hat}{\widehat}

  \newcommand{\hatbin}[1]{\mathbin{\hat{#1}}}

  \renewcommand{\tilde}{\widetilde}

\DeclarePairedDelimiter\braces\lbrace\rbrace

\declaretheorem[style=plain,within=subsection]{corollary}
\declaretheorem[style=plain,numberlike=corollary]{lemma}
\declaretheorem[style=plain,numberlike=corollary]{proposition}
\declaretheorem[style=plain,numberlike=corollary]{theorem}
\declaretheorem[style=definition,numberlike=corollary]{remark}

\declaretheorem[style=plain,name=Corollary,within=section]{corollary-s}
\declaretheorem[style=plain,name=Lemma,numberlike=corollary-s]{lemma-s}
\declaretheorem[style=plain,name=Proposition,numberlike=corollary-s]{proposition-s}
\declaretheorem[style=definition,name=Remark,numberlike=corollary-s]{remark-s}

\Crefname{corollary}{Corollary}{Corollaries}
\Crefname{lemma}{Lemma}{Lemmas}
\Crefname{proposition}{Proposition}{Propositions}
\Crefname{theorem}{Theorem}{Theorems}
\Crefname{remark}{Remark}{Remarks}

\Crefname{corollary-s}{Corollary}{Corollaries}
\Crefname{lemma-s}{Lemma}{Lemmas}
\Crefname{proposition-s}{Proposition}{Propositions}
\Crefname{remark-s}{Remark}{Remarks}

  \usetikzlibrary{matrix,arrows,arrows.meta,positioning,decorations.markings,patterns,calc}

  \newenvironment{tikzeq}[1]
  {
    \begingroup
    \begin{equation}\label{#1}
    \begin{tikzpicture}[baseline=(current bounding box.center)]
  }
  {
    \end{tikzpicture}
    \end{equation}
    \endgroup
    \ignorespacesafterend
  }

  \newenvironment{tikzeq*}
  {
    \begingroup
    \begin{equation*}
    \begin{tikzpicture}[baseline=(current bounding box.center)]
  }
  {
    \end{tikzpicture}
    \end{equation*}
    \endgroup
    \ignorespacesafterend
  }

  \tikzset
  {
    diagram/.style=
    {
      matrix of math nodes,
      column sep={4.3em,between origins},
      row sep={4em,between origins},
      text height=1.5ex,
      text depth=.25ex
    },
    over/.style={preaction={draw=white,-,line width=6pt}},
    every to/.style={font=\footnotesize},
    inj/.style={right hook->},
    surj/.style={-{Latex[open]}},
    cof/.style={>->},
    fib/.style={->>},
    strike/.style={decoration={markings,mark=at position 0.5 with {\arrow{|}}},postaction={decorate}},
    ano/.style={cof,strike},
    tfib/.style={fib,strike},
    weq/.style={->},
  }

  \newcommand{\pb}[2]{\node at ($(#1)!0.25!(#2)$) {\tikz{\draw (3mm,0)--++(-90:3mm)--++(180:3mm)}}}

  \newcommand{\pbdr}[2]{\node at ($(#1)!0.25!(#2)$) {\tikz{\draw (3mm,180)--++(90:3mm)--++(0:3mm)}}}

  \Crefname{subsection}{Subsection}{Subsections}

  \setlist[enumerate]{label=(\roman*),itemsep=0ex}
  \Crefname{enumi}{Part}{Parts}
  \crefname{enumi}{part}{parts}

  \newlist{conditions}{enumerate}{1}
  \setlist[conditions]{label=(\roman*),itemsep=0ex}
  \Crefname{conditionsi}{Condition}{Conditions}
  \crefname{conditionsi}{condition}{Conditions}

  \newcommand{\axm}[1]{(#1\arabic*)}
  \newlist{axioms}{enumerate}{1}
  \Crefname{axiomsi}{Axiom}{Axioms}
  \crefname{axiomsi}{axiom}{axioms}

  \newlist{fibcat-axioms}{enumerate}{1}
  \setlist[fibcat-axioms]{label=\axm{F},resume}
  \Crefname{fibcat-axiomsi}{Axiom}{Axioms}
  \crefname{fibcat-axiomsi}{axiom}{axioms}

  \makeatletter
  \newcommand{\customlabel}[2]{#2\def\@currentlabel{#2}\label{#1}}
  \makeatother

  \DeclareFontFamily{U}{mathx}{\hyphenchar\font45}

  \DeclareFontShape{U}{mathx}{m}{n}{
    <5> <6> <7> <8> <9> <10>
    <10.95> <12> <14.4> <17.28> <20.74> <24.88>
    mathx10}{}

  \DeclareSymbolFont{mathx}{U}{mathx}{m}{n}

  \DeclareMathAccent{\widebar}{0}{mathx}{"73}

  \DeclareFontFamily{U}{MnSymbolA}{}

  \DeclareFontShape{U}{MnSymbolA}{m}{n}{
    <-6> MnSymbolA5
    <6-7> MnSymbolA6
    <7-8> MnSymbolA7
    <8-9> MnSymbolA8
    <9-10> MnSymbolA9
    <10-12> MnSymbolA10
    <12-> MnSymbolA12}{}

  \DeclareSymbolFont{MnSyA}{U}{MnSymbolA}{m}{n}

  \DeclareMathSymbol{\twoheaddownarrow}{\mathrel}{MnSyA}{27}

\makeatletter
\newcommand{\DeclareAbbrevation}[2]{\newcommand{#1}{\@ifnextchar{.}{#2}{#2.\@\xspace}}}
\makeatother

\DeclareAbbrevation{\ie}{i.e}
\DeclareAbbrevation{\eg}{e.g}
\DeclareAbbrevation{\cf}{cf}
\DeclareAbbrevation{\etc}{etc}
\DeclareAbbrevation{\resp}{resp}
\DeclareAbbrevation{\etal}{et al}
\DeclareAbbrevation{\ibid}{ibid}
\DeclareAbbrevation{\ca}{ca}
\DeclareAbbrevation{\vs}{vs}

  \author{Nicola Gambino\thanks{School of Mathematics, University of Leeds. Email: n.gambino@leeds.ac.uk} \and 
  Christian Sattler\thanks{Department of Computer Science and Engineering, Chalmers University of Technology. Email: sattler@chalmers.se} \and
  Karol Szumi{\l}o\thanks{School of Mathematics, University of Leeds. Email:  k.szumilo@leeds.ac.uk}}

  \title{The Constructive Kan--Quillen Model Structure: Two New Proofs}
  \date{\today}

\begin{document}

  \maketitle

  \begin{abstract}
    We present two new proofs of Simon Henry's result that
    the category of simplicial sets admits a constructive counterpart of
    the classical Kan--Quillen model structure.
    Our proofs are entirely self-contained and
    avoid complex combinatorial arguments on anodyne extensions.
    We also give new constructive proofs of
    the left and right properness of the model structure.
  \end{abstract}

\section*{Introduction}

The Kan--Quillen model structure on simplicial sets, \ie,
the model structure in which the fibrant objects are the Kan complexes and
the cofibrations are the monomorphisms~\cite{Quillen},
has long been recognised as the cornerstone of
modern simplicial homotopy theory~\cite{Goerss-Jardine}.
Over the past decade, however,
this model structure has become of great interest also in mathematical logic and
theoretical computer science, since it provides inspiration for
Voevodsky's Univalent Foundations programme~\cite{Voevodsky-MSCS} and
Homotopy Type Theory~\cite{HoTT-book}.
In particular, it plays an essential role in
the simplicial model of Univalent Foundations~\cites{Kapulkin}.

While there are several proofs of the existence of
the Kan--Quillen model structure~\cites{Goerss-Jardine,jt,Cisinski,Moss,Sattler},
all of them use non-constructive principles, \ie,
the law of excluded middle (EM) and the axiom of choice (AC).
Since these principles are not generally valid in
the internal logic of a Grothendieck topos,
the construction of an analogue of the Kan--Quillen model structure on
simplicial sheaves is very subtle~\cites{Jardine,Joyal}.
This situation is also an obstacle to the definition of
a constructive version of the simplicial model of Univalent Foundations, which
is still an open problem.
Furthermore, the results in~\cite{bcp} show that
some results on Kan fibrations are simply not provable constructively.

Recently, Simon Henry obtained a breakthrough result by
establishing a constructive counterpart of
the Kan--Quillen model structure~\cite{Henry-qms}, namely
a model structure whose existence can be proved using constructive methods, but
coincides with the usual model structure once EM and AC are assumed.
A key aspect of this model structure is that, in contrast to the classical case,
not all objects are cofibrant, but
only those in which degeneracy of simplices is decidable.
The existence of this model structure has already been applied to
provide a partial solution to the problem of
giving a constructive simplicial model of
Univalent Foundations~\cite{Gambino-Henry} and suggests
the possibility of defining a new model structure on simplicial sheaves.
Indeed, the results in this paper have led to the construction of a new model structure on
categories of simplicial objects in countably lextensive categories, obtained in collaboration with Simon
Henry~\cite{Gambino-Henry-Sattler-Szumilo}.

The main goal of this paper is to give two new proofs of the existence of
the constructive Kan--Quillen model structure.
We believe these proofs to be simpler than Henry's proof~\cite{Henry-qms},
formulated in clear category-theoretic terms and essentially self-contained.
In contrast, Henry's proof uses subtle combinatorial arguments on anodyne maps,
including results that do not seem to have been known even in a classical setting,
and relies on his earlier work on weak model structures~\cite{Henry-wms}.
Because of our category-theoretic approach, our arguments are
potentially easier to generalise so as to obtain
a new model structure on simplicial sheaves,
a task that we leave for future research.
We also provide two new proofs of the left and right properness of
the constructive Kan--Quillen model structure, which
were also already proved in~\cite{Henry-qms}.

Overall, this paper establishes
all the results of constructive simplicial homotopy theory needed for~\cite{Gambino-Henry}.
Indeed, the desire to give self-contained and streamlined proofs of
these results was one of the motivations for this paper, which
can then be seen as contributing to the effort to define
a constructive simplicial model of Univalent Foundations. We also give new proofs of
some results in~\cite{Gambino-Henry} and establish constructive versions of
well-known theorems, such as Quillen's Theorem~A.

Cofibrancy considerations play a key role in both of our proofs.
On the one hand, we had to check carefully that
the decidability assumptions encapsulated in the notion of cofibrancy allow us
to carry over some classical arguments.
This is sometimes subtle, for example when extending results about
the $\Ex$ functor to the $\Ex^\infty$ functor for our first proof and
when proving a version of the equivalence extension property for
our second proof.
On the other hand, we also had to develop new arguments, necessary to
extend results from the full subcategory of cofibrant simplicial sets to
the category of all simplicial sets, which do not have counterparts in the classical setting.
Furthermore, this situation requires us to work with more notions of
\whe{} than in the classical setting and then check that
they are mutually consistent.

We should mention that we had to refine the assumptions of the arguments in~\cite{Sattler} before we could apply them in the cofibrant fragment.
For example, the category of cofibrant simplicial sets cannot constructively be shown locally cartesian closed.
However, the exponentials appearing in the proof of the equivalence extension property nevertheless preserve cofibrancy, and this verification is rather subtle.
For these reasons,  we hope that the methods developed in this paper are
valuable not only for obtaining a new model structure on simplicial sheaves but
also for defining other model structures in which not all objects are cofibrant.
We will comment in more detail on the differences between our proofs and Simon Henry's proof
in \cref{thm:first-diff-henry,thm:second-diff-henry,thm:properness-diff-henry},
after concluding our proofs of the existence of the model structure and of
its properness.

We regret that this paper is longer than we originally intended, but we hope that kind readers will appreciate that the proofs are given in some detail, hopefully making our results more widely accessible.

\smallskip

\noindent
\textbf{Outline of the paper.}
Our development begins in \cref{sec:preliminaries} with
some material useful for both proofs.
We begin in \cref{sec:dec-inc} with some remarks on decidable inclusions and split surjections
in the category of sets. These are used in \cref{sec:ssets}  to define
 the \wfs{}s on simplicial sets  of
cofibrations and trivial fibrations and of trivial cofibrations and fibrations.
The pushout product properties for these weak factorisation systems are proved in \cref{sec:two-wfss}.
We then identify the weak factorisation system of cofibrations and trivial fibrations as the
Reedy weak factorisation system on simplicial sets induced by the \wfs{} of decidable inclusions
and split surjections on sets in \cref{sec:cofibrations}.
We conclude the section by introducing weak homotopy
equivalences in \cref{sec:htpy-and-whe} and
proving that, for a cofibrant simplicial set $B$,
the full subcategory of the slice $\sSet \slice B$ spanned by
fibrations with cofibrant domain is a fibration category in \cref{fibration-category}.

\Cref{sec:first-proof} presents our first proof of the existence of
the Kan--Quillen model structure, which is inspired by
classical ideas of simplicial homotopy theory, in particular~\cite{ltw}.
The proof is organised in five steps, each presented in a subsection.
In \cref{sec:the-cofibration-category-of-cofibrant-ssets}, we show that
the full subcategory of $\sSet$ spanned by cofibrant simplicial sets admits
the structure of a cofibration category.
In \cref{sec:diagonals-of-ssets}, we obtain a constructive proof of
the so-called diagonal lemma, asserting that if
a map between cofibrant bisimiplicial sets is
pointwise a weak homotopy equivalence, then so is its diagonal.
In \cref{sec:subdivision-and-ex}, we prove counterparts of
standard facts on Kan's $\Ex^\infty$ functor on cofibrant simplicial sets.
In \cref{sec:explicit-cofibrant-replacement}, we prove
a version of Quillen's Theorem A and use it to introduce
a cofibrant replacement functor with good properties.
Finally, in \cref{sec:the-model-structure-first-proof}, we combine these results
to present our first proof of the existence of the model structure.

\Cref{sec:second-proof} presents our second proof, which is based on
the ideas in~\cites{Gambino-Sattler,Sattler}.
The proof is organised in three subsections.
In \cref{sec:frobenius-property}, we establish a restricted version of
the Frobenius property, showing that
trivial cofibrations are closed under pullback along
fibrations with cofibrant domain.
In \cref{sec:equivalence-extension-property}, we prove
the so-called equivalence extension property in
the full subcategory of simplicial sets spanned by cofibrant objects.
In \cref{sec:model-structure-second-proof}, we combine these results to establish the restriction of the model structure to cofibrant objects.
We then extend this model structure to all simplicial sets.

Our two new proofs of the left and right properness of the model structure are
presented in \cref{sec:properness}.
Here, it should be noted that, in contrast with the classical setting,
left properness is not immediate since not every object is cofibrant.
For right properness, one proof uses $\Ex^\infty$ functor, while
the other uses the Frobenius property.
For left properness, both proofs use an argument dual to that
for right properness using $\Ex^\infty$.

\smallskip

\noindent
\textbf{Remarks on constructivity.} To fix ideas, we shall work in Constructive Zermelo--Fraenkel set theory (CZF), a set theory based on constructive logic \cite{Aczel}. See~\cite{Aczel-Rathjen} for more information on CZF.
Readers who are unfamiliar with constructive set theory may think of our category of sets as being an arbitrary Grothendieck topos. By a finite set, we always mean
a set with a bijection to a set of the form $\braces{1, \ldots, n}$ for some $n \in \mathbb{N}$. Equality on such sets is always decidable, which is why
they are sometimes referred to also as finite decidable sets.
To simplify our presentation, we adopt an abuse of language and say ``for all $\ldots$ there exists $\ldots$'' to mean that we have a function giving witnesses for existential quantifiers. In particular, when we speak of a map having a right (or left) lifting property with respect to a given class of maps
and in particular in the notion of a \wfs{} and a model structure, we mean that the map is equipped with a  function providing diagonal fillers for the appropriate class of diagrams.
Here, a class of maps means a class together with a forgetful function to the class of maps.
However, after we have established the model structure, one can also derive a variation with subclasses, where for example the fibrations are the subclass of maps for which there exists a function providing diagonal fillers.

\smallskip

\noindent
\textbf{Notation.}  We will use distinct notations for various types of morphisms, which
are summarised in \cref{tab:arrows} for the convenience of the readers.

\begin{table}[htb]
\begin{tikzeq*}
\matrix[diagram]
{
  |(weq)| \text{Acyclic map (weak equivalence)}  &[4em] |(weql)|  & |(weqr)| \\[-4ex]
  |(f)| \text{Fibration}                         &[4em] |(fl)|    & |(fr)|   &[7em]
  |(c)| \text{Cofibration}                       &[4em] |(cl)|    & |(cr)|   \\[-4ex]
  |(tf)| \text{Trivial fibration}                &      |(tfl)|   & |(tfr)|  &
  |(tc)| \text{Trivial cofibration}              &      |(tcl)|   & |(tcr)|  \\[-4ex]
  |(af)| \text{Acyclic fibration}                &      |(afl)|   & |(afr)|  &
  |(ac)| \text{Acyclic cofibration}              &      |(acl)|   & |(acr)|  \\[-4ex]
  |(m)| \text{Monomorphism / Face operator}      &      |(ml)|    & |(mr)|   &
  |(e)| \text{Epimorphism / Degeneracy operator} &      |(el)|    & |(er)|   \\[-4ex]
};

\draw[weq] (weql) to node[above] {$\weq$} (weqr);
\draw[fib]  (fl)  to (fr);
\draw[cof]  (cl)  to (cr);
\draw[tfib] (tfl) to (tfr);
\draw[ano]  (tcl) to (tcr);

\draw[fib] (afl) to node[above] {$\weq$} (afr);
\draw[cof] (acl) to node[above] {$\weq$} (acr);

\draw[inj]  (ml) to (mr);
\draw[surj] (el) to (er);
\end{tikzeq*}
\caption{Notation for different classes of morphisms.}
\label{tab:arrows}
\end{table}

\smallskip

\noindent
\textbf{Acknowledgements.}
Nicola Gambino gratefully acknowledges
the support of the US Air Force Office for Scientific Research
under agreements FA8655-13-1-3038 and FA9550-21-1-0007
and of EPSRC under
grant EP/M01729X/1, as well as the hospitality
of the University of Manchester during a study leave from the University of Leeds.
Christian Sattler was supported by Swedish Research Council grant 2019-03765.
We are very grateful to Simon Henry for helpful discussions,
sharing early drafts of his work and for pointing out a mistake in a previous version of this paper.
We also thank Steve Awodey, Thierry Coquand, Andr\'e Joyal and Raffael Stenzel for
useful conversations and the referee for comments which helped to improve the exposition.

  \section{Preliminaries}
  \label{sec:preliminaries}

\subsection{Decidable inclusions}
\label{sec:dec-inc}

We begin by verifying some basic properties of the category of sets in our constructive setting. In particular, we will show that it admits a \wfs{} consisting of
decidable inclusions and split surjections. This will be useful to construct \wfs{}s on simplicial sets in~\cref{sec:ssets}. Recall that a map of sets $i \from A \to B$ is a \emph{decidable inclusion}
if there is a map $j \from C \to B$ \st{}
$i$ and $j$ exhibit $B$ as a coproduct of $A$ and $C$.
The map $j$ (or just the set $C$ by abuse of language) is called
the \emph{complement} of $i$ (or $A$).
A \emph{split surjection} is a function that admits a section.
Note that, assuming EM, every injective function is a decidable inclusion and,
assuming AC, every surjection is split.

We recall the notion of a van Kampen colimit (sometimes also called descent for colimits) in a category $\cat{C}$ with pullbacks.
A colimit in $\cat{C}$ is said to be \emph{van Kampen} if it is preserved by the bifunctor $\cat{C} \slice \uvar$ from $\cat{C}^\op$ to categories whose functorial action is given by reindexing.
If $\cat{C}$ is cocomplete, this concretely means the following for a colimit $\colim A$: given a natural transformation $u \from X \to A$ that is cartesian, \ie, whose naturality squares are pullbacks, a cocone under $X$ with summit $\overline{X}$, and a map $\overline{X} \to \colim A$ cohering with $X \to A$, the cocone with summit $\overline{X}$ is colimiting if and only if the square from $X_s \to A_s$ to $\overline{X} \to \colim A$ is a pullback for each index $s$.
If pullback functors preserve colimits, such as in $\Set$, then the reverse direction of this equivalence holds.

Certain classes of van Kampen colimits have special names.
We call $\cat{C}$:
\begin{itemize}
\item \emph{extensive} if coproducts exist and are van Kampen;
\item \emph{adhesive} if pushouts along monomorphisms exist and are van Kampen;
\item $\omega$-\emph{exhaustive}~\cite{nlab-exhaustive} if sequences of monomorphisms (forming $\omega$-indexed diagrams in $\cat{C}$) have colimits that are van Kampen.
\end{itemize}

The next lemma is well-known classically and essentially straightforward also constructively.

\begin{lemma}\label{Set-van-Kampen}
  The category of sets is 
  \begin{enumerate}
  \item \label{item:extensive} extensive;
  \item \label{item:adhesive} adhesive;
  \item \label{item:exhaustive} $\omega$-exhaustive.
  \end{enumerate}
\end{lemma}

\begin{proof}
For~\cref{item:extensive}, the claim is that, given functions $X_i \to A_i$ for $i \in I$, each square
\begin{tikzeq*}
\matrix[diagram]
{
  |(Xi)| X_i & |(X)| \bigcoprod_i X_i \\
  |(Ai)| A_i & |(A)| \bigcoprod_i A_i \\
};

\draw[->] (Xi) to (Ai);
\draw[->] (X)  to (A);
\draw[->] (Xi) to (X);
\draw[->] (Ai) to (A);
\end{tikzeq*}
is a pullback.
This is immediate.

The remaining parts are consequences of effectivity of quotients of equivalence relations.
For \cref{item:adhesive}, the claim reduces by \cite{garner-lack-adhesive}*{Theorem~A} to pushouts along monomorphisms being pullback squares.
Consider a pushout square
\begin{tikzeq*}
\matrix[diagram]
{
  |(A)| A & |(C)| C \\
  |(B)| B & |(P)| P \\
};

\draw[inj] (A) to node[left]  {$f$} (B);
\draw[->]  (A) to node[above] {$g$} (C);
\draw[->]  (B) to (P);
\draw[->]  (C) to (P);
\end{tikzeq*}
with $f \colon A \ito B$ a monomorphism. The pushout $P$ is the quotient of the equivalence relation on $B \coprod C$ generated by identifying the images of $f(a)$ and $g(a)$ for $a \in A$.
Given $b \in B$ and $c \in C$ mapping to the same element of $P$, effectivity of the quotient implies that $b$ and $c$ are related in $B \coprod C$ by the equivalence relation.
A direct calculation using that $f$ is a monomorphism shows that there exists a unique $a \in A$ such that $f(a) = b$ and $g(a) = c$.
This makes the square a pullback.

For \cref{item:exhaustive}, consider a sequence of monomorphisms $(A_i \ito A_{i+1})_{i \in \omega}$.
The colimit $A_\omega$ is the quotient of the equivalence relation on $\bigcoprod_i A_i$ that identifies $a_j \in A_j$ and $a_k \in A_k$ if they coincide in $A_{\max(j, k)}$.
Effectivity of the quotient implies that the coprojections $A_i \to A_\omega$ are monomorphisms.
Consider now another sequence $(X_i \to X_{i+1})_{i \in \omega}$ with a cartesian natural transformation $f \from X \to A$.
In particular, each map $X_i \to X_{i+1}$ is a monomorphism, and so we may construct the colimit $X_\omega$ as above.
Given $i \in \omega$, the claim is that the induced square
\begin{tikzeq*}
\matrix[diagram]
{
  |(Xi)| X_i & |(X)| X_\omega \\
  |(Ai)| A_i & |(A)| A_\omega \\
};
\draw[->]  (Xi) to node[left]  {$f_i$}     (Ai);
\draw[->]  (X)  to node[right] {$f_\omega$} (A);
\draw[inj] (Xi) to (X);
\draw[inj] (Ai) to (A);
\end{tikzeq*}
is a pullback.
Given $x_\omega \in X_\omega$ such that $f_\omega(x_\omega)$ lifts to $A_i$, we have to show that $x_\omega$ lifts to $X_i$.
By construction of $X_\omega$ as a quotient of $\bigcoprod_j X_j$, there exists $j \geq i$ such that $x_\omega$ lifts to $x_j \in X_j$.
Since the naturality square of $f$ at $i \to j$ is cartesian, $x_j$ lifts to $X_i$.
\end{proof}

\begin{remark}
Categorically, \cref{item:adhesive,item:exhaustive} of \cref{Set-van-Kampen} say that a Barr-exact category is adhesive if it is finitely extensive and $\omega$-exhaustive if it is countably extensive.
Although this seems folklore, we were unable to locate precise references.
For adhesivity, see~\cite{Johnstone}*{Lemma A.2.4.3} for a related argument in the context of topos theory (using a subobject classifier).
For exhaustivity, see~\cite{nlab-exhaustive} for a criterion (mirroring that of \cite{garner-lack-adhesive}*{Theorem~A} for adhesivity) that we have essentially followed in our proof.
\end{remark}

\begin{remark}
Although we liberally use colimits of sets in this paper, all appearing colimits can be built from countable coproducts.
In particular, pushouts along monomorphisms and sequential colimits of monomorphisms will be considered only for \emph{decidable inclusions}.
This makes the development independent from the availability of quotients.
\end{remark}

The results in the rest of this subsection depend only on the fact that the category of sets is extensive.
The other van Kampen properties established by \cref{Set-van-Kampen} will be used in \cref{fibration-extension}.

\begin{lemma}\label{decidable-properties}
  In the category of sets:
  \begin{enumerate}
  \item \label{decidable-monomorphism}
    decidable inclusions are monomorphisms;
  \item \label{decidable-pullback} decidable inclusions are
    closed under pullback along arbitrary functions;
  \item \label{decidable-retract}
    decidable inclusions are closed under retracts.
  \end{enumerate}
\end{lemma}

\begin{proof}
  All these are consequences of extensivity.
  Parts \ref{decidable-monomorphism} and~\ref{decidable-pullback} follow from
  \cite{clw-extensive}*{Proposition~2.6}.
  Moreover, if $i$ is a retract of a decidable inclusion $j$, then
  it is also a pullback of $j$, and thus \cref{decidable-retract} is
  a consequence of \cref{decidable-pullback}.
\end{proof}

\begin{lemma}\label{pullback-coproduct}
  In the category of sets, pullbacks commute with coproducts.
\end{lemma}

\begin{proof}
  Let $A_i \to B_i \leftarrow C_i$ be a family of cospans indexed by $i \in I$.
  In the cube
  \begin{tikzeq*}
  \matrix[diagram,column sep={between origins,6em}]
  {
      |(p)| A_i \pull_{B_i} C_i & & |(pc)| \bigcoprod_i A_i \pull_{\bigcoprod_i B_i} \bigcoprod_i C_i & \\
    & |(C)| C_i                 & & |(cC)| \bigcoprod_i C_i                                             \\
      |(A)| A_i                 & & |(cA)| \bigcoprod_i A_i                                           & \\
    & |(B)| B_i                 & & |(cB)| \bigcoprod_i B_i                                             \\
  };

  \draw[->,shorten <=0.5mm]               (p)  to (A);
  \draw[->,shorten <=1mm,shorten >=0.5mm] (pc) to (cA);
  \draw[->]                               (p)  to (pc);
  \draw[->]                               (A)  to (cA);

  \draw[->,over]                            (C)  to (B);
  \draw[->,shorten <=0.5mm,shorten >=0.5mm] (cC) to (cB);
  \draw[->,over]                            (C)  to (cC);
  \draw[->]                                 (B)  to (cB);

  \draw[->]                                 (p)  to (C);
  \draw[->,shorten <=0.5mm,shorten >=0.5mm] (pc) to (cC);
  \draw[->]                                 (A)  to (B);
  \draw[->,shorten <=0.5mm,shorten >=0.5mm] (cA) to (cB);
  \end{tikzeq*}
  the left and right faces are pullbacks by construction.
  The front face is a pullback by extensivity and thus so is the back one.
  Using the extensivity once more, we conclude that
  $\bigcoprod_i A_i \pull_{\bigcoprod_i B_i} \bigcoprod_i C_i$ coincides with
  $\bigcoprod_i (A_i \pull_{B_i} C_i)$.
\end{proof}

\begin{proposition}\label{decidable-wfs}
  Decidable inclusions and split surjections form a \wfs{}.
  The \wfs{} is cofibrantly generated by $\{ \emptyset \to 1 \}$.
\end{proposition}

\begin{proof}
  Every map $f \from S \to T$ factors in an evident way as $S \to S \coprod T \to T$,
  where the first map is a decidable inclusion by construction and
  the second one has a section given by
  the coproduct inclusion $T \to S \coprod T$.
  The lifting properties are immediate.

For the claim on cofibrant generation, first note that  $\emptyset \to 1$ is a decidable inclusion.
 Also,  a map $X \to Y$ with the \rlp{} \wrt{} $\emptyset \to 1$ is a split surjection.
  Indeed, it has the \rlp{} also \wrt{} $\emptyset \to Y$ since that map is
  the coproduct $\bigcoprod_{y \in Y} \emptyset \to \bigcoprod_{y \in Y} 1$.
\end{proof}

Since decidable inclusions $A \to B$ are monomorphisms
by \cref{decidable-monomorphism} of \cref{decidable-properties},
they can be seen as subobjects of $B$.
Subobjects corresponding to decidable inclusions will be called
\emph{decidable subsets}. We now establish some closure properties of
decidable subsets.

\begin{lemma}\label{decidable-closure}
  \leavevmode
  \begin{enumerate}
  \item \label{decidable-union}
    Decidable subsets are closed under finite unions.
  \item \label{decidable-limit}
  Decidable inclusions are closed under finite limits, \ie,
  if $X \to Y$ is
  a natural transformation between finite diagrams of sets that is
  a levelwise decidable inclusion,
  then so is the induced function $\lim X \to \lim Y$.
  \end{enumerate}
\end{lemma}

\begin{proof} For \cref{decidable-union}, let $B$ be an arbitrary set.
  The nullary case is clear since $\emptyset$ is a decidable subset of $B$.
  For the binary case, consider decidable subsets $A_0$ and $A_1$ with
  complements $C_0$ and $C_1$.
  The intersection and the union of $A_0$ and $A_1$ are given by
  the pullback square on the left and the pushout square on the right:
  \begin{tikzeq*}
  \matrix[diagram]
  {
    |(l)|  A_\emptyset & |(0l)| A_0 & |(r)|  A_\emptyset & |(0r)|  A_0    &[-1em]                    \\
    |(1l)| A_1         & |(Bl)| B   & |(1r)| A_1         & |(01r)| A_{01} &                          \\[-4ex]
                       &            &                    &                & |(Br)| B \rlap{\text{.}} \\
  };

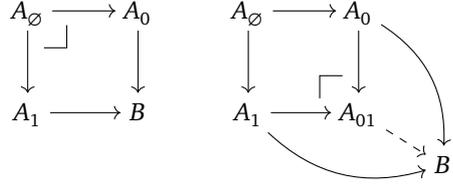
  \draw[->] (l)  to (0l);
  \draw[->] (1l) to (Bl);
  \draw[->] (l)  to (1l);
  \draw[->] (0l) to (Bl);

  \draw[->] (r)  to (0r);
  \draw[->] (1r) to (01r);
  \draw[->] (r)  to (1r);
  \draw[->] (0r) to (01r);

  \draw[->] (0r) to[bend left]  (Br);
  \draw[->] (1r) to[bend right] (Br);

  \draw[->,dashed] (01r) to (Br);

  \pb{l}{Bl};
  \pbdr{01r}{r};
  \end{tikzeq*}
  By the extensivity of the category of sets, we have an isomorphism
  $ B \iso (A_0 \pull_B A_1) \coprod (A_0 \pull_B C_1)
      \coprod (C_0 \pull_B A_1) \coprod (C_0 \pull_B C_1)$,
  where the first summand is $A_\emptyset$.
  Call the other three $C'_0$, $C'_1$ and $C'_{01}$.
  Then the map $A_{01} \to B$ is isomorphic to
  \begin{equation*}
    A_\emptyset \coprod C'_0 \coprod C'_1 \to A_\emptyset
                \coprod C'_0 \coprod C'_1 \coprod C'_{01}
  \rlap{,}\end{equation*}
  which is a decidable inclusion. The case of general finite unions follows by induction.

For \cref{decidable-limit},
  the conclusion holds trivially for terminal objects and
  it will be enough to verify it for pullbacks.
  (See \cite{awodey-ct}*{Proposition~5.21}.)
  Consider $X$ and $Y$ as cospans indexed over $0 \to 01 \leftarrow 1$ and
  assume that $X_0 \to Y_0$, $X_1 \to Y_1$ and $X_{01} \to Y_{01}$ are
  decidable inclusions.
  First, we treat the case when
  both $X_0 \to Y_0$, $X_1 \to Y_1$ are isomorphisms.
  In this case, the rows of the diagram
  \begin{tikzeq*}
  \matrix[diagram]
  {
    |(X0)|  X_0    & |(Y0)|  Y_0    & |(C0)|  \emptyset \\
    |(X01)| X_{01} & |(Y01)| Y_{01} & |(C01)| C_{01}    \\
    |(X1)|  X_1    & |(Y1)|  Y_1    & |(C1)|  \emptyset \\
  };

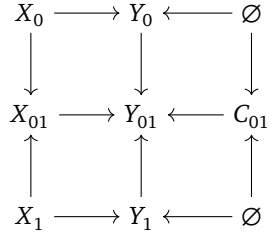
  \draw[->] (X0) to (Y0);
  \draw[->] (C0) to (Y0);

  \draw[->] (X01) to (Y01);
  \draw[->] (C01) to (Y01);

  \draw[->] (X1) to (Y1);
  \draw[->] (C1) to (Y1);

  \draw[->] (X0) to (X01);
  \draw[->] (X1) to (X01);

  \draw[->] (C0) to (C01);
  \draw[->] (C1) to (C01);

  \draw[->] (Y0) to (Y01);
  \draw[->] (Y1) to (Y01);
  \end{tikzeq*}
 are coproduct diagrams.
  Since pullbacks commute with coproducts
  by \cref{pullback-coproduct}, the induced diagram
  $X_0 \pull_{X_{01}} X_1 \to Y_0 \pull_{Y_{01}} Y_1 \leftarrow \emptyset \pull_{C_{01}} \emptyset$
  is also a coproduct.
  The map on the left is therefore a decidable inclusion.
  Next, assume only that $X_0 \to Y_0$ is an isomorphism.
  We pull back the coproduct diagram
  $X_1 \to Y_1 \leftarrow C_1$
  along $Y_0 \to Y_{01}$, which yields the coproduct diagram
  $Y_0 \pull_{Y_{01}} X_1 \to Y_0 \pull_{Y_{01}} Y_1 \leftarrow Y_0 \pull_{Y_{01}} C_1$.
  By the preceding case, $X_0 \pull_{X_{01}} X_1 \to Y_0 \pull_{Y_{01}} X_1$ is
  a decidable inclusion and thus
  the composite $X_0 \pull_{X_{01}} X_1 \to Y_0 \pull_{Y_{01}} Y_1$ is
  a decidable inclusion.
  The general statement is reduced to this case in the same way.
\end{proof}

Another approach to proving \cref{decidable-limit} above
reduces limits in the arrow category to limits in slice categories.
Let $I$ denote the indexing category;
we only require that $I$ has a finite set of objects.
In the slice over $\lim Y$, the object $\lim X$ is
the limit of $\lim Y \pull_{Y_i} X_i$ over $i \in I$.
The maps $\lim Y \pull_{Y_i} X_i \to \lim Y$ are decidable inclusions
by \cref{decidable-pullback} of \cref{decidable-properties},
in particular monomorphisms
by \cref{decidable-monomorphism} of \cref{decidable-properties}.
As subobjects of $\lim Y$, their limit is isomorphic to their intersection.
But decidable subsets are closed under finite intersection by
the dual version of~\cref{decidable-union}.

\subsection{Simplicial sets and the weak factorisation systems}
\label{sec:ssets}

We now move on to consider the category of simplicial sets and define the two weak factorisation systems that will be part of the constructive Kan--Quillen model
structure. For this, let us fix some notation and terminology.
We write $\Simp$ for the \emph{category of simplices}, \ie,
of non-empty finite ordinals, written $[n]$, for $n \in \nat$, and
order-preserving maps, to which we refer as \emph{simplicial operators}.
This category is a Reedy category~\cite{hovey}*{Chapter~5}.
Morphisms of its direct part $\Simp_\sharp$ are called \emph{face operators},
with generators denoted $\face_i \from [n-1] \to [n]$ (omitting $i$);
morphisms of its inverse part $\Simp_\flat$ are called \emph{degeneracy operators},
with generators denoted $\dgn_i \from [n+1] \to [n]$ (identifying $i$ and $i + 1$).
For a simplicial operator $\phi$, we write $\phi = \phi^\sharp \phi^\flat$
for its unique decomposition into a degeneracy operator followed by a face operator.

The category of simplicial sets $\sSet$ is the category of presheaves over $\Simp$.
We write $\simp{m}$ for the representable simplicial set represented by $[m]$.
Our convention is that simplicial operators act on the right,
\ie, if $X \in \sSet$, $x \in X_n$ and $\phi \from [m] \to [n]$, then
the image of $x$ under the action of $\phi$ is denoted by $x \phi$.

Being a presheaf category, $\sSet$ admits all (small) limits and colimits and
 is locally cartesian closed.
For $M \in \sSet$, we write $\sSet \slice M$ for the slice category over~$M$.
With a slight abuse of notation,
we sometimes refer to an object $p \from X \to M$ of $\sSet \slice M$ simply by
its domain and call $p$ its \emph{structure map}.
For a map $f \from M \to N$, we write
\begin{tikzeq*}
\matrix[diagram,column sep={between origins,6em}]
{
  |(Nl)| \sSet \slice N & |(Ml)| \sSet \slice M &
  |(Mr)| \sSet \slice M & |(Nr)| \sSet \slice N \\
};

\draw[->] (Nl) to node[above] {$f^*$}   (Ml);
\draw[->] (Mr) to node[above] {$\Pi_f$} (Nr);
\end{tikzeq*}
for the induced pullback functor and its right adjoint, to which
we refer as the \emph{dependent product} along $f$.

For the following statement, recall the discussion of van Kampen colimits in \cref{sec:dec-inc}.

\begin{lemma} \label{sSet-van-Kampen}
  The category of simplicial sets is extensive, adhesive and $\omega$-exhaustive.
\end{lemma}

\begin{proof}
  The category of sets satisfies these properties by \cref{Set-van-Kampen}.
  Thus so does the category of simplicial sets since monomorphisms, pullbacks and colimits of presheaves are determined pointwise.
\end{proof}

Denote the inclusions $\{ 0 \} \ito \simp{1}$ and $\{ 1 \} \ito \simp{1}$ by $\iota_0$ and $\iota_1$, respectively.
The simplex $\simp{1}$ will serve as an interval object, but we will occasionally use other intervals (such as $J$ in the proof of \cref{shrinkable-collapse})
with endpoint inclusions also denoted by $\iota_0$ and $\iota_1$.
The \emph{cylinder} on a simplicial set $X$ is the product $X \times \simp{1}$.
Let $f$ and $g$ simplicial maps $X \to Y$.
A \emph{homotopy} from $f$ to $g$ is a simplicial map $H \from X \times \simp{1} \to Y$ \st{} $f = H (X \times \iota_0)$ and $g = H (X \times \iota_1)$
(we will usually abbreviate these to $H \iota_0$ and $H \iota_1$).
Simplicial maps $f, g \from X \to Y$ are \emph{homotopic}, written $f \htp g$, if they can be connected by a zig-zag of homotopies.
The \emph{constant} homotopy on $f \from X \to Y$ is the composite $f \pi$ where $\pi \from X \times \simp{1} \to X$  is the projection.
If $X$ and $Y$ are simplicial sets over $M$, then a homotopy $H$ as above is \emph{fiberwise (over $M$)}
if it becomes constant when composed with the structure map $Y \to M$.

For $m \ge 0$, the \emph{boundary inclusion} $\bdsimp{m} \ito \simp{m}$ corresponds to the sieve on $[m]$ of all maps that lift through $\delta_j \from [m-1] \to [m]$ for some $j$.
We write $I$ for the set of boundary inclusions.
We say that a map is a \emph{trivial fibration} if it has the \rlp{} \wrt{} $I$
and that a map is a  \emph{cofibration} if it has the \llp{} \wrt{} to trivial fibrations.
A simplicial set $X$ is \emph{cofibrant} if
the map $\emptyset \to X$ is a cofibration.

For $m > 0$ and $0 \le i \le m$, the \emph{horn inclusion} $\horn{m,i} \ito \simp{m}$ corresponds to the sieve on $[m]$ of all maps that lift through $\delta_j \from [m-1] \to [m]$ for some $j \neq i$.
We write $J$ for the set of horn inclusions.
We say that a map is a \emph{Kan fibration} if it has the \rlp{} \wrt{} $J$
and that a map is a  \emph{trivial cofibration} if it has the \llp{} \wrt{} Kan fibrations.
By definition, a simplicial set is a \emph{Kan complex} if
the map $X \to \simp{0}$ is a Kan fibration.

Given a set of maps~$K$, a \emph{$K$-cell complex of height $\alpha \leq \omega$} is an $\alpha$-composition of pushouts of coproducts of maps in $K$.

\begin{lemma}\label{triv-cof-cof}
Every trivial cofibration is a cofibration and every trivial fibration is a fibration.
\end{lemma}

\begin{proof}
  It suffices to verify that every horn inclusion is a cofibration.
  Indeed, a horn inclusion is an $I$-cell complex  of height $2$.
\end{proof}

In order to define the weak factorisation systems, we adopt a slight variation of the
well-known small object argument. For this, we need a few results about colimits of diagrams of (trivial) fibrations
that will be needed also later on. In our constructive setting, these statements are somewhat more delicate than usual to prove.
Recall that a fibration is a map together with choice of lifts (from the right) against horn inclusions.
A \emph{structure morphism} of fibrations from $X \fto Y$ to $X' \fto Y'$ is a commuting square
\begin{tikzeq*}
\matrix[diagram]
{
  |(X)| X & |(X')| X' \\
  |(Y)| Y & |(Y')| Y' \\
};

\draw[fib]  (X) to (Y);
\draw[fib]  (X') to (Y');
\draw[->] (X) to (X');
\draw[->] (Y) to (Y');
\end{tikzeq*}
such that for any lifting problem of a horn inclusion $\horn{m,i} \ito \simp{m}$ against $X \fto Y$, the following diagram of chosen lifts commutes:
\begin{tikzeq*}
\matrix[diagram]
{
  |(h)| \horn{m,i} & |(X)| X & |(X')| X'                 \\
  |(s)| \simp{m}   & |(Y)| Y & |(Y')| Y' \rlap{\text{.}} \\
};

\draw[ano] (h) to (s);
\draw[fib] (X) to (Y);
\draw[fib] (X') to (Y');
\draw[->] (h) to (X);
\draw[->] (s) to (Y);
\draw[->] (X) to (X');
\draw[->] (Y) to (Y');
\draw[->,dashed] (s) to (X);
\draw[->,dashed] (s) to (X');
\end{tikzeq*}
The notion of structure morphism of trivial fibrations is defined similarly.
For the benefit of the readers, we remark that the use of these notions is confined to this subsection and to the derivation of the fibration extension property in \cref{fibration-extension}.

\begin{lemma} \label{colim-of-structured-fibrations}
\leavevmode
\begin{enumerate}
\item \label{omega-colim-of-structured-fibrations}
Let $p$ be a sequential diagram of (trivial) fibrations $p_k \from X_k \fto Y_k$ such that the naturality squares of $p$ are structure morphisms.
Then $\colim X \to \colim Y$ is a (trivial) fibration.
\item
\label{van-kampen-colim-of-structured-fibrations}
Let $p$ be a diagram of (trivial) fibrations $p_k \from X_k \fto Y_k$ such that the naturality squares of $p$ are structure morphisms and are pullbacks.
Assume that the induced square from $X_k \to Y_k$ to $\colim X \to \colim Y$ is a pullback.
Then $\colim X \to \colim Y$ is a (trivial) fibration.
\end{enumerate}
\end{lemma}

\begin{proof}
\Cref{omega-colim-of-structured-fibrations} is a formal consequence of the fact that $\horn{m,i}$ and $\simp{m}$ are finite colimits of representables, hence finitely presented, \ie, mapping out of them preserves filtered colimits, and that pullbacks commute with filtered colimits; we provide the details of the case of fibrations for the convenience of the readers.
We wish to construct a section of the map
\[
\sSet(\simp{m}, \colim X) \to \sSet(\simp{m}, \colim Y) \pull_{\sSet(\horn{m,i}, \colim Y)} \sSet(\horn{m,i}, \colim X)
\text{.}\]
Since $\horn{m,i}$ and $\simp{m}$ are finitely presented, this is
\[
\colim_k \sSet(\simp{m}, X_k) \to (\colim_k \sSet(\simp{m}, Y_k)) \pull_{\colim_k \sSet(\horn{m,i}, Y_k)} (\colim_k \sSet(\horn{m,i}, X_k))
\text{.}\]
Since pullbacks commute with sequential colimits, this is
\[
\colim_k \sSet(\simp{m}, X_k) \to \colim_k \sSet(\simp{m}, Y_k) \pull_{\sSet(\horn{m,i}, Y_k)} \sSet(\horn{m,i}, X_k)
\text{.}\]
By functoriality of colimits, it thus suffices to have sections of
\[
\sSet(\simp{m}, X_k) \to \sSet(\simp{m}, Y_k) \pull_{\sSet(\horn{m,i}, Y_k)} \sSet(\horn{m,i}, X_k) \text{,}
\]
naturally in $k$.
We have such a section for each $k$ since $X_k \to Y_k$ is a fibration and they are natural since the naturality squares of $p$ are structure morphisms of fibrations.

\Cref{van-kampen-colim-of-structured-fibrations} is a formal consequence of the fact that mapping out of $\simp{m}$ preserves arbitrary colimits.
Again, we do the case of fibrations in detail.
We wish to construct a section of the map
\[
\sSet(\simp{m}, \colim X) \to \sSet(\simp{m}, \colim Y) \pull_{\sSet(\horn{m,i}, \colim Y)} \sSet(\horn{m,i}, \colim X)
\text{.}\]
Since mapping out of the representable $\simp{m}$ preserves colimits, this is
\[
\colim_k \sSet(\simp{m}, X_k) \to (\colim_k \sSet(\simp{m}, Y_k)) \pull_{\sSet(\horn{m,i}, \colim Y)} \sSet(\horn{m,i}, \colim X)
\text{.}\]
Since pullbacks preserve colimits, this is
\[
\colim_k \sSet(\simp{m}, X_k) \to \colim_k \sSet(\simp{m}, Y_k) \pull_{\sSet(\horn{m,i}, \colim Y)} \sSet(\horn{m,i}, \colim X)
\text{.}\]
By assumption, $\colim X$ pulls back along $Y_k \to \colim Y$ to $X_k$, so this is
\[
\colim_k \sSet(\simp{m}, X_k) \to \colim_k \sSet(\simp{m}, Y_k) \pull_{\sSet(\horn{m,i}, Y_k)} \sSet(\horn{m,i}, X_k)
\text{.}\]
From here, we conclude as in \cref{omega-colim-of-structured-fibrations}.
\end{proof}

\begin{lemma} \label{aligning-structured-fibrations}
In any commuting square
\begin{tikzeq*}
\matrix[diagram]
{
  |(X)| X & |(X')| X'                 \\
  |(Y)| Y & |(Y')| Y' \rlap{\text{,}} \\
};

\draw[->] (X) to (Y);
\draw[->] (X') to (Y');
\draw[->] (X) to (X');
\draw[->] (Y) to (Y');
\end{tikzeq*}
with $X \to Y$ and $X' \to Y'$ (trivial) fibrations and $X \to X'$ and $Y \to Y'$ levelwise decidable inclusions, there is a replacement for the choice of lifts of $X' \to Y'$ such that the square forms a structure morphism.
\end{lemma}

\begin{proof}
We only discuss the case of fibrations.
Since $\horn{m,i}$ is a finitely presented, the inclusion $\sSet(\horn{m,i}, X) \to \sSet(\horn{m,i}, X')$ is decidable by \cref{decidable-limit} of \cref{decidable-closure}, and so is $\sSet(\simp{m}, Y) \to \sSet(\simp{m}, Y')$.
If we now consider the set of lifting problems of horn inclusions against $X' \to Y'$, its subset consisting of those lifting problems that factor via the given square is a
decidable subset. Thus, given a lifting problem of a horn inclusion against $X' \to Y'$, we may perform a case distinction and pick the lift given by either the fibration $X \to Y$ or the fibration $X' \to Y'$.
\end{proof}

\begin{proposition} \label{omega-colim-of-fibrations}
Let $p$ be an sequential diagram of (trivial) fibrations $p_k \from X_k \fto Y_k$ where $X_k \to X_{k+1}$ and $Y_k \to Y_{k+1}$ are levelwise decidable inclusions for $k \geq 0$.
Then $\colim X \to \colim Y$ is a (trivial) fibration.
\end{proposition}

\begin{proof}
Recursively in $k$, we apply \cref{aligning-structured-fibrations} to the naturality square of $p$ at $k \to k + 1$ to put a new choice of lifts on $X_{k+1} \to Y_{k+1}$ such that this square becomes a structure morphism.
Then the conclusion follows from \cref{omega-colim-of-structured-fibrations} of \cref{colim-of-structured-fibrations}.
\end{proof}

\begin{lemma} \label{small-object-argument-lemma}
Let $p$ be an sequential diagram of objects $p_k \from X_k \to Y$ over $Y$ where $X_k \to X_{k+1}$ is levelwise decidable inclusion for $k \geq 0$.
Assume that for each horn inclusion (boundary inclusion) $A \to B$, we have lifts for any lifting problem against $X_{k+1} \to Y$ that factors through $X_k \to X_{k+1}$.
Then $\colim X \to Y$ is a (trivial) fibration.
\end{lemma}

\begin{proof}
We have an sequential diagram of triangles
\begin{tikzeq}{small-object-argument:triangle}
\matrix[diagram]
{
  |(Xk)| X_k &        & |(XSk)| X_{k+1} \rlap{\text{,}} \\
             &|(Y)| Y &                                 \\
};

\draw[->] (Xk) to (Y);
\draw[->] (XSk) to (Y);
\draw[->] (Xk) to (XSk);
\end{tikzeq}
each of which has \emph{relative right lifts} against $A \to B$, meaning that for any lifting problem of $A \to B$ against the left map, the induced lifting problem against the right map has a lift.
Since $A$ is finitely presented, the maps $\sSet(A, X_k) \to \sSet(A, X_{k+1})$ are decidable inclusion.
Thus, as in \cref{aligning-structured-fibrations}, we can choose relative right lifts at stage $k+1$ that cohere with those at stage $k$ whenever the relative lifting problem factors.
Recursively in $k$, as in \cref{omega-colim-of-fibrations}, we obtain a choice of relative right lifts for~\eqref{small-object-argument:triangle} that is natural in $k$.
Taking the sequential colimit as in \cref{omega-colim-of-structured-fibrations} of \cref{colim-of-structured-fibrations}, we obtain relative right lifts for
\begin{tikzeq*}
\matrix[diagram]
{
  |(Xk)| \colim X &        & |(XSk)| \colim X \rlap{\text{,}} \\
                  &|(Y)| Y &                                  \\
};

\draw[->] (Xk) to (Y);
\draw[->] (XSk) to (Y);
\draw[->] (Xk) to (XSk);
\end{tikzeq*}
{\ie}, lifts of $A \to B$ against $\colim X \to Y$.
\end{proof}

For the next statement, recall the sets $I$ of boundary inclusions and $J$ of horn inclusions.

\begin{proposition}[\Wfs{}s] \label{thm:wfs-via-soa}
\leavevmode
\begin{enumerate}
\item \label{thm:wfs-via-soa:triv-fib}
Every map $X \to Y$ factors functorially as a cofibration followed by a trivial fibration.
Every cofibration is a codomain retract of a relative $I$-cell complex of height $\omega$.
\item \label{thm:wfs-via-soa:fib}
Every map $X \to Y$ factors functorially as a trivial cofibration followed by a fibration.
Every trivial cofibration is a codomain retract of a relative $J$-cell complex of height $\omega$.
\end{enumerate}
\end{proposition}

\begin{proof}
We prove this by a constructive version of the  small object argument.
We take $X_0 = X$ and $X_k \to X_{k+1}$ as the pushout of a coproduct of horn inclusions (boundary inclusions) indexed by lifting problems against $X_k \to Y$.
Importantly, since horn inclusions (boundary inclusions) are levelwise decidable inclusions and these are closed under coproducts and pushout, so is $X_k \to X_{k+1}$.
That means we can use \cref{small-object-argument-lemma} to conclude that $X_\omega \to Y$ is a (trivial) fibration.
The relative cell complex $X \to X_\omega$ is a trivial cofibration (cofibration) by standard saturation properties of the left lifting closure.
The claim about trivial cofibrations (cofibrations) follows from the retract argument.
\end{proof}

For any simplicial set $M$, the \wfs{}s of \cref{thm:wfs-via-soa} induce \wfs{}s in
the slice $\sSet \slice M$.
A map over $M$ will be called a (trivial) (co)fibration if
it is a (trivial) (co)fibration in $\sSet$.

\begin{remark} \label{rem:acyclic-vs-trivial}
A map is \emph{acyclic} if it is a weak equivalence in a context dependent sense of \emph{fiberwise homotopy equivalence} of \cref{fibration-category} or \emph{weak homotopy equivalence} of \cref{whes,sec:model-structure-second-proof}.
In particular, an \emph{acyclic (co)fibration} is a (co)fibration that is also a weak equivalence.
We will be careful about the distinction between the notions of
acyclic (co)fibrations and trivial (co)fibrations until
they are proved equivalent. The notation introduced in~\cref{tab:arrows}
is intended to help readers keep track of the different notions.
\end{remark}

We conclude this subsection with two short but critical lemmas.

\begin{lemma}\label{simplex-cofibrant}
  Both $\simp{m}$ and $\bdsimp{m}$ are cofibrant for all $m$.
  In fact,
  $\bdsimp{m}$ is a cell complex \wrt{} $\set{\bdsimp{k} \to \simp{k}}{k < m}$.
\end{lemma}

\begin{proof}
  For $\bdsimp{m}$ consider
  a filtration $X^{(-1)} \ito X^{(0)} \ito \ldots X^{(m - 1)}$ where
  \begin{equation*}
    X^{(k)}_i = \set{\phi \from [i] \to [m]}{\phi \text{ factors through } [k]} \text{.}
  \end{equation*}
  Then we have $X^{(-1)} = \emptyset$, $X^{(m - 1)} = \bdsimp{m}$ and
  for each $k \in \{ 0, \ldots, m - 1 \}$ there is a pushout square
  \begin{tikzeq*}
  \matrix[diagram,column sep={10em,between origins}]
  {
    |(b)| \Simp_\sharp([k], [m]) \times \bdsimp{k} & |(k1)| X^{(k - 1)} \\
    |(s)| \Simp_\sharp([k], [m]) \times   \simp{k} & |(k)|  X^{(k)}     \\
  };

  \draw[cof] (b)  to (s);
  \draw[->]  (k1) to (k);
  \draw[->]  (b)  to (k1);
  \draw[->]  (s)  to (k);
  \end{tikzeq*}
  so $\bdsimp{m}$ is indeed
  a cell complex \wrt{} $\set{\bdsimp{k} \to \simp{k}}{k < m}$.
  It follows that $\simp{m}$ is cofibrant as well.
\end{proof}

As a consequence we derive a cancellation property for trivial fibrations,
which is crucially used in later arguments to
extend certain results about cofibrant simplicial sets to all simplicial sets.

\begin{lemma}\label{trivial-fibration-cancellation}
  If $f \from X \to Y$ and $g \from Y \to Z$ are simplicial maps \st{}
  $f$ and $g f$ are trivial fibrations, then so is $g$.
\end{lemma}

\begin{proof}
  Consider a lifting problem
  \begin{tikzeq*}
  \matrix[diagram,column sep={between origins,6em}]
  {
    |(b)| \bdsimp{m} & |(Y)| Y                   \\
    |(s)|   \simp{m} & |(Z)| Z \rlap{\textrm{.}} \\
  };

  \draw[->]  (b) to node[above] {$u$} (Y);
  \draw[cof] (b) to node[left]  {$i$} (s);

  \draw[->] (s) to node[below] {$v$} (Z);
  \draw[->] (Y) to node[right] {$g$} (Z);
  \end{tikzeq*}
  Since $\bdsimp{m}$ is cofibrant by \cref{simplex-cofibrant},
  $u$ lifts along $f$:
  \begin{tikzeq*}
  \matrix[diagram,column sep={between origins,6em}]
  {
                     & |(X)| X                 \\
    |(b)| \bdsimp{m} & |(Y)| Y \rlap{\text{.}} \\
  };

  \draw[->,dashed] (b) to node[above left] {$\tilde u$} (X);

  \draw[->]   (b) to node[below] {$u$} (Y);
  \draw[tfib] (X) to node[right] {$f$} (Y);
  \end{tikzeq*}
  This leads to a lifting problem
  \begin{tikzeq*}
  \matrix[diagram,column sep={between origins,6em}]
  {
    |(b)| \bdsimp{m} & |(X)| X                   \\
    |(s)|   \simp{m} & |(Z)| Z \rlap{\textrm{.}} \\
  };

  \draw[->]  (b) to node[above] {$\tilde u$} (X);
  \draw[cof] (b) to node[left]  {$i$}        (s);

  \draw[->]   (s) to node[below] {$v$}   (Z);
  \draw[tfib] (X) to node[right] {$g f$} (Z);
  \end{tikzeq*}
  Call its solution $\tilde w \from \simp{m} \to X$, then
  $w = f \tilde w$ is a solution to the original problem.
  Indeed,
  \begin{equation*}
    g w = g f \tilde w = v \text{ and }
    w i = f \tilde w i = f \tilde u = u \text{.} \qedhere
  \end{equation*}
\end{proof}

\subsection{Pushout product properties}
\label{sec:two-wfss}

We now establish that the \wfs{s} of cofibrations and trivial fibrations and of trivial cofibrations and fibrations satisfy various forms of the pushout product property.
In order to do this, we need to introduce some notation.
Given $M \in \sSet$ and two objects $X \to M$ and $Y \to M$ in the slice category
$\sSet \slice M$, we write $\exp_M(X, Y) \to M$ for
their exponential in $\sSet \slice M$.
The category $\sSet \slice M$ is also a $\sSet$-enriched category
in a canonical way and we write $\hom_M(X, Y)$ for the simplicial hom-object.
When considering slices over the terminal object $\simp{0}$ of $\sSet$,
we often drop the subscript and write $Y^X$ for the common value of
$\exp_{\simp{0}}(X,Y)$ and $\hom_{\simp{0}}(X,Y)$.
As a $\sSet$-enriched category, $\sSet \slice M$ admits tensors and cotensors.
For $A \in \sSet$ and~$X \in \sSet \slice M$,
the tensor is written $A \tensor X \in \sSet \slice M$ and
is given by the cartesian product $A \times X$
(the structure map is the composite of the product projection and the structure map of $X$).
The cotensor is written $A \cotensor X \in \sSet \slice M$ and is given by
the pullback of $X^A \to M^A$ along the map $M \to M^A$ (the adjoint transpose of the product projection).
For maps $i \from A \to B$ and $p \from X \to Y$, we write
\begin{tikzeq*}
\matrix[diagram,column sep={between origins,8em}]
{
  |(pp)| (A \times Y) \coprod_{A \times X} (B \times X) &[2em] |(BY)| B \times Y          &
  |(XB)| X^B                                            &      |(pe)| X^A \pull_{Y^A} Y^B \\
};

\draw[->] (pp) to node[above] {$i \hat{\times} p$} (BY);
\draw[->] (XB) to node[above] {$\hat{\exp}(f, g)$} (pe);
\end{tikzeq*}
for their pushout product and their pullback exponential, respectively.
Analogous notation is used when we consider the tensor and cotensor functors
instead of the product and exponential functors.

The following proposition is well-known and easy to prove in classical logic.
However, it is less trivial constructively due to the delicate nature of cofibrations and thus we provide a complete proof.

\begin{proposition} \label{first-pushout-product}
  Let $M \in \sSet$.
  \begin{enumerate}
  \item\label{cofibration-pushout-product}
    If $i \from A \cto B$ and $j \from C \cto D$ are
    cofibrations in $\sSet \slice M$, then
    so is their pushout product $i \hat{\times}_M j$.
  \item\label{cof-tcof-pullback-hom}
    If $i \from A \cto B$ is a cofibration and
    $p \from X \tfto Y$ is a trivial fibration in $\sSet \slice M$, then
    their pullback exponential $\hat{\exp}_M(i, p)$  is also a trivial fibration.
    In particular, if $A$ is cofibrant and
    $p \from X \tfto Y$ is a trivial fibration, then $\exp_M(A, p)$ is a trivial fibration.
  \end{enumerate}
\end{proposition}

\begin{proof}
  For~\cref{cofibration-pushout-product},
  it suffices to verify that
  for any pair of simplices $\simp{m} \to M$ and $\simp{n} \to M$
  the pushout product of $\bdsimp{m} \to \simp{m}$ and $\bdsimp{n} \to \simp{n}$
  is a cofibration.
  Indeed, there is a filtration
  \begin{equation*}
    \bdsimp{m} \pull_M \simp{n} \union \simp{m} \pull_M \bdsimp{n}
    = X^{(-1)} \ito X^{(0)} \ito \ldots \ito X^{(m + n)} = \simp{m} \pull_M \simp{n}
  \end{equation*}
  where $X^{(k)}$ is formed iteratively by taking pushouts
  \begin{tikzeq*}
  \matrix[diagram,column sep={8em,between origins}]
  {
    |(b)| S_k \times \bdsimp{k} & |(k1)| X^{(k - 1)} \\
    |(s)| S_k \times   \simp{k} & |(k)|  X^{(k)}     \\
  };

  \draw[cof] (b)  to (s);
  \draw[->]  (k1) to (k);
  \draw[->]  (b)  to (k1);
  \draw[->]  (s)  to (k);
  \pbdr{k}{b};
  \end{tikzeq*}
  with
  \begin{align*}
    S_k = \left\{(\phi, \psi) \in \simp{m} \pull_M \simp{n} \right. \mid
    & \text{ at least one of } \phi \from [k] \to [m] \text{ or } \psi \from [k] \to [n] \text{ is surjective} \\
    & \left. \text{ and } (\phi, \psi) \from [k] \to [m] \times [n] \text{ is injective} \right\} \text{.}
  \end{align*}
  \Cref{cof-tcof-pullback-hom} follows from~\cref{cofibration-pushout-product} by adjointness.
\end{proof}

\begin{corollary} \label{pullback-of-cofibrations}
  If $A \to B$ is a simplicial map with $A$ cofibrant,
  then the induced pullback functor $\sSet \slice B \to \sSet \slice A$
  preserves cofibrations.
\end{corollary}

\begin{proof}
  If $X \cto Y$ is a cofibration over $B$, then
  the induced map $A \pull_B X \to A \pull_B Y$ coincides with
  the pushout product (in $\sSet \slice B$) of
  $\emptyset \cto A$ and $X \cto Y$.
  The conclusion follows from \cref{cofibration-pushout-product}
  of \cref{first-pushout-product}.
\end{proof}

\begin{proposition} \label{second-pushout-product}
  \leavevmode
  \begin{enumerate}
  \item\label{ano-cof-pushout-product}
    If $i \from A \ato B$ is a trivial cofibration and
    $j \from C \cto D$ is a cofibration
    then their pushout product $i \hat{\times} j$ is a trivial cofibration.
  \item\label{cof-fib-pullback-hom}
    If $i \from A \cto B$ is a cofibration and
    $p \from X \fto Y$ is a Kan fibration, then
    their pullback exponential $\hat{\hom}(i, p)$ is a Kan fibration.
    In particular, if $i \from C \cto D$ is a cofibration and $K$ a Kan complex,
    then $\hom(i, K)$ is a Kan fibration.
    In particular, if $A$ is cofibrant and $K$ is a Kan complex, then
    $K^A$ is a Kan complex.
  \item\label{ano-fib-pullback-hom}
    If $i \from A \ato B$ is a trivial cofibration and
    $p \from X \fto Y$ is a Kan fibration, then
    their pullback exponential $\hat{\hom}(i, p)$ is a trivial fibration.
    In particular, if $i \from A \ato B$ is a trivial cofibration and
    $K$ a Kan complex, then
    $\hom(i, K)$ is a trivial fibration.
  \end{enumerate}
\end{proposition}

\begin{proof}
  \Cref{ano-cof-pushout-product} is proved in \cite{gz}*{Proposition~IV.2.2}
  with a constructive argument.
  \Cref{cof-fib-pullback-hom,ano-fib-pullback-hom} follow by adjointness.
\end{proof}

\begin{corollary}\label{cof-tensor-cotensor}
  \leavevmode
  \begin{enumerate}
  \item\label{cof-ano-pushout-tensor}
    If $j \from A \cto B$ is a cofibration and
    $i \from C \ato D$ is a trivial cofibration over $M$,
    then their pushout tensor $j \hat{\tensor} i$ is a trivial cofibration.
  \item\label{ano-cof-pushout-tensor}
    If $i \from A \ato B$ is a trivial cofibration and
    $j \from C \cto D$ is a cofibration over $M$, then
    their pushout tensor $i \hat{\tensor} j$ is a trivial cofibration.
  \item\label{cof-fib-pullback-cotensor}
    If $i \from A \cto B$ is a cofibration and
    $p \from X \fto Y$ is a Kan fibration over $M$, then
    their pullback cotensor $i \hat{\cotensor} p$ is a Kan fibration.
  \item\label{ano-fib-pullback-cotensor}
    If $j \from A \ato B$ is a trivial cofibration and
    $p \from X \fto Y$ is a Kan fibration over $M$,
    then their pullback cotensor $j \hat{\cotensor} p$ is a trivial fibration.
  \end{enumerate}
\end{corollary}

\begin{proof}
  \Cref{cof-ano-pushout-tensor,ano-cof-pushout-tensor} follow
  from \cref{ano-cof-pushout-product} of \cref{second-pushout-product} since
  the underlying map of the pushout tensor is the ordinary pushout product.
  \Cref{cof-fib-pullback-cotensor,ano-fib-pullback-cotensor} follow from these
  by adjointness.
\end{proof}

  \subsection{Cofibrations as Reedy decidable inclusions}
\label{sec:cofibrations}

The aim of this subsection is to exhibit
the \wfs{} of cofibrations and trivial fibrations on simplicial sets of~\cref{thm:wfs-via-soa} as
the Reedy \wfs{} induced by
the \wfs{} of decidable inclusions and split surjections on sets, and
use this fact to give streamlined proofs of
several closure properties of cofibrations and cofibrant objects. For background on Reedy weak
factorisation systems, and in particular the definition of latching and matching objects, we refer to~\cite{hovey}*{Chapter~5} and~\cite{riehl-verity-reedy}.

A simplicial map $A \to B$ is a \emph{Reedy decidable inclusion} if
for all $m$ the relative latching map $A_m \push_{L_m A} L_m B \to B_m$ is
a decidable inclusion.
Dually, a simplicial map $X \to Y$ is a \emph{Reedy split surjection} if
all its relative matching maps $X_m \to Y_m \times_{M_m Y} M_m X$ are
split surjections. We define Reedy decidable inclusions of cosimplicial sets similarly.

\begin{lemma}\label{cofibration-wfs}
  Reedy decidable inclusions coincide with cofibrations and
  Reedy split surjections coincide with trivial fibrations.
\end{lemma}

\begin{proof}
  This follows from \cref{decidable-wfs} and
  the fact the associated Reedy \wfs{} of a cofibrantly generated \wfs{}
  is also cofibrantly generated with generators given by pushout products of the original generators
  and boundary inclusions of representable functors.
  See \cite{riehl-verity-reedy}*{Lemmas~7.3, 7.4 and~Corollary~6.7} for details.
\end{proof}

\begin{lemma} \label{delta-elegancy}
Let $X$ be a simplicial set.
\begin{enumerate}
\item \label{delta-elegancy-deg-on-obj}
Every degeneracy operator $[m] \sto [n]$ acts by a monomorphism $X_n \to X_m$.
\item \label{delta-elegancy-latching-obj}
The latching map $L_m X \to X_m$ is a monomorphism.
Moreover, as a subset of $X_m$, $L_m X$ is the union of the subsets $X_n$ indexed over non-identity degeneracy operators $[m] \sto [n]$.
\end{enumerate}
Let $A \to B$ be a monomorphism  of simplicial sets.
\begin{enumerate}[resume]
\item \label{delta-elegancy-deg-on-mono}
For every degeneracy operator $[m] \sto [n]$, the set $A_n$ is the intersection of the subobjects $A_m$ and $B_n$ of $B_m$.
\item \label{delta-elegancy-latching-mono}
$L_m A$ is the intersection of the subobjects $A_m$ and $L_m B$ of $B_m$.
\end{enumerate}
\end{lemma}

\begin{proof}
\Cref{delta-elegancy-deg-on-obj} holds since $[m] \sto [n]$ has a section.

For \cref{delta-elegancy-latching-obj}, we use the fact that every span of degeneracy operators in $\Simp$
has an absolute pushout\footnote{This means that $\Simp$ is an \emph{elegant Reedy category} in the sense of \cite{bergner-rezk-elegant}.}
(\ie, one that is preserved by all functors) which is proven in \cite{jt}*{Theorem~1.2.1}.
It follows that each latching diagram $(\bd([m] \slice \Simp_\flat))^\op \to \Set$ induced by $X$
sends pushouts in $\bd([m] \slice \Simp_\flat)$ to pullbacks, and hence the family of subsets $X_n \to X_m$ (as $[m] \sto [n]$ varies over non-identity degeneracy operators with source $[m]$) is closed under intersection.
Thus, its union coincides with its colimit $L_m X$ in the slice over $X_m$.
In particular, $L_m X \to X_m$ a monomorphism.

For \cref{delta-elegancy-deg-on-mono}, we need to show that the square
\begin{tikzeq*}
\matrix[diagram]
{
  |(An)| A_n & |(Bn)| B_n \\
  |(Am)| A_m & |(Bm)| B_m \\
};

\draw[->] (An) to (Bn);
\draw[->] (Am) to (Bm);
\draw[->] (An) to (Am);
\draw[->] (Bn) to (Bm);
\end{tikzeq*}
is a pullback.
Indeed, let $X$ be a set and $u \from X \to A_m$ and $v \from X \to B_n$ maps \st{} $v \dgn = i_m u$.
Let $\face \from [n] \to [m]$ be a section of $\dgn$.
Then, by a direct calculation, $w = u \face \from X \to A_n$ is the unique map
\st{} $w \dgn = u$ and $i_n w = v$.

\Cref{delta-elegancy-latching-mono}, follows from \cref{delta-elegancy-latching-obj,delta-elegancy-deg-on-mono} using that intersecting with $A_m$ preserves unions.
\end{proof}

With the knowledge of \cref{delta-elegancy},  \cref{cofibration-wfs} corresponds to \cite{Henry-wms}*{Proposition~5.1.4}, there proved without explicit reference to Reedy \wfs{}s.

Given a simplicial set $A$ and a cosimplicial set $W$, the set $A \times_\Simp W$ is defined as the coend $\coend^{k} A_k \times W_k$.
This set is also known as the \emph{weighted colimit} of $A$ with weight $W$.
While the following two statements do not depend on the theory of weighted colimits and their pushout constructions,
readers familiar with it may find it useful to think of them in these terms (\cf~\cite{riehl-verity-reedy}).

\begin{lemma}\label{Reedy-decidable-pushout-colimit}
  The bifunctor $A, W \mapsto A \times_\Simp W$ satisfies the pushout product property \wrt{} Reedy decidable inclusions,
  \ie, if $A \to B$ and $V \to W$ are Reedy decidable inclusions of simplicial and cosimplicial sets, respectively,
  then $A \times_\Simp W \push_{A \times_\Simp V} B \times_\Simp V \to B \times_\Simp W$ is a decidable inclusion.
\end{lemma}

\begin{proof}
  It suffices to verify this property on generators.
  (\cite{riehl-verity-reedy}*{Corollary~6.7} implies that
  cosimplicial Reedy decidable inclusions are generated by
  $\set{L_m \Simp([m], \uvar) \to \Simp([m], \uvar)}{m \in \nat}$.)
  Applying the construction to $\bdsimp{n} \to \simp{n}$ and $L_m \Simp([m], \uvar) \to \Simp([m], \uvar)$ yields
  the inclusion of $\set{\phi \in \Simp([m], [n])}{\phi \ne \id}$ into $\Simp([m], [n])$, which is decidable.
\end{proof}

We will now prove a useful characterisation of cofibrations.

\begin{proposition} \label{cof-characterisation}
The following are equivalent for a simplicial map $A \to B$:
\begin{conditions}
\item
  $A \to B$ is a cofibration,
\item
  $A_0 \to B_0$ is a decidable inclusion and, for each generating degeneracy operator $[n + 1] \sto [n]$,
  the induced map $A_{n+1} \push_{A_n} B_n \to B_{n+1}$ is a decidable inclusion.
\item
  $A \to B$ is a levelwise decidable inclusion and, for each degeneracy operator $[m] \sto [n]$,
  the induced map $A_m \push_{A_n} B_n \to B_m$ is a decidable inclusion.
\end{conditions}
\end{proposition}

\begin{proof}
We go from~(i) to~(ii).
The Reedy condition in dimension $0$ means that $A_0 \to B_0$ is a decidable inclusion.
Next, take a degeneracy operator $\sigma \from [m] \to [n]$.
We will check that
the induced map of cosimplicial sets $\Simp([n], \uvar) \to \Simp([m], \uvar)$
is a Reedy decidable inclusion.
Indeed, for each $k$, the map $\Simp([n], [k]) \to \Simp([m], [k])$ is
(isomorphic to) the inclusion of the subset of
those simplicial operators $[m] \to [k]$ that factor through $\sigma$, which
is decidable.
For each $a$, note that $L_k \Simp([a], \uvar) \to \Simp([a], [k])$ is the (decidable) subset of non-surjective operators.
Thus, $L_k \Simp([n], \uvar)$ is the intersection over $\Simp([m], [k])$ of
$L_k \Simp([m], \uvar)$ and $\Simp([n], [k])$.
It follows by \cref{decidable-union} of \cref{decidable-closure} that
$\Simp([n], [k]) \push_{L_k \Simp([n], \uvar)} L_k \Simp([m], \uvar) \to \Simp([m], [k])$
is a decidable inclusion.
The map $A_m \push_{A_n} B_n \to B_m$ coincides with $A \times_\Simp \Simp([m], \uvar) \push_{A \times_\Simp \Simp([n], \uvar)} B \times_\Simp \Simp([n], \uvar) \to B \times_\Simp \Simp([m], \uvar)$ induced by $A \to B$ and $\Simp([n], \uvar) \to \Simp([m], \uvar)$, so is a decidable inclusion by \cref{Reedy-decidable-pushout-colimit}.

We go from~(ii) to~(iii).
Each degeneracy operator $[m] \sto [n]$ is a finite composition of generators.
Thus, $A_m \push_{A_n} B_n \to B_m$ is a finite composition of pushouts of maps $A_{k+1} \push_{A_k} B_k \to B_{k+1}$ for generators $\dgn_i \from [k+1] \to [k]$, hence is a decidable inclusion.
For each $m$, the map $A_m \to B_m$ factors as $A_m \to A_m \push_{A_0} B_0 \to B_m$ (where the first map is a pushout of $A_0 \to B_0$) using the degeneracy operator $[m] \sto [0]$, hence is a decidable inclusion.

We go from~(iii) to~(i).
Given $m$ and working over $B_m$, we have to show that $A_m \push_{L_m A} L_m B$ is a decidable subobject.
Recall from \cref{delta-elegancy-latching-obj} of \cref{delta-elegancy} that $L_m B$ is the finite union of $B_n$ indexed over non-identity degeneracy operators $[m] \sto [n]$.
By \cref{delta-elegancy-deg-on-mono} of \cref{delta-elegancy}, $A_m \push_{A_n} B_n$ is the union of the subobjects $A_m$ and $B_n$.
Similarly, by \cref{delta-elegancy-latching-mono} of \cref{delta-elegancy}, $A_m \push_{L_m A} L_m B$ is the union of the subobjects $A_m$ and $L_m B$.
Distributing unions, it follows that $A_m \push_{L_m A} L_m B$ is the union of the decidable subobject $A_m$ with the finite union of decidable subobjects $A_m \push_{A_n} B_n$, hence is a decidable subobject by \cref{decidable-union} of \cref{decidable-properties}.
\end{proof}

\begin{corollary} \label{cofibrant-degeneracy}
A simplicial set is cofibrant \iff{} all (generating) degeneracy operators act on it by decidable inclusions.
\end{corollary}

\begin{proof}
This is \cref{cof-characterisation} for a map with initial domain.
\end{proof}

\begin{corollary}\label{cofibration-levelwise-decidable}
\leavevmode
\begin{enumerate}
\item \label{cof-to-dec} Every cofibration is a levelwise decidable inclusion.
\item \label{dec-to-cof} A levelwise decidable inclusion with cofibrant codomain is
  a cofibration.
  \end{enumerate}
  In particular, a map between cofibrant objects is a cofibration if and only if it is
  a levelwise decidable inclusion.
\end{corollary}

\begin{proof}
  \Cref{cof-to-dec} was shown in \cref{cof-characterisation}.
  For \cref{dec-to-cof}, let $A \to B$ be a levelwise decidable inclusion and
  let $B$ be cofibrant.
  Then, for any degeneracy operator $[m] \sto [n]$,
  the map $B_n \to B_m$ is a decidable inclusion by \cref{cofibrant-degeneracy}.
  Thus \cref{decidable-union} of \cref{decidable-closure} implies that
  $A_m \push_{A_n} B_n$ is a decidable subobject of $B_m$ as
  the union of decidable subobjects $B_n$ and $A_m$.
  The conclusion follows from \cref{cofibrant-degeneracy}.
\end{proof}

Next, we obtain some closure properties of cofibrations and cofibrant
objects.

\begin{lemma}\label{cofibrant-finite-limits}
  Cofibrant objects are closed under finite limits.
\end{lemma}

\begin{proof}
  This follows from \cref{decidable-limit} of \cref{decidable-closure}
  and \cref{cofibrant-degeneracy}.
\end{proof}

\begin{lemma}\label{cofibration-cancellation}
  If $A \ito B$ is a monomorphism and $B$ is cofibrant, then so is $A$.
\end{lemma}

\begin{proof}
  Let $[m] \sto [n]$ be a degneracy operator.
  The map $B_n \to B_m$ is a decidable inclusion by \cref{cofibrant-degeneracy}.
  The map $A_n \to A_m$ is a pullback of $B_n \to B_m$ by \cref{delta-elegancy-deg-on-mono} of \cref{delta-elegancy}, hence is a decidable inclusion by \cref{decidable-pullback} of \cref{decidable-properties}.
  Using \cref{cofibrant-degeneracy} again, we conclude that $A$ is cofibrant.
\end{proof}

The above statement is a special case of the closure of cofibrations under pullback along monomorphisms.
However, we will not need that more general fact.

\begin{remark}\label{bisimplicial-cofibrant-degeneracy}
The notion of cofibration and the associated statements in this subsection up to \cref{cofibration-cancellation} generalise, with the same arguments, to presheaves over an elegant Reedy category (of countable height) with decidable identities (as defined below), finite slices of face operators and finite coslices of degeneracy operators.
In \cref{sec:diagonals-of-ssets}, we will need the case of bisimplicial sets, where a cofibration is defined as a Reedy decidable inclusion of presheaves over the elegant Reedy category $\Simp \times \Simp$.
There, we will use bisimplicial versions of the current simplicial statements.
In particular, a bisimplicial set is cofibrant \iff{} all degeneracy operators act on it by decidable inclusions.
(For this, it suffices to separately examine the two cases of a degeneracy operator that is an identity in one of the two direction.)
\end{remark}

\begin{corollary}\label{horn-cofibrant}
  Every horn $\horn{m,i}$ is cofibrant.
\end{corollary}

\begin{proof}
  The horn inclusion $\horn{m,i} \to \simp{m}$ is a cofibration by \cref{triv-cof-cof}
  and thus the horn $\horn{m,i}$ is cofibrant by \cref{cofibration-cancellation}.
\end{proof}

\medskip

As a consequence of the characterisation of cofibrant objects of \cref{cofibrant-degeneracy} we can distinguish a class of categories with cofibrant nerves,
which will be useful in \cref{sec:first-proof}. Let us say that
a category $J$ has \emph{decidable identities} if the function $\ob J \to \mor J$ that sends each object to its identity morphism is a decidable inclusion.
A functor $I \to J$ is a \emph{decidable inclusion} if both functions $\ob I \to \ob J$ and $\mor I \to \mor J$ are decidable inclusions.
For example, the category $\Simp$ and its wide subcategories $\Simp_\flat$ and $\Simp_\sharp$ have decidable identities.

\begin{lemma} \label{nerve-cof}
\leavevmode
\begin{enumerate}
\item  \label{nerve-cofibrant}  If $J$ is a category with decidable identities, then $\nerve J$ is cofibrant.
\item \label{nerve-cofibration}
  If $I \to J$ is a decidable inclusion, then $\nerve I \to \nerve J$ is
  a levelwise decidable inclusion.
  In particular, if $I$ and $J$ have moreover decidable identities, then
  $\nerve I \to \nerve J$ is a cofibration.
  \end{enumerate}
\end{lemma}

\begin{proof} For \cref{nerve-cofibrant},
  for each $m$ the set $(\nerve J)_m$ is
  the iterated $m$-fold pullback $\mor J \pull_{\ob J} \ldots \pull_{\ob J} \mor J$.
  By \cref{cofibrant-degeneracy}, it will be enough to check that
  the generating degeneracy operators $\sigma_i \from [m + 1] \to [m]$
  act by decidable inclusions.
  The set $(\nerve J)_m$ can be seen as an analogous $(m+1)$-fold pullback with
  a copy of $\ob J$ inserted in the $i$-th place.
  Then the action of $\sigma_i$ is induced by
  a transformation of these limit diagrams that consists of
  identities and the map $\ob J \to \mor J$.
  The conclusion follows by \cref{decidable-limit} of \cref{decidable-closure}.

  \Cref{nerve-cofibration} is a direct consequence of
  \cref{decidable-limit} of \cref{decidable-closure},
  using the pullback presentation of $\nerve I$ and $\nerve J$
  as in the proof of \cref{nerve-cofibrant}.
\end{proof}

Given a presheaf $F$ on a category $\cat{C}$, we write $\cat{C} \slice F$ for the category of elements of $F$.

\begin{lemma} \label{discrete-fibration-decidable-identities}
Let $F$ be a presheaf on a category $\cat{C}$.
If $\cat{C}$ has decidable identities, then so does $\cat{C} \slice F$.
\end{lemma}

\begin{proof}
We form the below diagram, in which the right square is a pullback:
\begin{tikzeq*}
\matrix[diagram,column sep={8em,between origins}]
{
  |(E0)| \ob (\cat{C} \downarrow F) & |(E1)| \mor (\cat{C} \downarrow F) & |(E0')| \ob (\cat{C} \downarrow F)                 \\
  |(C0)| \ob \cat{C}                & |(C1)| \mor \cat{C}                & |(C0')| \ob \cat{C}                \rlap{\text{.}} \\
};

\draw[->] (E0)  to node[above] {$\id$}  (E1) ;
\draw[->] (E1)  to node[above] {$\cod$} (E0');
\draw[->] (C0)  to node[above] {$\id$}  (C1) ;
\draw[->] (C1)  to node[above] {$\cod$} (C0');

\draw[->] (E0)  to (C0) ;
\draw[->] (E1)  to (C1) ;
\draw[->] (E0') to (C0');

\pb{E1}{C0'};
\end{tikzeq*}
By pullback pasting, the left square is also a pullback.
Since the identity structure map of $\cat{C} \slice F$ is a pullback of the identity structure map of $\cat{C}$, it is a decidable inclusion by \cref{decidable-pullback} of \cref{decidable-properties}.
\end{proof}

For the reader familiar with discrete Grothendieck fibrations, we note that the above lemma has a natural phrasing in terms of a discrete Grothendieck fibration $\cat{E} \to \cat{C}$.

\medskip

In order to prove closure of cofibrant objects under certain exponentials, we need some preliminary lemmas.
A simplicial map $A \to B$ is called a \emph{generalised degeneracy} if for every cofibrant simplicial set $X$
the induced map $\sSet(B, X) \to \sSet(A, X)$ is a decidable inclusion.

\begin{lemma}\label{poset-epimorphism-decidable}
  The class of generalised degeneracies contains the degeneracy operators and is closed under composition, finite colimits, pushouts along arbitrary maps
  and retracts.
\end{lemma}

\begin{proof}
  Degeneracy operators are generalised degeneracies by \cref{cofibrant-degeneracy}.
  The closure properties hold since the class of decidable inclusions satisfies the dual closure properties by \cref{decidable-properties,decidable-closure}.
\end{proof}



\begin{lemma}\label{cylinder-cofinal}
  If a simplicial set $A$ is a finite colimit of representables, then the product projection $A \times \simp{1} \to A$ is a generalised degeneracy.
\end{lemma}

\begin{proof}
  Since $A$ is a finite colimit of representables, the claim reduces to the case $A = \simp{m}$ by \cref{poset-epimorphism-decidable}.
  The projection $\simp{m} \times \simp{1} \to \simp{m}$ is the row-wise colimit of the diagram
  \begin{tikzeq*}
  \matrix[diagram,column sep={6em,between origins}]
  {
    |(P0)| \simp{m + 1} & |(P01)| \simp{m} & |(P1)| \simp{m + 1} & |(Pd)| \ldots & |(Pm1)| \simp{m + 1} & |(Pm1m)| \simp{m} & |(Pm)| \simp{m + 1} \\
    |(B0)| \simp{m}     & |(B01)| \simp{m} & |(B1)| \simp{m}     & |(Bd)| \ldots & |(Bm1)| \simp{m}     & |(Bm1m)| \simp{m} & |(Bm)| \simp{m}     \\
  };

  \draw[->] (P0)  to node[left]  {$\dgn_0$}       (B0);
  \draw[->] (P1)  to node[left]  {$\dgn_1$}       (B1);
  \draw[->] (Pm1) to node[right] {$\dgn_{m - 1}$} (Bm1);
  \draw[->] (Pm)  to node[right] {$\dgn_m$}       (Bm);

  \draw[->] (P01)  to (B01);
  \draw[->] (Pm1m) to (Bm1m);

  \draw[->] (P01)  to node[above] {$\face_1$} (P0);
  \draw[->] (P01)  to node[above] {$\face_1$} (P1);
  \draw[->] (Pm1m) to node[above] {$\face_m$} (Pm1);
  \draw[->] (Pm1m) to node[above] {$\face_m$} (Pm);

  \draw[->] (Pd) to (P1);
  \draw[->] (Pd) to (Pm1);

  \draw[->] (B01)  to (B0);
  \draw[->] (B01)  to (B1);
  \draw[->] (Bm1m) to (Bm1);
  \draw[->] (Bm1m) to (Bm);
  \draw[->] (Bd)   to (B1);
  \draw[->] (Bd)   to (Bm1);
  \end{tikzeq*}
  (where all unlabelled morphisms are identities), \ie, it is a finite colimit of degeneracy operators and identity maps.
  Hence the conclusion follows from \cref{poset-epimorphism-decidable}.
\end{proof}

A \emph{deformation section} of a map $f \from A \to B$ is a map $s \from B \to A$ \st{} $f s = \id_Y$
and there is a zig-zag of fiberwise homotopies from $s f$ to $\id_A$ over $A$.
A map that admits a deformation section is called \emph{shrinkable}.
Note that shrinkability is an entirely fiberwise notion.
It follows that shrinkable maps are closed under pullback.

\begin{lemma}\label{shrinkable-collapse}
  A shrinkable map between simplicial sets that are finite colimits of representables is a generalised degeneracy.
\end{lemma}

\begin{proof}
  Let $f \from A \to B$ be a shrinkable map with a deformation section $s$.
  A zig-zag of homotopies connecting $s f$ to $\id_A$ can be represented as a map $H \from A \times J \to A$ where $J$ is its indexing zig-zag, \ie,
  a colimit of a finite number of copies of $\simp{1}$ attached to each other at their endpoints in the directions matching the directions of the homotopies.
  The product projection $A \times J \to A$ is a finite colimit of projections $A \times \simp{1} \to A$ (and identity maps) and thus it is a generalised degeneracy
  by \cref{poset-epimorphism-decidable,cylinder-cofinal}.
  In the diagram
  \begin{tikzeq*}
  \matrix[diagram,column sep={6em,between origins}]
  {
    |(A0)| A & |(CA)|            A \times J  & |(A1)| A                 \\
    |(B0)| B & |(Ci)| B \push_A (A \times J) & |(B1)| B \rlap{\text{,}} \\
  };

  \draw[->] (A0) to (B0);
  \draw[->] (A1) to (B1);
  \draw[->] (CA) to (Ci);

  \draw[->] (A0) to (CA);
  \draw[->] (B0) to (Ci);

  \draw[->] (CA) to (A1);
  \draw[->] (Ci) to (B1);
  \end{tikzeq*}
  the left square is a pushout by construction and so is the outer rectangle since the horizontal composites are identities.
  Thus the right square is also a pushout.
  Moreover, $H$ exhibits $f$ a retract of $B \push_A (A \times J) \to B$ and thus \cref{poset-epimorphism-decidable} completes the proof.
\end{proof}

\begin{lemma}\label{degeneracy-shrinkable}
  Any degeneracy operator $\dgn \from \simp{m} \to \simp{n}$ is shrinkable.
\end{lemma}

\begin{proof}
  Let $\face \from \simp{n} \to \simp{m}$ be the minimal section of $\dgn$.
  Then we have $\face \dgn (i) \le i$ for each $i \in [m]$, which induces a fiberwise homotopy from $\face \dgn$ to $\id_{\simp{m}}$.
\end{proof}

\begin{corollary} \label{pullback-of-degeneracy-shrinkable}
Let $A \to B$ be a pullback of a degeneracy operator $\simp{m} \to \simp{n}$ with $B$ a finite colimit of representables.
Then $A \to B$ is a generalised degeneracy.
\end{corollary}

\begin{proof}
Simplicial sets that are finite colimits of representables are closed under finite limits, in particular pullback.
It follows that also $A$ is a finite colimit of representables.
Since $\simp{m} \to \simp{n}$ is shrinkable by \cref{degeneracy-shrinkable}, so is its pullback $A \to B$.
Then $A \to B$ is a generalised degeneracy by \cref{shrinkable-collapse}.
\end{proof}

\begin{corollary} \label{exponential-finite-cofibrant}
Let $A$ be a finite colimit of simplices.
Exponentiation with $A$ preserves cofibrant simplicial sets.
\end{corollary}

\begin{proof}
Let $X$ be a cofibrant simplicial set.
To show that $X^A$ is cofibrant, it suffices to show for every degeneracy operator $\simp{m} \to \simp{n}$ that $A \times \simp{m} \to A \times \simp{n}$ is a generalised degeneracy.
Since simplicial sets that are finite colimits of representables are closed under finite limits, in particular products, $A \times \simp{n}$ is a finite colimit of representables.
Then the claim is given by \cref{pullback-of-degeneracy-shrinkable}.
\end{proof}

We can alternatively derive \cref{exponential-finite-cofibrant} from the following, more general statement.

\begin{proposition} \label{pushforward-preserves-cofibrancy}
Let $f \from X \to Y$ be a map such that $Y$ is cofibrant and the pullback of $X$ along every map $\simp{n} \to Y$ is a finite colimit of representables.
Then the dependent product $\Pi_f \from \sSet \slice X \to \sSet \slice Y$ preserves cofibrant objects.
\end{proposition}

\begin{proof}
Consider $A \to X$ with $A$ cofibrant.
Take a degeneracy operator $\sigma \from [m] \sto [n]$ and a simplex $y \in Y_n$.
By assumption, $f^* \simp{n}$ is a finite colimit of representables, so $f^* \simp{m} \to f^* \simp{n}$ is a generalised degeneracy by \cref{pullback-of-degeneracy-shrinkable}.
Thus, $\sSet(f^* \simp{n}, A) \to \sSet(f^* \simp{m}, A)$ and $\sSet(f^* \simp{n}, X) \to \sSet(f^* \simp{m}, X)$ are decidable inclusions.
Since $(\sSet \slice X)(f^* \simp{n}, A)$ is a fiber of $\sSet(f^* \simp{n}, A) \to \sSet(f^* \simp{n}, X)$ and similarly for $m$, the case for pullbacks of \cref{decidable-limit} of \cref{decidable-closure} implies that
\begin{equation*}
(\sSet \slice X)(f^* \simp{n}, A) \to (\sSet \slice X)(f^* \simp{m}, A)
\end{equation*}
is a decidable inclusion.
Finally, we have
\begin{equation*}
(\Pi_f A)_n = \bigcoprod_{y \in Y_n} (\sSet \slice X)(f^* \simp{n}, A)
\end{equation*}
and similarly for $m$, hence $(\Pi_f A)_n \to (\Pi_f A)_m$ decomposes into a coproduct of decidable inclusions followed by a decidable inclusion induced by $Y_n \to Y_m$ and thus is a decidable inclusion itself.
\end{proof}

The following statement is also proved in~\cite{Gambino-Henry}*{Lemma~5.1} by different methods.

\begin{corollary} \label{cofibrant-pushforward}
Let $i \from X \to Y$ be a cofibration between cofibrant objects.
Then the dependent product $\Pi_i \from \sSet \slice X \to \sSet \slice Y$ preserves cofibrant objects.
\end{corollary}

\begin{proof}
This reduces to \cref{pushforward-preserves-cofibrancy} once we verify for every map $\simp{n} \to Y$ that $i^* \simp{n}$ is a finite colimit of representables.
By \cref{pullback-of-cofibrations}, the map $i^* \simp{n} \to \simp{n}$ is a cofibration.
Finally, a levelwise decidable subobject of a finite colimit of representables is again a finite colimit of representables.
\end{proof}

\begin{remark}
We limited ourselves to proving only the facts on cofibrations needed for the constructive Kan--Quillen model structure, but more can be said.
For example, a map over a cofibrant simplicial set $X$ is a cofibration \iff{} its pullback along every map $\simp{m} \to X$ is a cofibration.
\end{remark}

\subsection{Homotopies and \he{}s}
\label{sec:htpy-and-whe}

The issue of defining the weak equivalences of the Kan--Quillen model structure
is fairly delicate in the constructive framework.
In fact, our two proofs use different definitions that
are only concluded to agree once both arguments are complete.
However, in both cases the definitions ultimately go back to the notion of
a \he{} between cofibrant Kan complexes.
In this section we establish some basic properties of \he{}s
common to both approaches.

\begin{lemma}\label{composite-homotopy}
  Let $f$, $g$ and $h$ be maps $X \to Y$ and
  let $G$ and $H$ be homotopies
  from $f$ to $g$ and from $f$ to $h$ respectively.
  If $X$ is cofibrant and
  there is a Kan fibration $p \from Y \fto M$ \st{} $p G = p H$
  (in particular, $p g = p h$), then
  there is a fiberwise homotopy from $g$ to $h$ over $M$.
\end{lemma}

\begin{proof}
  There is a commutative square
  \begin{tikzeq*}
  \matrix[diagram,column sep={between origins,6em}]
  {
    |(h)| X \times \horn{2,0} & |(Y)| Y                 \\
    |(s)| X \times \simp{2}   & |(M)| M \rlap{\text{,}} \\
  };

  \draw[->] (h) to node[above] {$[G, H]$}     (Y);
  \draw[->] (s) to node[below] {$p G \dgn_1$} (M);

  \draw[ano] (h) to (s);

  \draw[fib] (Y) to node[right] {$p$} (M);
  \end{tikzeq*}
  which has a diagonal filler $J$
  by \cref{ano-cof-pushout-product} of \cref{second-pushout-product}.
  Then $J \face_0$ is a homotopy from $g$ to $h$ over $M$.
\end{proof}

Using \cref{composite-homotopy}, note that maps $f, g \from X \to Y$ from a cofibrant object $X$ to a fibrant object $Y$ are homotopic \iff{} there is a homotopy from $f$ to $g$ or vice versa, {\ie}, the homotopy relation is witnessed by single-step homotopies.
A similar remark applies to fiberwise homotopies.
We will frequently use this without further reference.

A simplicial map $f \from X \to Y$ is a \emph{\he{}}
if there is a map $g \from Y \to X$ \st{} $g f$ is homotopic to $\id_X$
and $f g$ is homotopic to $\id_Y$.
If $X$ and $Y$ are both cofibrant and fibrant, this means there is a homotopy from $g f$ to $\id_X$ and from $f g$ to $\id_Y$.

\begin{lemma} Homotopy equivalences satisfy: \label{he-closure}
\begin{enumerate}
\item \label{he-closure:2-out-of-6} 2-out-of-6,
\item \label{he-closure:retract} closure under retracts.
\end{enumerate}
\end{lemma}

\begin{proof}
Recall that the homotopy relation forms a congruence.
A map is a homotopy equivalence if and only if it is an isomorphism under this congruence.
Since isomorphisms satisfy 2-out-of-6 and closure under retracts, so do homotopy equivalences.
\end{proof}

\begin{lemma}\label{he-hom}
  If a simplicial map $f \from X \to Y$ is a \he{},
  then for all simplicial sets $Z$ so is the induced map $Z^Y \to Z^X$.
\end{lemma}

\begin{proof}
  The functor $Z^{(\uvar)}$ is simplicial and thus preserves homotopies.
\end{proof}

Dually to the definition of a shrinkable map of the preceding section,
a \emph{strong deformation retraction} of a map $f \from X \to Y$ is a map $r \from Y \to X$ \st{} $r f = \id_X$ and there is a homotopy $H$ from $f r$ to $\id_Y$ under $X$, \ie, $H f$ is constant.

\begin{lemma} \label{trivial-to-she}
\leavevmode
\begin{enumerate}
\item \label{trivial-fibration-shrinkable}
A trivial fibration between cofibrant simplicial sets is shrinkable.
In particular, it is a fiberwise homotopy equivalence over its codomain.
\item \label{trivial-cofibration-sdr}
A trivial cofibration between fibrant simplicial admits a strong deformation retraction.
In particular, it is a homotopy equivalence.
\end{enumerate}
\end{lemma}

\begin{proof}
  We only do \cref{trivial-fibration-shrinkable} as \cref{trivial-cofibration-sdr} is dual.
  Let $p \from X \to Y$ be a trivial fibration.
  Consider the square
  \begin{tikzeq*}
  \matrix[diagram]
  {
    |(e)| \emptyset & |(X)|  X \\
    |(Y0)| Y        & |(Y1)| Y \rlap{\textrm{,}} \\
  };

  \draw[->]  (e) to (X);
  \draw[cof] (e) to (Y0);

  \draw[->]   (Y0) to node[below] {$\id_Y$} (Y1);
  \draw[tfib] (X)  to node[right] {$p$}     (Y1);
  \end{tikzeq*}
  which has a lift $s \from Y \to X$ since $Y$ is cofibrant.
  Then there is a square
  \begin{tikzeq*}
  \matrix[diagram,column sep={between origins,7em}]
  {
    |(bd)| X \times \bd \simp{1} & |(X)| X \\
    |(s1)| X \times \simp{1}     & |(Y)| Y \rlap{\textrm{,}} \\
  };

  \draw[->]  (bd) to node[above] {$[\id_X, s p]$} (X);

  \draw[cof] (bd) to (s1);

  \draw[->]   (s1) to node[below] {$p \pi$} (Y);
  \draw[tfib] (X)  to node[right] {$p$}     (Y);
  \end{tikzeq*}
  which also admits a lift $H \from X \times \simp{1} \to X$
  by \cref{ano-cof-pushout-product} of \cref{second-pushout-product} since
  $X$ is cofibrant.
  Therefore, $s$ is a deformation section of $p$.
\end{proof}

\subsection{The fibration category of Kan fibrations over a base in cofibrant simplicial sets}
\label{fibration-category}

In this final subsection of~\cref{sec:preliminaries},
we establish that if $M$ is a cofibrant simplicial set, then
the full subcategory of cofibrant simplicial sets over $M$ spanned by fibrant objects
admits a fibration category structure (\cref{fibration-category-slice-cofibrant}).
This result completes the preliminaries needed to carry out our proofs of
the existence of the constructive Kan--Quillen model structure. 
Let us begin by recalling that a \emph{fibration category} is
a category $\cat{C}$ equipped with a subcategory of \emph{weak equivalences}
(denoted by $\weto$) and
a subcategory of \emph{fibrations} (denoted by $\fto$)
subject to the axioms (F1--4) listed below;
note that a fibration is called \emph{acyclic} if it is a weak equivalence.

\begin{fibcat-axioms}
\item\label{fibcat-terminal}
  $\cat{C}$ has a terminal object $1$ and all objects are fibrant.
\item\label{fibcat-pullback}
  Pullbacks along fibrations exist in $\cat{C}$ and
  (acyclic) fibrations are stable under pullback.
\item\label{fibcat-factorisation}
  Every morphism factors as a weak equivalence followed by a fibration.
\item\label{fibcat-2-out-of-6}
  Weak equivalences satisfy the 2-out-of-6 property.
\end{fibcat-axioms}

Let $M$ be a simplicial set and $X$ and $Y$ simplicial sets over $M$.
A simplicial map $f \from X \to Y$ over $M$ is a \emph{\fhe{}} if
there is a map $g \from Y \to X$ over $M$ \st{}
$g f$ is fiberwise homotopic to $\id_X$ and
$f g$ is fiberwise homotopic to $\id_Y$.
In this subsection, we call a map in $\sSet_\cof \fslice M$ \emph{acyclic} if it is a \fhe{}.

\begin{lemma}\label{he-cotensor}
  If a simplicial map $f \from X \to Y$ is a \he{},
  then for all simplicial sets $Z$ over $M$
  the map $f \cotensor Z \from Y \cotensor Z \to X \cotensor Z$ is a \fhe{}.
\end{lemma}

\begin{proof}
  The functor $\uvar \cotensor Z$ is simplicial and thus
  carries homotopies to fiberwise homotopies.
\end{proof}

Let $\sSet_\cof$ be the full subcategory of $\sSet$ spanned by
the cofibrant simplicial sets and
let $\sSet_\cof \fslice M$ be
the full subcategory of $\sSet_\cof \slice M$ spanned by
Kan fibrations over $M$. Recall the notion of a shrinkable map from \cref{sec:cofibrations}.

\begin{lemma}\label{acyclic-fibration-shrinkable}
  An acyclic Kan fibration in $\sSet_\cof \fslice M$ is shrinkable.
\end{lemma}

\begin{proof}
  Let $K$ and $L$ be objects of $\sSet_\cof \fslice M$ and
  let $p \from K \to L$ be an acyclic Kan fibration over $M$.
  Since $p$ is a \he{}, there exist a map $\tilde s \from L \to K$ and
  homotopies $\tilde H$ from $p \tilde s$ to $\id_L$ and
  $\tilde G$ from $\tilde s p$ to $\id_K$.
  The commutative square
  \begin{tikzeq*}
  \matrix[diagram,column sep={between origins,6em}]
  {
    |(0)| L \times \{ 0 \}  & |(K)| K                   \\
    |(1)| L \times \simp{1} & |(L)| L  \\
  };

  \draw[->] (0) to node[above] {$\tilde s$} (K);
  \draw[->] (1) to node[below] {$\tilde H$} (L);

  \draw[ano] (0) to (1);

  \draw[fib] (K) to node[right] {$p$} (L);
  \end{tikzeq*}
  admits a lift $H \from L \times \simp{1} \to K$
  by \cref{ano-cof-pushout-product} of \cref{second-pushout-product}.
  Set $s = H \iota_1$ so that $H$ is a homotopy from $\tilde s$ to $s$.
  Then we have $p s = p H \iota_1 = \tilde H \iota_1 = \id_L$.
  Moreover, $s p \htp \tilde s p$ so we can pick
  a direct homotopy $\hat G$ from $s p$ to $\id_K$ 
  by \cref{composite-homotopy}.
  Now, $s p \hat G$ is a homotopy from $s p s p = s p$ to $s p$ and
  we also have $p s p \hat G = p \hat G$.
  Thus by \cref{composite-homotopy} there is
  a homotopy $G$ from $p s$ to $\id_K$ over $L$.
\end{proof}

Note that the proof above only used the assumption that $p$ is a \he{}, which is
in general weaker than a \fhe{}.
This argument implies that these two notions coincide for fibrations,
but in the current section we put emphasis on \fhe{}.

\begin{lemma}\label{shrinkable-fibration-trivial}
  A shrinkable Kan fibration is trivial.
\end{lemma}

\begin{proof}
  Assume that $p \from K \to L$ is a shrinkable fibration with a deformation section $s \from L \to K$ with $p s = \id_L$
  and a zigzag of homotopies from $s p$ to $\id_K$ over $L$.
  Let this zig-zag be represented by $ H \from K \times J \to K$ with $J$ as its indexing zig-zag.  
  Consider a cofibration $i \from A \cto B$ and the first of the following two squares.
  \begin{tikzeq*}
  \matrix[diagram]
  {
    |(A)| A & |(K)| K &[5em] |(A1)| A \times J \union B &[4em] |(K1)| K \\
    |(B)| B & |(L)| L &      |(B1)| B \times J          &      |(L1)| L \\
  };

  \draw[->]  (A) to node[above] {$u$} (K);
  \draw[cof] (A) to node[left]  {$i$} (B);

  \draw[->]  (B) to node[below] {$v$} (L);
  \draw[fib] (K) to node[right] {$p$} (L);

  \draw[->]  (A1) to node[above] {$[H u, \tilde w]$} (K1);

  \draw[ano] (A1) to (B1);

  \draw[->]  (B1) to node[below] {$v \pi$} (L1);
  \draw[fib] (K1) to node[right] {$p$}     (L1);
  \end{tikzeq*}
  Set $\tilde w = s v$.
  Then $\tilde w i = s v i = s p u$, so $H u$ is a zig-zag of homotopies from $\tilde w i$ to $u$ over $L$.
  Moreover, $p \tilde w = p s v = v$ and so the right square above commutes and admits a lift $G \from B \times J \to K$
  by \cref{ano-cof-pushout-product} of \cref{second-pushout-product} (note that the inclusion of an endpoint into $J$ is a trivial cofibration).
  Set $w = G \iota_1$.
  Then $p w = p G \iota_1 = v \pi \iota_1 = v$ and $w i = G \iota_1 i = G i \iota_1 = H u \iota_1 = u$ so that
  $w$ is a lift in the original square and hence $p$ is a trivial fibration.
\end{proof}

Alternatively, \cref{shrinkable-fibration-trivial} can be concluded as a special case of \cref{she-as-retract},
at least for shrinkable maps with homotopies indexed over $\simp{1}$, which is sufficient for our purposes (see \cref{she-for-cylinders} for the general case).
Indeed, such a shrinkable Kan fibration is a strong homotopy equivalence, hence a retract of its pullback exponential with $\{ 0 \} \to \simp{1}$,
which is a trivial fibration by \cref{second-pushout-product}.
However, we gave a self-contained proof above for the sake of clarity.

\begin{corollary}\label{fiberwise-acyclic-fibration-trivial-Kan}
  A morphism in $\sSet_\cof \fslice M$ is
  an acyclic Kan fibration \iff{}
  it is a trivial fibration.
\end{corollary}

\begin{proof}
  An acyclic Kan fibration is trivial
  by \cref{acyclic-fibration-shrinkable,shrinkable-fibration-trivial}.
  Conversely, a trivial fibration is a Kan fibration by \cref{triv-cof-cof}.
  Moreover, it is shrinkable by \cref{trivial-fibration-shrinkable} of \cref{trivial-to-she}.
  Hence it is a \fhe{} over its codomain and thus also over $M$.
\end{proof}


\begin{theorem}\label{fibration-category-slice-cofibrant}
  If $M$ is cofibrant, the category $\sSet_\cof \fslice M$ with \fhe{}s and (underlying) Kan fibrations is a fibration category.
  The acyclic fibrations are the trivial fibrations.
\end{theorem}

\begin{proof}
  \Cref{fibcat-terminal} holds since $\id_M$ is the terminal object and
  all objects are fibrant by definition.
  For \cref{fibcat-pullback}, Kan fibrations are stable under pullback since
  they are defined by a \rlp{}
  (and cofibrancy is preserved by \cref{cofibrant-finite-limits}).
  The same argument applies to acyclic fibrations, which are the trivial fibrations by \cref{fiberwise-acyclic-fibration-trivial-Kan}.
  To verify \cref{fibcat-factorisation} it suffices
  (by \cite{br}*{p.~421, Factorization lemma}) to
  factor the diagonal map $K \to K \pull_M K$.
  In the factorisation $K \to \simp{1} \cotensor K \to \bdsimp{1} \cotensor K \iso K \pull_M K$,
  the first map is a \fhe{} by \cref{he-cotensor} (since $\simp{1} \to \simp{0}$ is a \he{}) and
  the second one is a Kan fibration by \cref{cof-fib-pullback-cotensor} of \cref{cof-tensor-cotensor}.
  Note that $\simp{1} \cotensor K$ is cofibrant by \cref{exponential-finite-cofibrant} and \cref{cofibrant-finite-limits}.
  \Fhe{}s satisfy the 2-out-of-6 property (\cref{fibcat-2-out-of-6})
  by the same argument as in
  the proof of \cref{he-closure:2-out-of-6} of \cref{he-closure}.
\end{proof}

  \section{The model structure via the $\Ex^\infty$ functor}
  \label{sec:first-proof}

In this section we present the first proof of our main theorem.
Our approach follows closely classical
simplicial homotopy theory and readers familiar with that area will recognise
many standard concepts and ideas such as Kan's $\Ex$ functor
(following the treatment of Latch--Thomason--Wilson~\cite{ltw}),
diagonals of bisimplicial sets or Quillen's Theorem A.
Constructively, however, much of that theory is
valid only for cofibrant simplicial sets and requires more delicate arguments,
which occupy most of this section.
It is only in the final subsection where we are able to
go beyond cofibrant objects and establish enough of their properties to
construct the full Kan--Quillen model structure.

\subsection{The \whe{}s}
\label{whes}

Recall that in \cref{sec:preliminaries} we defined the fibrations and cofibrations of
the Kan--Quillen model structure.
We now move on to introduce its weak equivalences, called \emph{\whe{}s}.
However, we do not give a direct general definition.
Instead, we split it into four cases, each building on the previous one.

Some of these definitions will depend on the notion of a \emph{strong cofibrant replacement} of a simplicial set $X$,
which is a cofibrant simplicial set $\tilde X$ equipped with a trivial fibration $\tilde X \to X$.
A strong cofibrant replacement of a simplicial map $f \from X \to Y$ is
a simplicial map $\tilde f \from \tilde X \to \tilde Y$ equipped with a square
\begin{tikzeq*}
\matrix[diagram]
{
  |(tX)| \tilde X & |(X)| X \\
  |(tY)| \tilde Y & |(Y)| Y \\
};

\draw[->] (tX) to node[left]  {$\tilde f$} (tY);
\draw[->] (X)  to node[right] {$f$}        (Y);

\draw[tfib] (tX) to (X);
\draw[tfib] (tY) to (Y);
\end{tikzeq*}
where $\tilde X$ and $\tilde Y$ are cofibrant and
both horizontal maps are trivial fibrations.
A strong cofibrant replacement always exists by~\cref{thm:wfs-via-soa}.

Let $f \from X \to Y$ be a simplicial map.
\begin{axioms}[label=\axm{W}]
\item\label{whe-cof-Kan-def} If $X$ and $Y$ are cofibrant Kan complexes, then
  $f$ is a \whe{} if it is a \he{}.
\item\label{whe-Kan-def} If $X$ and $Y$ are Kan complexes, then
  $f$ is a \whe{} if it has a strong cofibrant replacement that is a \whe{}
  in the sense of \ref{whe-cof-Kan-def}.
\item\label{whe-cof-def} If $X$ and $Y$ are cofibrant, then
  $f$ is a \whe{} if for every Kan complex $K$
  the induced map $f^* \from K^Y \to K^X$ is
  a \whe{} in the sense of \ref{whe-Kan-def}.
  Note that this is a valid definition since
  $K^X$ and $K^Y$ are Kan complexes by
  \cref{cof-fib-pullback-hom} of \cref{second-pushout-product}.
\item\label{whe-arbitrary-def} If $X$ and $Y$ are arbitrary, then
  $f$ is a \whe{} if it has a strong cofibrant replacement that is a \whe{}
  in the sense of \ref{whe-cof-def}.
\end{axioms}

We collect some basic properties of definition~\ref{whe-cof-Kan-def}.

\begin{lemma}\label{whe-cof-Kan-properties}
  In the category of cofibrant Kan complexes:
  \begin{enumerate}
  \item\label{whe-cof-Kan-2-out-of-6} \whe{}s in the sense of \ref{whe-Kan-def}
    satisfy 2-out-of-6;
  \item\label{whe-cof-Kan-retract} \whe{}s in the sense of \ref{whe-Kan-def}
    are closed under retracts;
  \item\label{whe-cof-Kan-trivial-fibration}
    every trivial fibration satisfies \ref{whe-Kan-def};
  \end{enumerate}
\end{lemma}

\begin{proof}
\Cref{whe-cof-Kan-2-out-of-6,whe-cof-Kan-retract} were verified in \cref{he-closure}.
\Cref{whe-cof-Kan-trivial-fibration} was verified in \cref{trivial-fibration-shrinkable} of \cref{trivial-to-she}.
\end{proof}

We proceed to establish the analogous properties of the remaining definitions.
Having done that, we will be able to show, in \cref{whe-compatible} below, that
these apparently different notions are mutually consistent.

\begin{lemma}\label{whe-Kan-independent}
  Definition \ref{whe-Kan-def} does not depend on the choice of strong cofibrant replacements.
\end{lemma}

\begin{proof}
  Let $f \from X \to Y$ be a map between Kan complexes and
  let $\tilde f$ and $\tilde f'$ be strong cofibrant replacements of it.
  Form a diagram
  \begin{tikzeq*}
  \matrix[diagram]
  {
    |(tX'')| \tilde X'' & |(bX)| \bullet &                 & |(tX')| \tilde X' &         \\
                        &                & |(tX)| \tilde X &                   & |(X)| X \\
    |(tY'')| \tilde Y'' & |(bY)| \bullet &                 & |(tY')| \tilde Y' &         \\
                        &                & |(tY)| \tilde Y &                   & |(Y)| Y \\
  };

  \draw[->] (tX'') to (tY'');
  \draw[->] (bX)   to (bY);

  \draw[->] (tX') to node[above left,yshift=8pt] {$\tilde f'$} (tY');

  \draw[tfib] (tX'') to (bX);
  \draw[tfib] (bX)   to (tX');
  \draw[tfib] (tY'') to (bY);
  \draw[tfib] (bY)   to (tY');

  \draw[->,over] (tX) to node[above left,yshift=8pt] {$\tilde f$} (tY);
  \draw[->]      (X)  to node[right]      {$f$}        (Y);

  \draw[tfib,over] (tX) to (X);
  \draw[tfib]      (tY) to (Y);

  \draw[tfib] (bX)  to (tX);
  \draw[tfib] (tX') to (X);
  \draw[tfib] (bY)  to (tY);
  \draw[tfib] (tY') to (Y);
  \end{tikzeq*}
  where the top and bottom squares are pullbacks and
  the square on the left is another strong cofibrant replacement.
  Then the maps $\tilde X'' \to \tilde X$ and $\tilde X'' \to \tilde X'$
  (as well as their counterparts for $Y$) are
  trivial fibrations between cofibrant Kan complexes and hence satisfy \ref{whe-cof-Kan-def}
  by \cref{whe-cof-Kan-trivial-fibration} of \cref{whe-cof-Kan-properties}.
  Thus, by \cref{whe-cof-Kan-2-out-of-6} of \cref{whe-cof-Kan-properties},
  $\tilde f$ satisfies \ref{whe-cof-Kan-def} \iff{} $\tilde f'$ does.
\end{proof}

\begin{lemma}\label{whe-Kan-properties}
  In the category of all Kan complexes:
  \begin{enumerate}
  \item\label{whe-Kan-2-out-of-6} \whe{}s in the sense of \ref{whe-Kan-def}
    satisfy 2-out-of-6;
  \item\label{whe-Kan-retract} \whe{}s in the sense of \ref{whe-Kan-def}
    are closed under retracts;
  \item\label{whe-Kan-trivial-fibration}
    every trivial fibration satisfies \ref{whe-Kan-def};
  \item\label{whe-Kan-he} every \he{} satisfies \ref{whe-Kan-def}.
  \end{enumerate}
\end{lemma}

\begin{proof}
  For \cref{whe-Kan-2-out-of-6,whe-Kan-retract}, it follows from functorial cofibrant replacement (as given by \cref{thm:wfs-via-soa}) and \cref{whe-Kan-independent} that \whe{}s in the sense of~\ref{whe-Kan-def} are created by a functor from \whe{}s in the sense of~\ref{whe-cof-Kan-def}.
  The latter satisfy 2-out-of-6 and closure under retracts by \cref{whe-cof-Kan-2-out-of-6,whe-cof-Kan-retract} of \cref{whe-cof-Kan-properties}, hence so do the former.

  For \cref{whe-Kan-trivial-fibration}, if $f \from X \to Y$ is a trivial fibration between Kan complexes, pick a strong cofibrant replacement
  $\tilde{X} \tfto X$. The identity on $\tilde{X}$ is a strong cofibrant replacement of $f$, which therefore satisfies \ref{whe-Kan-def}.

  For \cref{whe-Kan-he}, first note that the endpoint projections $X^{\simp{1}} \to X$ of path objects are trivial fibrations by \cref{ano-fib-pullback-hom} of \cref{second-pushout-product}, hence weak homotopy equivalences by \cref{whe-Kan-trivial-fibration}.
  Next, every map homotopic to a weak homotopy equivalence is a weak homotopy equivalence.
  For this, we use 2-out-of-3 to see that in every diagram
  \begin{tikzeq*}
  \matrix[diagram]
  {
             & |(X)| X             &          \\
    |(Y0)| Y & |(Yp)| Y^{\simp{1}} & |(Y1)| Y \\
  };

  \draw[->] (X) to (Y0);
  \draw[->] (X) to (Y1);
  \draw[->] (X) to (Yp);

  \draw[->] (Yp) to node[below] {$\weq$} (Y0);
  \draw[->] (Yp) to node[below] {$\weq$} (Y1);
  \end{tikzeq*}
  if one of the maps $X \to Y$ is a weak homotopy equivalence, then so is the other.
  Finally, let $f \from X \to Y$ be a map with a homotopy inverse $g$.
  Since $g f$ and $f g$ are homotopic to identities, they are weak homotopy equivalences.
  Then $f$ is a weak homotopy equivalence by 2-out-of-6.
\end{proof}

\begin{remark}
One may avoid using functoriality of cofibrant replacements in the proof of \cref{whe-Kan-properties} by using Reedy cofibrant replacements.
For this, we note that the indexing categories of the respective diagrams ($[3]$ in \cref{whe-Kan-2-out-of-6} and the walking retract in \cref{whe-Kan-retract}) are Reedy categories.
\end{remark}

\begin{lemma}\label{whe-cof-properties}
  In the category of cofibrant simplicial sets:
  \begin{enumerate}
  \item \label{whe-cof-2-out-of-6} \whe{}s in the sense of \ref{whe-cof-def} satisfy 2-out-of-6;
  \item \label{whe-cof-retract} \whe{}s in the sense of \ref{whe-cof-def} are closed under retracts;
  \item \label{whe-cof-trivial-fibration} every trivial fibration satisfies \ref{whe-cof-def};
  \item \label{whe-cof-he} every \he{} satisfies \ref{whe-cof-def}.
  \end{enumerate}
\end{lemma}

\begin{proof}
  \Cref{whe-cof-2-out-of-6,whe-cof-retract} reduce to the respective parts of \cref{whe-Kan-properties}.
  We prove \cref{whe-cof-he} before dealing with \cref{whe-cof-trivial-fibration}.
    If $f \from X \to Y$ is a \he{} with $X$ and $Y$ are cofibrant and $K$ is a Kan complex, then $f^* \from K^Y \to K^X$ is a \he{} by \cref{he-hom}.
    Thus, it satisfies \ref{whe-Kan-def} by \cref{whe-Kan-he} of \cref{whe-Kan-properties}, and so $f$ satisfies \ref{whe-cof-def}.
  For \cref{whe-cof-trivial-fibration}, a trivial fibration between cofibrant simplicial sets is a \he{}
    by \cref{whe-cof-Kan-trivial-fibration} of \cref{whe-cof-Kan-properties} and thus
    it satisfies \ref{whe-cof-def} by \cref{whe-cof-he}.
\end{proof}

\begin{lemma}\label{whe-arbitrary-independent}
  Definition \ref{whe-arbitrary-def} does not depend on the choice of strong cofibrant replacements.
\end{lemma}

\begin{proof}
  This follows by the same argument as \cref{whe-Kan-independent}, using
  \cref{whe-cof-properties} instead of \cref{whe-cof-Kan-properties}.
\end{proof}

\begin{lemma}\label{whe-arbitrary-properties}
  In the category of all simplicial sets:
  \begin{enumerate}
  \item \label{whe-2-out-of-6} \whe{}s in the sense of \ref{whe-arbitrary-def} satisfy 2-out-of-6;
  \item \label{whe-retract} \whe{}s in the sense of \ref{whe-arbitrary-def} are closed under retracts;
  \item \label{whe-trivial-fibration} every trivial fibration satisfies \ref{whe-arbitrary-def};
  \item \label{whe-he} every \he{} satisfies \ref{whe-arbitrary-def}.
  \end{enumerate}
\end{lemma}

\begin{proof}
  \Cref{whe-2-out-of-6,whe-retract,whe-trivial-fibration} follow by the arguments for the respective parts of \cref{whe-Kan-properties}, with \cref{whe-cof-properties,whe-arbitrary-independent} used in the place of \cref{whe-cof-Kan-properties,whe-Kan-independent}.
  For \cref{whe-he}, pick a strong cofibrant replacement $\tilde X \to X$.
    Then the square
    \begin{tikzeq*}
    \matrix[diagram,column sep={between origins,6em}]
    {
      |(tX1)| \tilde X^{\simp{1}} & |(X1)| X^{\simp{1}} \\
      |(tX)|  \tilde X            & |(X)|  X            \\
    };

    \draw[tfib] (tX1) to (X1);
    \draw[tfib] (tX)  to (X);

    \draw[->] (tX1) to (tX);
    \draw[->] (X1)  to (X);
    \end{tikzeq*}
    is a strong cofibrant replacement of the projection $X^{\simp{1}} \to X$ (for either endpoint)
    by \cref{exponential-finite-cofibrant} and
    \cref{cof-tcof-pullback-hom} of \cref{first-pushout-product}.
    The projection $\tilde X^{\simp{1}} \to \tilde X$ satisfies \ref{whe-cof-def}
    by \cref{whe-cof-he} of \cref{whe-cof-properties} since it is a homotopy equivalence by \cref{he-hom},
    so $X^{\simp{1}} \to X$ satisfies \ref{whe-arbitrary-def} (and similarly for $Y$).
    Therefore, $f$ also satisfies \ref{whe-arbitrary-def} by \cref{whe-2-out-of-6} and the argument of \cref{whe-Kan-he} of \cref{whe-Kan-properties}.
    \end{proof}

\begin{proposition}\label{whe-compatible}
  Let $f \from X \to Y$ be a simplicial map.
  \begin{enumerate}
  \item \label{whe-compatible:cof-Kan-and-Kan}
    If $X$ and $Y$ are cofibrant Kan complexes, then
    $f$ satisfies \ref{whe-cof-Kan-def} \iff{} it satisfies \ref{whe-Kan-def}.
  \item \label{whe-compatible:cof-Kan-and-cof}
    If $X$ and $Y$ are cofibrant Kan complexes, then
    $f$ satisfies \ref{whe-cof-Kan-def} \iff{} it satisfies \ref{whe-cof-def}.
  \item \label{whe-compatible:Kan-and-arbitrary}
    If $X$ and $Y$ are Kan complexes, then
    $f$ satisfies \ref{whe-Kan-def} \iff{} it satisfies \ref{whe-arbitrary-def}.
  \item \label{whe-compatible:cof-and-arbitrary}
    If $X$ and $Y$ are cofibrant simplicial sets, then
    $f$ satisfies \ref{whe-cof-def} \iff{} it satisfies \ref{whe-arbitrary-def}.
  \end{enumerate}
\end{proposition}

\begin{proof}
  \Cref{whe-compatible:cof-Kan-and-Kan,whe-compatible:cof-and-arbitrary} follow from \cref{whe-Kan-independent,whe-arbitrary-independent}, respectively.
  \Cref{whe-compatible:Kan-and-arbitrary} is a consequence of \cref{whe-compatible:cof-Kan-and-Kan}.

  It remains to check \cref{whe-compatible:cof-Kan-and-cof}.
  If $f$ satisfies \ref{whe-cof-Kan-def}, then
  it satisfies \ref{whe-cof-def} by \cref{he-hom} and
  \cref{whe-cof-he} of \cref{whe-cof-properties}.
  Conversely, assume that $f$ satisfies \ref{whe-cof-def} and
  pick a strong cofibrant replacement of $f^* \from X^Y \to X^X$:
  \begin{tikzeq*}
  \matrix[diagram]
  {
    |(EY)| E_Y & |(XY)| X^Y \\
    |(EX)| E_X & |(XX)| X^X \\
  };

  \draw[->] (EY) to node[left]  {$\tilde f$} (EX);
  \draw[->] (XY) to node[right] {$f^*$}      (XX);

  \draw[tfib] (EY) to node[above] {$u_Y$} (XY);
  \draw[tfib] (EX) to node[below] {$u_X$} (XX);
  \end{tikzeq*}
  so that $\tilde f \from E_Y \to E_X$ is a \he{}.
  Let $\phi \from E_X \to E_Y$ be its homotopy inverse.
  Since $u_X$ is a trivial fibration, we can pick $i \in E_X$ \st{}
  $u_X i = \id_X$.
  If we set $g = u_Y \phi i \from Y \to X$,
  then we have
  \begin{equation*}
    g f = f^* g = f^* u_Y \phi i = u_X \tilde f \phi i \htp u_X i = \id_X
    \text{.}
  \end{equation*}
  Next, choose a strong cofibrant replacement of $f^* \from Y^Y \to Y^X$:
  \begin{tikzeq*}
  \matrix[diagram]
  {
    |(EY)| E'_Y & |(YY)| Y^Y \\
    |(EX)| E'_X & |(YX)| Y^X \\
  };

  \draw[->] (EY) to node[left]  {$\tilde f'$} (EX);
  \draw[->] (YY) to node[right] {$f^*$}      (YX);

  \draw[tfib] (EY) to node[above] {$u'_Y$} (YY);
  \draw[tfib] (EX) to node[below] {$u'_X$} (YX);
  \end{tikzeq*}
  with $\phi' \from E'_X \to E'_Y$ a homotopy inverse of $\tilde f'$.
  Since $u'_Y$ is a trivial fibration, we can pick $i', j' \in E'_Y$ \st{}
  $u'_Y i' = \id_Y$ and $u'_Y j' = f g$.
  Then we have $u'_X \tilde f' i' = f^* u'_Y i' = f^* \id_Y = f$ and
  $u'_X \tilde f' j' = f^* u'_Y j' = f^* (f g) = f g f$.
  However, we already know that $g f \htp \id_X$, and thus $f g f \htp f$.
  Since $u'_X$ is a trivial fibration, the latter homotopy lifts to
  $\tilde f' j' \htp \tilde f' i'$, and hence
  \begin{equation*}
    f g = u'_Y j' \htp u'_Y \phi' \tilde f' j' \htp u'_Y \phi' \tilde f' i'
    \htp u'_Y i' = \id_Y \text{.}
  \end{equation*}
  Therefore, $g$ is a homotopy inverse of $f$, \ie,
  $f$ satisfies \ref{whe-cof-Kan-def}.
\end{proof}



Having established that our definitions of \whe{}s are mutually compatible,
we will use them interchangeably, often without comment.
In particular, ``acyclic (co)fibration'' will refer to
a (co)fibration that is also a \whe{}. See~\cref{tab:arrows} for the notation used to denote these maps.

We conclude the subsection with an observation that will be useful later.
The following statements holds also without the cofibrancy assumption since
strong cofibrant replacements are closed under product (as cofibrant objects and trivial fibrations are).
However, we will not need that stronger statement.

\begin{corollary}\label{whe-product}
  If $f \from X \to Y$ is a \whe{} between cofibrant
  simplicial sets and
  $A$ is a cofibrant simplicial set, then $f \times A$ is a \whe{}.
\end{corollary}

\begin{proof}
  For any Kan complex $K$, we have a commutative square
  \begin{tikzeq*}
  \matrix[diagram,column sep={between origins,6em}]
  {
    |(YxA)| K^{Y \times A}   & |(XxA)| K^{X \times A}           \\
    |(YA)| (K^A)^Y           & |(XA)| (K^A)^X \rlap{\textrm{.}} \\
  };

  \draw[->] (YxA) to node[above] {$(f \times A)^*$} (XxA);
  \draw[->] (YA)  to node[below] {$f^*$}            (XA);

  \draw[->] (YxA) to node[left]  {$\iso$} (YA);
  \draw[->] (XxA) to node[right] {$\iso$} (XA);
  \end{tikzeq*}
  By \cref{cof-fib-pullback-hom} of \cref{second-pushout-product}, $K^A$ is a Kan complex, and therefore
  $f^*$ is a \whe{} (in the sense of \ref{whe-Kan-def}).
  Hence so is $(f \times A)^*$, and thus $f \times A$ is a \whe{}
  (in the sense of \ref{whe-cof-def}).
\end{proof}

\subsection{The (co)fibration category of (co)fibrant simplicial sets}
\label{sec:the-cofibration-category-of-cofibrant-ssets}

In this subsection,
we establish the fibration category of fibrant simplicial sets
(\ie, Kan complexes) and
the cofibration category of cofibrant simplicial sets.
This will give us
sufficient understanding of constructive simplicial homotopy theory to
prove a number of intermediate results in the following subsections.
We will then use these results to derive the full model structure.
The case of the fibration category reduces directly to
\cref{fibration-category-slice-cofibrant} (specialised to $M = \simp{0}$).

\begin{proposition}\label{acyclic-fibration-trivial-Kan}
  A simplicial map between Kan complexes is an acyclic Kan fibration \iff{}
  it is a trivial fibration.
\end{proposition}

\begin{proof}
  Let $f \from X \to Y$ be an acyclic Kan fibration between Kan complexes.
  Choose a strong cofibrant replacement $\tilde Y \tfto Y$, take a pullback and
  another strong cofibrant replacement:
  \begin{tikzeq*}
  \matrix[diagram]
  {
    |(tX)| \tilde X & |(b)|  \bullet  & |(X)| X                 \\
                    & |(tY)| \tilde Y & |(Y)| Y \rlap{\text{.}} \\
  };

  \draw[tfib] (tX) to (b);
  \draw[fib]  (b)  to (tY);

  \draw[fib] (X) to node[right] {$f$} (Y);

  \draw[tfib] (b)  to (X);
  \draw[tfib] (tY) to (Y);
  \end{tikzeq*}
  Then both $\tilde X \to X$ and $\tilde Y \to Y$ are trivial fibrations, and
  so is the composite $\tilde X \to \tilde Y$ by \cref{fiberwise-acyclic-fibration-trivial-Kan}.
  It follows that $f$ is also a trivial fibration
  by \cref{trivial-fibration-cancellation}.

  Conversely, a trivial fibration is a Kan fibration by \cref{triv-cof-cof}.
  Moreover, it is acyclic by \cref{whe-Kan-trivial-fibration} of \cref{whe-Kan-properties}.
\end{proof}

The fibration category of Kan complexes established in the next theorem
satisfies certain additional axioms
beyond~(F1--4) given in \cref{fibration-category}, listed below.
They assert that certain infinite limits exist and are
well-behaved \wrt{} fibrations and acyclic fibrations.
Such a fibration category is called \emph{complete}.
The only reason why
the fibration category of \cref{fibration-category-slice-cofibrant}
is not complete is that cofibrant objects are
not generally closed under infinite limits.

\begin{fibcat-axioms}[]
\item\label{fibcat-product} $\cat{C}$ has products and
  (acyclic) fibrations are stable under products.
\item\label{fibcat-tower}
  $\cat{C}$ has limits of countable towers of fibrations and
  (acyclic) fibrations are stable under such limits.
\end{fibcat-axioms}

\begin{theorem}\label{fibration-category-Kan}
  The category of Kan complexes
  (\ie, the full subcategory of the category of simplicial sets spanned by
  the Kan complexes) with \whe{}s (in the sense of \ref{whe-Kan-def}) and
  Kan fibrations is a complete fibration category.
\end{theorem}

\begin{proof}
\Cref{fibcat-terminal} holds since $\simp{0}$ is the terminal Kan complex and
  all Kan complexes are fibrant by definition.
  For \cref{fibcat-pullback}, Kan fibrations are stable under pullback since
  they are defined by a \rlp{}.
  The same argument applies to acyclic fibrations
  by \cref{acyclic-fibration-trivial-Kan}.
  For \cref{fibcat-factorisation}, it suffices
  (by \cite{br}*{p.~421, Factorization lemma}) to
  factor the diagonal map $K \to K \times K$.
  In the factorisation $K \to K^{\simp{1}} \to K \times K$,
  the first map is a homotopy equivalence by \cref{he-hom}
  and hence a \whe{} by \cref{whe-Kan-he} of \cref{whe-Kan-properties} and
  the second one is a Kan fibration
  by \cref{cof-fib-pullback-hom} of \cref{second-pushout-product}.
  For \cref{fibcat-2-out-of-6}, \whe{}s satisfy the 2-out-of-6 property
  by \cref{whe-Kan-2-out-of-6} of \cref{whe-Kan-properties}.
  Finally, \cref{fibcat-product,fibcat-tower} follow by
  the same argument as \cref{fibcat-pullback}.
\end{proof}

We go on to show that the category of cofibrant simplicial sets carries
a structure of a \emph{cocomplete cofibration category}, {\ie},
it satisfies axioms (C1--6) dual to
the axioms (F1--4) of \cref{fibration-category} and (F5--6) above.
As usual, the critical difficulty lies in showing that
acyclic cofibrations are closed under pushout.
In contrast to the arguments
of \cref{fibration-category-slice-cofibrant,fibration-category-Kan},
we do not establish any lifting property of acyclic cofibrations.
Instead, the theorem below is proved by
reduction to \cref{fibration-category-Kan} via
the exponential functors $K^{(\uvar)}$ for all Kan complexes $K$
(which justifies the choice of \ref{whe-cof-def} as
the definition of \whe{}s between cofibrant objects).
Later, in \cref{acyclic-cofibration-anodyne}, we will show that
acyclic cofibrations coincide with trivial cofibrations and thus
are characterised by a lifting property.

\begin{theorem}\label{cofibration-category-cofibrant}
  The category of cofibrant simplicial sets with \whe{}s
  (in the sense of \ref{whe-cof-def}) and cofibrations is
  a cocomplete cofibration category.
\end{theorem}

\begin{proof}
  Axiom~(C1) holds since $\emptyset$ is the initial cofibrant simplicial sets
  and all objects are cofibrant by definition.
  For axiom~(C2), cofibrations are stable under pushout since they are defined by
  a \llp{}.
  To verify that acyclic cofibrations are stable under pushout,
  consider a pushout square
  \begin{tikzeq*}
  \matrix[diagram]
  {
    |(C)| C & |(X)| X \\
    |(D)| D & |(Y)| Y \\
  };

  \draw[cof] (C) to node[left] {$i$} node[right] {$\weq$} (D);
  \draw[cof] (X) to node[right] {$j$} (Y);
  \draw[->]  (C) to node[above] {$u$} (X);
  \draw[->]  (D) to node[below] {$v$} (Y);
  \pbdr{Y}{C};
  \end{tikzeq*}
  where $i$ is an acyclic cofibration (and all objects are cofibrant).
  Then for any Kan complex $K$ there is
  an induced pullback square of Kan complexes
  (by \cref{cof-fib-pullback-hom} of \cref{second-pushout-product})
  \begin{tikzeq*}
  \matrix[diagram]
  {
    |(Y)| K^Y & |(D)| K^D \\
    |(X)| K^X & |(C)| K^C \\
  };

  \draw[fib] (D) to node[right] {$i^*$} node[left] {$\weq$} (C);
  \draw[fib] (Y) to node[left]  {$j^*$} (X);
  \draw[->]  (X) to node[below] {$u^*$} (C);
  \draw[->]  (Y) to node[above] {$v^*$} (D);
  \pb{Y}{C};
  \end{tikzeq*}
  where $i^*$ is a Kan fibration
  by \cref{cof-fib-pullback-hom} of \cref{second-pushout-product} and
  a \whe{} by definition.
  Thus $j^*$ is a \whe{} by \cref{fibration-category-Kan},
  and hence $j$ is a \whe{}.
  To verify axiom~(C3), it suffices (by \cite{br}*{p.~421, Factorization lemma})
  to factor the codiagonal map $X \coprod X \to X$.
  In the factorisation $X \coprod X \to X \times \simp{1} \to X$,
  the first map is a cofibration
  by \cref{cofibration-pushout-product} of \cref{first-pushout-product} and
  the second one is a \whe{} by \cref{whe-cof-he} of \cref{whe-cof-properties}.
  For axiom~(C4), \whe{}s satisfy the 2-out-of-6 property
  by \cref{whe-cof-2-out-of-6} of \cref{whe-cof-properties}.
  Axioms (C5) and (C6) follow by the same argument as (C2).
\end{proof}

\subsection{Diagonals of bisimiplicial sets}
\label{sec:diagonals-of-ssets}

In this subsection we prove \cref{bisimplicial-diagonal-whe},
a constructive version of the classical result saying that
the diagonal of a bisimplicial map that is a levelwise \whe{} is itself
a \whe{}.
We can establish this only under a suitable cofibrancy assumption, but
our argument follows a standard approach
(e.g., as in \cite{Goerss-Jardine}*{Proposition~1.9}), which relies only on
the cocomplete cofibration category of cofibrant simplicial sets that
we constructed in \cref{cofibration-category-cofibrant}.

For purposes of the present subsection, we consider bisimplicial sets as
simplicial objects in the category of simplicial sets.
In particular, we will use the fact that a cofibration of bisimplicial sets
(\ie, a Reedy decidable inclusion over $\Simp \times \Simp$) is
the same thing as a Reedy cofibration over $\Simp$ \wrt{}
cofibrations of simplicial sets (see \cite{riehl-verity-reedy}*{Example~4.6}).
If $A$ and $B$ are simplicial sets, then their \emph{external product} is the bisimplicial set $A \etimes B$ given as $(A \etimes B)_{m, n} = A_m \times B_n$.

The $k$-th \emph{skeleton} of a bisimplicial set $X$ is the coend $\Sk^k X = \coend^{[m] \in \Simp} X_m \etimes (\Sk^k \simp{m})$
where $\Sk^k \simp{m}$ stands for the simplicial subset of $\simp{m}$ spanned by its $k$-simplices.
For completeness, we will establish a few basic facts about skeleta, a more abstract discussion can be found in \cite{riehl-verity-reedy}*{Section 6}.

\begin{lemma}\label{bisimplicial-skeletal-filtration}
  The skeleta of a bisimplicial set $X$ come with
  canonical morphisms $\Sk^k X \to \Sk^l X$ for all $l \ge k \ge -1$.
  These morphisms exhibit $X$ as the colimit of the resulting sequence
  \begin{equation*}
    \Sk^{-1} X \to \Sk^0 X \to \Sk^1 X \to \ldots
  \end{equation*}
\end{lemma}

\begin{proof}
  The functors in the definition of the skeleta come with maps
  \begin{equation*}
    \Sk^{-1} \simp{\uvar} \to \Sk^0 \simp{\uvar} \to \Sk^1 \simp{\uvar} \to \ldots
  \end{equation*}
  and the colimit of this sequence is $\simp{\uvar}$.
  Thus the colimit of
  \begin{equation*}
    \Sk^{-1} X \to \Sk^0 X \to \Sk^1 X \to \ldots
  \end{equation*}
  is $\coend^{[m]} X_m \etimes \simp{m} \iso X$.
\end{proof}

\begin{lemma}
  For all natural numbers $k$, $m$ and $n$, the square
  \begin{tikzeq*}
  \matrix[diagram,column sep={20em,between origins}]
  {
      |(b)| \bdsimp{k}_m \times \Simp([k], [n]) \union \simp{k}_m \times \bd\Simp([k], \uvar)_n
    & |(k-1)| \Sk^{k - 1} \simp{n}_m \\
      |(p)| \simp{k}_m \times \Simp([k], [n])
    & |(k)| \Sk^k \simp{n}_m \\
  };

  \draw[->] (b) to (k-1);
  \draw[->] (p) to (k);

  \draw[inj] (b)   to (p);
  \draw[inj] (k-1) to (k);
  \end{tikzeq*}
  is a pushout.
\end{lemma}

\begin{proof}
  The set $\bd\Simp([k], \uvar)$ is the boundary of the cosimplicial set $\Simp([k], \uvar)$, i.e.,
  $\bd\Simp([k], \uvar)_n$ consists of simplicial operators $\alpha \from [k] \to [n]$ \st{} $\alpha^\flat \ne \id_{[k]}$.
  Similarly, $\bdsimp{k}_m$ consists of operators $\beta \from [m] \to [k]$ \st{} $\beta^\sharp \ne \id_{[k]}$.
  Thus the union in the upper left corner is the set of pairs $(\beta, \alpha)$ \st{} the composite $\alpha \beta$ factors through $[k - 1]$.
  The conclusion follows since $\Sk^k \simp{n}_m$ is the set of operators $[m] \to [n]$ that factor through $[k]$ and $\Sk^{k - 1} \simp{n}_m$
  is the set of those that factor through $[k - 1]$.
\end{proof}

\begin{corollary}\label{bisimplicial-skeleton-to-skeleton}
  For any bisimplicial set $X$ the square
  \begin{tikzeq*}
  \matrix[diagram,column sep={11em,between origins}]
  {
      |(b)| \bdsimp{k} \etimes X_k \union \simp{k} \etimes L_k X & |(k-1)| \Sk^{k - 1} X \\
      |(p)|   \simp{k} \etimes X_k                               & |(k)|   \Sk^k X \\
  };

  \draw[->] (b) to (k-1);
  \draw[->] (p) to (k);

  \draw[->] (b)   to (p);
  \draw[->] (k-1) to (k);
  \end{tikzeq*}
  is a pushout.
\end{corollary}

\begin{proof}
  When $m$ and $n$ vary in the square of the preceding lemma, we obtain a pushout square of functors $\Simp \to \sSet$.
  Applying the coend $\coend^{[m]} X_m \etimes (\uvar)_m$ yields the required pushout square.
\end{proof}

In the remainder of this section, we will freely use
the cocomplete cofibration category of cofibrant simplicial sets
established in \cref{cofibration-category-cofibrant} in order to
invoke some standard results from \cite{rb}.

\begin{lemma}\label{bisimplicial-latching-whe}
  If $X \to Y$ is a map between cofibrant bisimplicial sets \st{}
  $X_k \to Y_k$ is a \whe{} for all $k$, then
  the induced map $L_k X \to L_k Y$ is also a \whe{} for all $k$.
\end{lemma}

\begin{proof}
  The latching object $L_k X$ can be written as a colimit over $\bd([k] \slice \Simp_\flat)^\op$ (the opposite of the latching category).
  That category is direct and the diagram (sending $[k] \sto [l]$ to $X_l$) is Reedy cofibrant since $X$ is
  (as follows from the fact that the latching categories of $\bd([k] \slice \Simp_\flat)$ are isomorphic to the latching categories of $\Simp_\flat$).
  Thus the conclusion follows from \cite{rb}*{Theorem~9.3.5~(1c)}.
\end{proof}

\begin{proposition}\label{bisimplicial-diagonal-whe}
  If $X \to Y$ is a map between cofibrant bisimiplicial sets \st{}
  $X_k \to Y_k$ is a \whe{} for all $k$, then
  the induced map $\diag X \to \diag Y$ is also a \whe{}.
\end{proposition}

\begin{proof}
  First, we will prove by induction \wrt{} $k$ that
  the induced map $\diag \Sk^k X \to \diag \Sk^k Y$ is a \whe{}.

  For $k = -1$, both $\diag \Sk^k X$ and $\diag \Sk^k Y$ are empty, so
  the statement holds.
  For $k \ge 0$, consider the cube
  \begin{tikzeq*}
  \matrix[diagram,column sep={8em,between origins}]
  {
      |(LX)| L_k X \times \simp{k} \union X_k \times \bdsimp{k}
    & & |(Xk1)| \diag \Sk^{k-1} X & \\
    & |(X)| X_k \times \simp{k} & & |(Xk)| \diag \Sk^k X \\
      |(LY)| L_k Y \times \simp{k} \union Y_k \times \bdsimp{k}
    & & |(Yk1)| \diag \Sk^{k-1} Y & \\
    & |(Y)| Y_k \times \simp{k} & & |(Yk)| \diag \Sk^k Y \\
  };

  \draw[->] (LX)  to (X);
  \draw[->] (Xk1) to (Xk);
  \draw[->] (Xk1) to (Yk1);
  \draw[->] (Xk)  to (Yk);

  \draw[->] (LX) to (Xk1);
  \draw[->,over] (X) to (Xk);

  \draw[->] (LY)  to (Y);
  \draw[->] (Yk1) to (Yk);

  \draw[->] (LY) to (Yk1);
  \draw[->] (Y)  to (Yk);

  \draw[->] (LX) to (LY);
  \draw[->,over] (X) to (Y);
  \end{tikzeq*}
  of which the top and bottom squares arise by applying $\diag$ to
  the squares of \cref{bisimplicial-skeleton-to-skeleton} since $\diag$ carries external products to products.
  These squares are pushouts since $\diag$ preserves pushouts.

  The map $L_k X \to L_k Y$ is a \whe{} by \cref{bisimplicial-latching-whe} and
  therefore so are
  $L_k X \times \simp{k} \to L_k Y \times \simp{k}$ and
  $L_k X \times \bdsimp{k} \to L_k Y \times \bdsimp{k}$ as well as
  $X_k \times \bdsimp{k} \to Y_k \times \bdsimp{k}$ by \cref{whe-product}.
  Thus the left vertical map in the back of the cube
  is a \whe{} by the Gluing Lemma \cite{rb}*{Lemma~1.4.1~(1b)}
  (using the fact that $X$ and $Y$ are cofibrant as well as
  \cref{ano-cof-pushout-product} of~\cref{second-pushout-product}).
  The map $X_k \times \simp{k} \to Y_k \times \simp{k}$ is a \whe{}
  by \cref{whe-product} and so is $\diag \Sk^{k - 1} X \to \diag \Sk^{k - 1} Y$
  by the inductive hypothesis.
  Moreover, the diagonal maps on the left of the cube are cofibrations by
  cofibrancy of $X$ and $Y$ and \cref{cofibration-pushout-product} of~\cref{first-pushout-product}.
  Thus the Gluing Lemma \cite{rb}*{Lemma~1.4.1~(1b)} implies that
  $\diag \Sk^k X \to \diag \Sk^k Y$ is a \whe{}.

  It also follows that the diagonal maps on the right are cofibrations.
  Hence, \cref{bisimplicial-skeletal-filtration} and \cite{rb}*{Theorem~9.3.5~(1c)} imply that
  $\diag X \to \diag Y$ is a \whe{}.
\end{proof}

\subsection{The $\Ex^\infty$ functor}
\label{sec:subdivision-and-ex}

We turn to the constructive treatment of the $\Ex^\infty$ functor.
Classically, it is a fibrant replacement functor with some convenient properties,
most notably preservation of finite limits and Kan fibrations.
In the constructive setting, we are only able to show that
it is a fibrant replacement functor
in the subcategory of cofibrant simplicial sets.
Some of the material below is treated also in~\cite{Henry-qms}*{Section~3}.
However, we do not need to establish some of the more intricate results on
trivial cofibrations proved therein.

If $P$ is a poset, then let $\sd P$ denote
the poset of finite, non-empty, totally ordered subsets of $P$
ordered by inclusion.
(This defines a functor $\sd \from \Pos \to \Pos$.)
Let $\max_P \from \sd P \to P$ denote the (natural, order-preserving) map
sending a finite, non-empty, totally ordered subset of $P$ to
its maximal element.
Let $\Sd \from \sSet \to \sSet$ be
the essentially unique
colimit-preserving functor \st{}
$\Sd \simp{m} = N \sd [m]$ (as functors $\Simp \to \sSet$)
and
$\mu \from \Sd \to \id_\sSet$ be the natural transformation \st{} $\mu_{\simp{m}} = N \max_{[m]}$
(there is a unique such up to isomorphism since $\Sd$ is a left Kan extensions of a functor $\Simp \to \sSet$).
For a poset $P$, we have a natural isomorphism $\Sd \nerve P \iso \nerve \sd P$.
Moreover, if $K$ is a simplicial subset of $\nerve P$, then $\Sd K$ is a simplicial subset of $\nerve \sd P$;
it consists of those simplices whose vertices lie in $K$ (when seen as non-degenerate simplices of $\nerve P$).

Before discussing homotopy theoretic properties of the subdivision, we establish two preliminary lemmas needed for the verification of cofibrancy of various
simplicial and bisimimplicial sets constructed later in this section.

\begin{lemma}\label{Sd-degeneracy-collapse}
  If $\dgn \from [m] \sto [n]$ is a degeneracy operator, then the induced map $\Sd \simp{m} \to \Sd \simp{n}$ is a generalised degeneracy.
\end{lemma}

\begin{proof}
  Let $\face \from [n] \to [m]$ be a section of $\dgn$.
  Then $\Sd \face$ is a section of $\Sd \dgn$ and we will show that it is a deformation section.
  Let $f \from \sd [m] \to \sd [n]$ be given by $A \mapsto A \union \face \dgn A$ for every finite non-empty subset $A \subseteq [m]$.
  Then we have $A \subseteq f A \supseteq (\sd \face) (\sd \dgn) A$, which induces
  a zig-zag of fiberwise homotopies $\id_{\Sd\simp{m}} \to \nerve f \leftarrow (\Sd \face) (\Sd \dgn)$.
  Thus, $\Sd \dgn$ is shrinkable and hence a generalised degeneracy by \cref{shrinkable-collapse}.
\end{proof}

\begin{lemma}\label{Sd-simplex-he}
  For every $m$, the map $\mu_{\simp{m}} \from \Sd \simp{m} \to \simp{m}$ is shrinkable.
  In particular, it is a \he{} and a generalised degeneracy.
\end{lemma}

\begin{proof}
  Define an order-preserving map $\iota_{[m]} \from [m] \to \sd [m]$
  by $\iota_{[m]}(i) = \braces{0, \ldots, i}$.
  Then $\max_{[m]} \iota_{[m]} = \id_{[m]}$ and
  $\iota_{[m]} \max_{[m]} \supseteq \id_{\sd [m]}$ so that
  $N \iota_{[m]}$ is a deformation section of $\mu_{\simp{m}}$.
  Thus $\mu_{\simp{m}}$ is shrinkable and a \he{}.
  It is a generalised degeneracy by \cref{shrinkable-collapse}.
\end{proof}


The subdivision functor $\Sd \from \sSet \to \sSet$ has
a right adjoint denoted by $\Ex$, which
can be constructed as $(\Ex X)_m = \sSet(\Sd \simp{m}, X)$.
Under this adjunction, $\mu \from \Sd \to \id_\sSet$ corresponds to
a natural transformation $\id_\sSet \to \Ex$, which will be denoted by $\nu$.

\begin{proposition}\label{Ex-properties}
  The functor $\Ex$ satisfies the following conditions.
  \begin{enumerate}
  \item \label{Ex-limits}
  It preserves limits.
  \item \label{Ex-fibration}
  It preserves Kan fibrations.
  \item \label{Ex-triv-fibration}
  It preserves trivial fibrations.
  \item \label{Ex-pointing-cofibration}
  If $X$ is cofibrant, then $\nu_X$ is a cofibration. In particular, $\Ex X$ is cofibrant.
  \end{enumerate}
\end{proposition}

\begin{proof} \Cref{Ex-limits} holds since $\Ex$ is a right adjoint.

For \cref{Ex-fibration}, it suffices to show that
    $\Sd$ carries horn inclusions to trivial cofibrations.
    Note that any permutation of $[m]$ induces
    a simplicial automorphism of $\Sd \simp{m}$ that carries horns to horns.
    Thus, it is enough to check that
    $\Sd \horn{m,0} \ito \Sd \simp{m}$ is a trivial cofibration.
    $\Sd \horn{m,0}$ is the nerve of the subposet of $\sd[m]$ spanned by
    all elements except $[m]$ and $[m] \setminus \{ 0 \}$.
    Thus, the inclusion in question factors as
    $\Sd \horn{m,0} \ito \Sd \horn{m,0} \union X \ito \Sd \simp{m}$ where
    $X$ is the nerve of the subposet of $\sd[m]$ spanned by
    all elements $S \subseteq [m]$ \st{} $0 \in S$.
    Moreover, let $Y$ be the nerve of the subposet of $\sd[m]$ spanned by
    all elements except $\{ 0 \}$.
    Then there are pushout squares
    \begin{tikzeq*}
    \matrix[diagram,column sep={10em,between origins}]
    {
      |(i)|  \Sd \horn{m,0} \inter X            &
      |(h)|  \Sd \horn{m,0}                     &[-1em]
      |(j)|  (\Sd \horn{m,0} \union X) \inter Y &
      |(hX)| \Sd \horn{m,0} \union X            \\
      |(X)|  X                                  &
      |(u)|  \Sd \horn{m,0} \union X            &
      |(Y)|  Y                                  &
      |(s)|  \Sd \simp{m} \rlap{\text{.}}       \\
    };

    \draw[inj] (i) to (h);
    \draw[inj] (X) to (u);
    \draw[inj] (i) to (X);
    \draw[inj] (h) to (u);

    \draw[inj] (j)  to (hX);
    \draw[inj] (Y)  to (s);
    \draw[inj] (j)  to (Y);
    \draw[inj] (hX) to (s);
    \end{tikzeq*}

    In the left one, $\Sd \horn{m,0} \inter X$ is the nerve of
    the subposet of $\sd[m]$ spanned by
    all elements $S \subseteq [m]$ \st{} $0 \in S$ except $[m]$.
    Thus $\Sd \horn{m,0} \inter X \ito X$ can be identified with
    the $m$-fold pushout product of $\horn{1,0} \to \simp{1}$ and so
    it is a trivial cofibration
  by \cref{ano-cof-pushout-product} of \cref{second-pushout-product}.
    Hence $\Sd \horn{m,0} \ito \Sd \horn{m,0} \union X$ is also a trivial cofibration.
    Similarly, in the right square,  $(\Sd \horn{m,0} \union X) \inter Y \ito Y$ can be identified with the pushout product of
    $\horn{1,1} \ito \simp{1}$ and $\Sd \bdsimp{m-1} \ito \Sd \simp{m-1}$ and so
    it is a trivial cofibration
  by \cref{ano-cof-pushout-product} of \cref{second-pushout-product}.
    (Under this identification, $\simp{m-1}$ corresponds to the face
    $\face_0 \from \simp{m-1} \ito \simp{m}$.)
    Hence, $\Sd \horn{m,0} \union X \ito \Sd \simp{m}$ is also a trivial cofibration and thus
    so is $\Sd \horn{m,0} \ito \Sd \simp{m}$.


For \cref{Ex-triv-fibration}, we need to verify for all $m$ that
    the map $\Sd \bdsimp{m} \to \Sd \simp{m}$ is a cofibration.
    This follows by \cref{nerve-cof} since it is the nerve of
    a decidable inclusion between categories with decidable identities.

For \cref{Ex-pointing-cofibration}, we first check that $\Ex X$ is cofibrant.
    Indeed, by \cref{cofibrant-degeneracy}
    it is enough to check that each degeneracy operator $[m] \sto [n]$ induces
    a decidable inclusion $(\Ex X)_n \to (\Ex X)_m$.
    This inclusion is induced by $\Sd \simp{m} \to \Sd \simp{n}$, so the conclusion follows from \cref{Sd-degeneracy-collapse}.
    Next, by \cref{cofibration-levelwise-decidable} it suffices to verify that $\nu_X \from X \to \Ex X$ is a levelwise decidable inclusion.
    This is a consequence of \cref{Sd-simplex-he} since $X_m \to (\Ex X)_m$
    is induced by $\mu_{\simp{m}} \from \Sd \simp{m} \to \simp{m}$. \qedhere
\end{proof}

For a simplicial set $X$, we define $\Ex^\infty X$ to be the colimit of the sequence
\begin{tikzeq*}
\matrix[diagram,column sep=4em]
{
  |(E0)| X & |(E1)| \Ex X & |(E2)| \Ex^2 X & |(ld)| \ldots \rlap{\text{.}} \\
};

\draw[->] (E0) to node[above] {$\nu_X$}       (E1);
\draw[->] (E1) to node[above] {$\nu_{\Ex X}$} (E2);

\draw[->] (E2) to (ld);
\end{tikzeq*}
We write $\nu^\infty_X \from X \to \Ex^\infty X$ for the resulting natural map.

\Cref{Ex-infty-properties} below extends some properties of the $\Ex$ functor to the $\Ex^\infty$ functor.

\begin{proposition} \label{Ex-infty-properties}
  The functor $\Ex^\infty$ satisfies the following conditions.
  \begin{enumerate}
  \item \label{Ex-infty-finite-limits} It preserves finite limits.
  \item \label{Ex-infty-fibration} It preserves Kan fibrations between cofibrant objects.
  \item \label{Ex-infty-triv-fibration} It preserves trivial fibrations between cofibrant objects.
  \item \label{Ex-infty-fibrant-replacement} If $X$ is cofibrant, then $\nu^\infty_X$ is a cofibration. In particular $\Ex^\infty X$ is cofibrant.
    Moreover, $\Ex^\infty X$ is a Kan complex.
  \end{enumerate}
\end{proposition}

\begin{proof} For \cref{Ex-infty-finite-limits}, observe that $\Ex^\infty$ is
    a filtered colimit of functors and each of these functors
    preserve limits by \cref{Ex-limits} of~\cref{Ex-properties}.
    Hence it preserves finite limits itself.

    \Cref{Ex-infty-fibration,Ex-infty-triv-fibration} follow
    by \cref{omega-colim-of-fibrations} from the corresponding
    \cref{Ex-fibration,Ex-triv-fibration} of \cref{Ex-properties}
    using \cref{Ex-pointing-cofibration} of \cref{Ex-properties}
    to satisfy the requirement that the step maps in the colimit
    are cofibrations.

    For \cref{Ex-infty-fibrant-replacement}, $\nu^\infty_X$ is a cofibration by \cref{Ex-pointing-cofibration} of \cref{Ex-properties}
    and thus $\Ex^\infty X$ is cofibrant.
    To show that it is a Kan complex, we appeal to \cref{small-object-argument-lemma}.
    The step maps $\Ex^k X \to \Ex^{k+1} X$ are cofibrations by \cref{Ex-pointing-cofibration}, hence levelwise decidable.
    It remains to construct the indicated lift in any lifting problem
    \begin{tikzeq*}
    \matrix[diagram,column sep={10em,between origins}]
    {
      |(h)| \horn{m,i} & |(Es)| \Ex^k X & |(Es1)| \Ex^{k + 1} X \rlap{\text{.}} \\
      |(s)| \simp{m}   &                &                                       \\
    };

    \draw[inj] (h) to (s);

    \draw[->] (h)  to node[above] {$x$}            (Es);
    \draw[->] (Es) to node[above] {$\nu_{\Ex^k X}$} (Es1);
    \draw[->,dashed] (s) to (Es1);
    \end{tikzeq*}
    It suffices to have this only for $k \geq 1$.
    Then by adjointness this problem rewrites as
    \begin{tikzeq*}
    \matrix[diagram,column sep={10em,between origins}]
    {
      |(S2h)| \Sd^2 \horn{m,i} & |(Sh)| \Sd \horn{m,i} & |(Es1)| \Ex^{k - 1} X \rlap{\text{.}} \\
      |(S2s)| \Sd^2 \simp{m}   &                       &                                       \\
    };

    \draw[inj] (S2h) to (S2s);

    \draw[->] (S2h) to node[above] {$\Sd \mu_{\horn{m,i}}$}  (Sh);
    \draw[->] (Sh)  to node[above] {$\tilde x$}             (Es1);
    \draw[->,dashed] (S2s) to (Es1);
    \draw[->,dashed] (S2s) to node[above] {$\phi$} (Sh);
    \end{tikzeq*}
    For this, it will suffice to construct the dashed map $\phi$.
    We do so by first defining a map $\phi \from \Sd^2 \simp{m} \to \Sd \simp{m}$,
  showing that its image is in  $\Sd \horn{m,i}$, and finally checking the required commutativity.
    All simplicial sets in the left triangle are nerves of finite posets, so
    it will be enough to define $\phi$ on the underlying posets.
  First, for a finite non-empty subset $A \subseteq [m]$, define
    \begin{equation*}
      \mu' A =
      \begin{cases}
        i      & \text{if } A \in \{ [m], [m] \setminus \{ i \} \}, \\
        \max A & \text{otherwise.}
      \end{cases}
    \end{equation*}
    We then let
    \begin{equation}
    \label{equ:def-of-phi}
    \phi(A_0 \subset \ldots \subset A_p) = \{ \mu' A_0, \ldots, \mu' A_p \},
    \end{equation}
     so that $\phi$ is an order-preserving map.
  We verify that the image of $\phi$  lies in $\Sd \horn{m,i}$. Since
    $\Sd \horn{m,i}$ is the nerve of the poset of the faces of~$\horn{m,i}$ and $\horn{m,i} \ito \simp{m}$ is a levelwise decidable inclusion by \cref{triv-cof-cof,cofibration-levelwise-decidable},
    we can do so by ruling out two cases.
    The first is that $\phi(A_0 \subset \ldots \subset A_p) = [m] \setminus \{ i \}$.
    In this case, we have $i \not\in \phi(A_0 \subset \ldots \subset A_p)$ so that $\phi(A_0 \subset \ldots \subset A_p) = \braces{\max A_0, \ldots, \max A_p}$ and $[m] \setminus \braces{i} \subseteq A_p$.
    This implies that $i \in \phi(A_0 \subset \ldots \subset A_p)$ and that is not possible.
      The second case is that $\phi(A_0 \subset \ldots \subset A_p) = [m]$. But then we would
 have $p = m$ and
    $\phi(A_0 \subset \ldots \subset A_{m - 1}) = [m] \setminus \{ i \}$, which is also
    not possible, as above.
    Finally, by the very definition in~\eqref{equ:def-of-phi}, $\phi$
restricts to $\Sd \mu_{\horn{m,i}}$, as required.
\end{proof}

\begin{proposition}\label{Ex-whe}
  For a cofibrant simplicial set $X$,
  the map $\nu_X \from X \to \Ex X$ is a \whe{}.
\end{proposition}

\begin{proof}
  We use the argument of \cite{ltw}*{Theorem~4.1} enhanced with cofibrancy checks necessary to make it constructive.
  We begin by noticing that $\Ex$ preserves homotopies.
  Indeed, a homotopy $X \times \simp{1} \to Y$ gives a map
  $\Ex X \times \simp{1} \to \Ex X \times \Ex \simp{1} \to \Ex Y$.
  Thus, $\Ex$ also preserves \he{}s.

  Consider the commutative square
  \begin{tikzeq*}
  \matrix[diagram,column sep={15em,between origins}]
  {
    |(i0)| \sSet(    \simp{m} \times \simp{0}, X) &
    |(in)| \sSet(    \simp{m} \times \simp{n}, X) \\
    |(D0)| \sSet(\Sd \simp{m} \times \simp{0}, X) &
    |(Dn)| \sSet(\Sd \simp{m} \times \simp{n}, X) \\
  };

  \draw[->] (i0) to (in);
  \draw[->] (D0) to (Dn);

  \draw[->] (i0) to (D0);
  \draw[->] (in) to (Dn);
  \end{tikzeq*}
  in the category of sets, which becomes a square of bisimplicial sets when $m$ and $n$ vary.
  We will verify that all these are cofibrant.
  By \cref{bisimplicial-cofibrant-degeneracy}, it is enough to verify that all degeneracy operators act on them via decidable inclusions.
  For the left objects, this amounts to cofibrancy of $X$ and $\Ex X$ (\cref{Ex-pointing-cofibration} of \cref{Ex-properties}).
  For the right objects, this reduces to cofibrancy of $X^{\simp{m}}$, $X^{\Sd \simp{m}}$, $X^{\simp{n}}$ and $\Ex (X^{\simp{n}})$,
  which follows from \cref{exponential-finite-cofibrant} and \cref{Ex-pointing-cofibration} of \cref{Ex-properties}.
  Since bisimplicial Reedy cofibrancy coincides with iterated simplicial Reedy cofibrancy,
  all these bisimplicial sets are Reedy cofibrant simplicial objects in $\sSet$ in both directions.

  By fixing either $m$ or $n$ we obtain two squares of simplicial sets
  \begin{tikzeq*}
  \matrix[diagram,column sep={7em,between origins}]
  {
    |(b0)| \bullet           & |(Xi)| X^{\simp{m}}      &
    |(i0)|     X^{\simp{0}}  & |(in)|     X^{\simp{n}}  \\
    |(b1)| \bullet           & |(XD)| X^{\Sd \simp{m}}  &
    |(E0)| \Ex(X^{\simp{0}}) & |(En)| \Ex(X^{\simp{n}}) \\
  };

  \draw[->] (b0) to (b1);

  \draw[->] (Xi) to node[right] {$\weq$} (XD);

  \draw[->] (b0) to (Xi);
  \draw[->] (b1) to (XD);

  \draw[->] (i0) to node[above] {$\weq$} (in);
  \draw[->] (E0) to node[below] {$\weq$} (En);

  \draw[->] (i0) to (E0);
  \draw[->] (in) to (En);
  \end{tikzeq*}
  where the right map in the left square is a \he{} as the image of
  the \he{} $\Sd \simp{m} \to \simp{m}$ of \cref{Sd-simplex-he}
  under $X^{(\uvar)}$,
  the top map in the right square is a \he{} as the image of
  the \he{} $\simp{n} \to \simp{0}$ under $X^{(\uvar)}$
  and then the bottom map in the right square is a \he{} since
  $\Ex$ preserves \he{}s as noted above.
  In the first two cases we use \cref{he-hom} to show that
  $X^{(\uvar)}$ preserves \he{}s.

  \He{}s are \whe{}s by \cref{whe-cof-he} of \cref{whe-cof-properties} and, consequently,
  taking the diagonal simplicial sets in the original square
  (\ie, setting $m = n$) yields
  \begin{tikzeq*}
  \matrix[diagram]
  {
    |(i)| X     & |(b0)| \bullet \\
    |(E)| \Ex X & |(b1)| \bullet \\
  };

  \draw[->] (i) to (E);

  \draw[->] (b0) to node[right] {$\weq$} (b1);

  \draw[->] (i) to node[above] {$\weq$} (b0);
  \draw[->] (E) to node[below] {$\weq$} (b1);
  \end{tikzeq*}
  in which both horizontal and the right vertical map are \whe{}s
  by \cref{bisimplicial-diagonal-whe}.
  Thus, $X \to \Ex X$ is also a \whe{} by the 2-out-of-3 property.
\end{proof}

\begin{proposition}\label{Ex-infty-whe}
  For a cofibrant simplicial set $X$
  the map $\nu^\infty_X \from X \to \Ex^\infty X$ is a \whe{}.
\end{proposition}

\begin{proof}
  This follows from \cref{Ex-pointing-cofibration} of \cref{Ex-properties}, \cref{Ex-whe} and
  \cite{rb}*{Theorem~9.3.5~(1c)}.
\end{proof}

\subsection{An explicit cofibrant replacement functor}
\label{sec:explicit-cofibrant-replacement}

Up to this point,
we have developed a fair amount of homotopy theory of cofibrant simplicial sets.
To move beyond cofibrant objects,
we need a sufficiently well-behaved cofibrant replacement functor.
(Specifically, we need it to preserve pushouts and cofibrations.)
There are a few functors that are suitable.
We use a functor $T$ where $T X$ is defined as
the nerve of the category of simplices of $X$.
However, even to prove all necessary facts about $T$,
we implicitly use another cofibrant replacement functor which
is the variation of $T$ using
the subcategory of face operators of the category of simplices.
Yet another cofibrant replacement functor, denoted $L U$ and
obtained from an adjunction between simplicial and semisimplicial sets,
will be discussed in \cref{sec:properness}.

To deal with homotopy theory of nerves of categories,
we employ the classical Theorem A of Quillen.
As usual, the standard proof technique \cite{q}*{p.~93} is applicable
but only for cofibrant objects.

A simplicial set is \emph{contractible} if the map $X \to \simp{0}$ is a homotopy equivalence.
It is \emph{weakly contractible} if $X \to \simp{0}$ is a weak homotopy equivalence.

\begin{theorem}[Quillen's Theorem A]\label{theorem-A}
  Let $f \from I \to J$ be
  a functor between categories with decidable identities.
  If for every $y \in J$ the nerve $\nerve (f \slice y)$ is weakly contractible,
  then the induced map $\nerve I \to \nerve J$ is a \whe{}.
\end{theorem}

\begin{proof}
  Let $S f$ be a bisimplicial set whose $(m, n)$-bisimplices are
  triples $(x, y, \phi)$ where $x \from [m] \to I$, $y \from [n] \to J$ and
  $\phi \from f x_m \to y_0$.
  It comes with two bisimplicial maps
  \begin{tikzeq*}
  \matrix[diagram,column sep={5em,between origins}]
  {
    |(I)| \nerve I \etimes \simp{0}                 &
    |(S)| S f                                       &
    |(J)| \simp{0} \etimes \nerve J \rlap{\text{.}} \\
  };

  \draw[->] (S) to (I);
  \draw[->] (S) to (J);
  \end{tikzeq*}
  All the bisimplicial sets in this diagram are cofibrant. For the codomains,
  recall that $\nerve I$ and $\nerve J$ are cofibrant by \cref{nerve-cofibrant} of \cref{nerve-cof}.
  \Cref{decidable-limit} of \cref{decidable-closure}, \cref{cofibrant-degeneracy}
  and \cref{bisimplicial-cofibrant-degeneracy}
  imply that the external product (defined in~\cref{sec:diagonals-of-ssets}) of cofibrant simplicial sets is cofibrant.
  For $S f$,
  the set of $(m, n)$-bisimplices can be written as
  the pullback $(\nerve I)_m \pull_{\ob J} \mor J \pull_{\ob_J} (\nerve I)_n$, and
  so the cofibrancy of $S f$ follows from cofibrancy of $\nerve I$ and $\nerve J$
  and \cref{decidable-limit} of \cref{decidable-closure}.

  For a fixed $m$, the left map becomes
  \begin{tikzeq*}
  \matrix[diagram,column sep={9em,between origins}]
  {
    |(I)| \displaystyle\bigcoprod_{[m] \to I} \simp{0}                                        &
    |(S)| \displaystyle\bigcoprod_{x \from [m] \to I} \nerve (f x_m \slice J) \rlap{\text{,}} \\
  };

  \draw[->,shorten <=-2mm] (S) to (I);
  \end{tikzeq*}
  which is a \whe{} since each $\nerve (f x_m \slice J)$ is contractible
  (as $f x_m \slice J$ has an initial object).
  For a fixed $n$, the left map becomes
  \begin{tikzeq*}
  \matrix[diagram,column sep={9em,between origins}]
  {
    |(S)| \displaystyle\bigcoprod_{y \from [n] \to J} \nerve (f \slice y_0)                 &
    |(J)| \displaystyle\bigcoprod_{[n] \to J} \simp{0}                      \rlap{\text{,}} \\
  };

  \draw[->] (S) to (J);
  \end{tikzeq*}
  which is a \whe{} since each $\nerve (f \slice y_0)$ is contractible by assumption.
  (Here, we use the fact that
  \whe{}s between cofibrant objects are closed under coproducts,
  see \cref{cofibration-category-cofibrant}.)

  Thus, by taking diagonals we obtain the diagram
  \begin{tikzeq*}
  \matrix[diagram,column sep={5em,between origins}]
  {
    |(I)|  \nerve I & |(S)|  \diag S f     & |(J)|  \nerve J \\
    |(Ii)| \nerve J & |(Si)| \diag S \id_J & |(Ji)| \nerve J \\
  };

  \draw[->] (S)  to node[above] {$\weq$} (I);
  \draw[->] (S)  to node[above] {$\weq$} (J);
  \draw[->] (Si) to node[below] {$\weq$} (Ii);
  \draw[->] (Si) to node[below] {$\weq$} (Ji);

  \draw[->] (I) to (Ii);
  \draw[->] (S) to (Si);
  \draw[->] (J) to (Ji);
  \end{tikzeq*}
  where all horizontal maps are \whe{}s by \cref{bisimplicial-diagonal-whe}
  (using the fact that
  $\id_{\nerve J}$ satisfies the hypotheses of the theorem).
  The left map is $\nerve f$ and the right one is $\id_{\nerve J}$.
  It follows by 2-out-of-3 that $\nerve f$ is a \whe{}.
\end{proof}

For a simplicial set $X$, set $T X = \nerve (\Simp \slice X)$.
There is a natural map $\tau_X \from T X \to X$ given as follows.
An $m$-simplex of $T X$ is a functor $[m] \to \Simp \slice X$, \ie,
a sequence of simplices
$\simp{k_0} \to \simp{k_1} \to \ldots \simp{k_m} \to X$.
Write $\phi_i$ for the map $[k_i] \to [k_m]$ of this sequence and
let $\bar \phi \from [m] \to [k_m]$ be given by $\bar \phi i = \phi_i k_i$.
Then $\tau_X$ sends the simplex above to
$\simp{m} \to \simp{k_m} \to X$ induced by $\bar \phi$.

\begin{lemma}\label{cofibrant-replacement-he}
  The functor $T$ carries \he{}s to \whe{}s.
\end{lemma}

\begin{proof}
  For a simplicial set $X$, consider the projection $p \from X \times \simp{1} \to X$ and the induced functor
  $p_* \from \Simp \slice (X \times \simp{1}) \to \Simp \slice X$.
  We will begin by verifying that it satisfies the hypothesis of \cref{theorem-A}.
  Note that all categories involved have decidable identities by \cref{discrete-fibration-decidable-identities}.
  Moreover, $p$ is a pullback of $\simp{1} \to \simp{0}$ and thus a slice of $p_*$ over $x \from \simp{m} \to X$ coincides with
  the slice of the induced functor $\Simp \slice [1] \to \Simp$ over $[m]$.
  The latter slice is isomorphic to $\Simp \slice ([m] \times [1])$ so it suffices to show that the nerve of this category is contractible.
  Let $s \from \Simp \slice ([m] \times [1]) \to \Simp \slice ([m] \times [1])$ be a functor that sends $(\phi, \psi) \from [k] \to [m] \times [1]$ to
  $(\phi', \psi') \from [k + 1] \to [m] \times [1]$ where $\phi' i  = \phi i$ and $\psi' i  = \psi i$ for $i \in [k]$
  while $\phi'(k + 1) = m$ and $\psi'(k + 1) = 1$.
  This functor admits natural transformations $c_{m, 1} \to s \leftarrow \id_{\Simp \slice ([m] \times [1])}$
  where $c_{m, 1}$ is the constant functor at $(m, 1) \from [0] \to [m] \times [1]$ which proves the claim.

  \Cref{theorem-A} implies that $\nerve p_* = T p$ is a \whe{}.
  Consequently, $T$ also carries the cylinder inclusions $X \to X \times \simp{1}$ to \whe{}s, and thus it carries all \he{}s to \whe{}s.
\end{proof}

\begin{lemma}\label{category-of-simplices-whe}
  For every simplicial set $X$,
  the inclusion functor $j_X \from \Simp_\sharp \slice X \to \Simp \slice X$
  induces a \whe{} on the nerves.
\end{lemma}

\begin{proof}
  As a preliminary step, we verify that the nerve of $j \slice [m]$ is contractible where $j$ is the inclusion $\Simp_\sharp \to \Simp$.
  Indeed, let $s$ be an endofunctor of that category given by $s \phi = \phi'$
  where $\phi' 0 = 0$ and $\phi' i = \phi(i + 1)$
  (assuming $\phi \from [k] \to [m]$ so that $\phi' \from [1 + k] \to [m]$).
  Then the diagram
  \begin{tikzeq*}
  \matrix[diagram,column sep={5em,between origins}]
  {
    |(0)| [0] & |(1k)| [1 + k] & |(k)| [k] \\
              & |(m)|  [m]     &           \\
  };

  \draw[->] (0)  to node[below left]  {$0$}     (m);
  \draw[->] (1k) to node[left]        {$\phi'$} (m);
  \draw[->] (k)  to node[below right] {$\phi$}  (m);

  \draw[->] (0) to node[above] {$0$}       (1k);
  \draw[->] (k) to node[above] {$\face_0$} (1k);
  \end{tikzeq*}
  exhibits natural transformations $c_0 \to s \leftarrow \id_{j \slice [m]}$ where $c_0$ is the constant functor at $0 \from [0] \to [m]$.
  Thus, $\nerve(j \slice [m])$ is contractible as claimed.

  Finally, note that for each $x \in X_m$, seen as an object of $\Simp \slice X$, the slice category $j_X \slice x$ is isomorphic to $j \slice [m]$
  (since a morphism $j_X y \to x$ is uniquely determined by a simplicial operator $[n] \to [m]$ for any $y \in X_n$).
  Hence the conclusion follows from \cref{theorem-A}
  (note that $\Simp_\sharp \slice X$ has decidable identities by \cref{discrete-fibration-decidable-identities} since the projection to $\Simp_\sharp$ is a discrete Grothendieck fibration).
\end{proof}

\begin{lemma}\label{cofibrant-replacement-trivial-fibration}
  The functor $T$ carries trivial fibrations to \whe{}s.
\end{lemma}

\begin{proof}
  Let $p \from X \to Y$ be a trivial fibration.
  By \cref{category-of-simplices-whe}, it will be enough to show that
  $p_* \from \Simp_\sharp \slice X \to \Simp_\sharp \slice Y$
  induces a \whe{} on the nerves.

  First, we construct a section $s$ of $p_*$ as follows.
  Given $y \from \simp{m} \to Y$, assume inductively that
  $s$ has been already defined on all object of degree less than $m$
  (and morphisms between them).
  We define $s y$ as the solution to the lifting problem
  \begin{tikzeq*}
  \matrix[diagram,column sep={10em,between origins}]
  {
    |(b)| \bdsimp{m} & |(X)| X \\
    |(s)|   \simp{m} & |(Y)| Y \\
  };

  \draw[cof] (b) to (s);

  \draw[tfib] (X) to node[right] {$p$}                                          (Y);
  \draw[->]   (b) to node[above] {$\left[ s(y \face_i) \mid i \in [m] \right]$} (X);
  \draw[->]   (s) to node[below] {$y$}                                          (Y);

  \draw[->,dashed,shorten >=1mm] (s) to node[above left] {$s y$} (X.225);
  \end{tikzeq*}
  where functoriality is guaranteed by the commutativity of the upper triangle.

  Similarly, we define a ``homotopy''
  $H \from \Simp_\sharp \slice X \to \Simp_\sharp \slice (X^{\simp{1}})$
  from $s p_*$ to $\id_{\Simp_\sharp \slice X}$ that
  is fiberwise in the sense that it becomes the ``constant homotopy'' at $p_*$
  when composed with the functor
  $\Simp_\sharp \slice (X^{\simp{1}}) \to \Simp_\sharp \slice (Y^{\simp{1}})$
  induced by $p$.
  Assuming that $H$ was defined at objects of degree less than $m$ and
  given $x \from [m] \to X$,
  we set $H x$ to be the solution to the lifting problem
  \begin{tikzeq*}
  \matrix[diagram,column sep={20em,between origins}]
  {
    |(b)| \bdsimp{m} \times \simp{1} \union \simp{m} \times \bdsimp{1} & |(X)| X \\
    |(s)|   \simp{m} \times \simp{1}                                   & |(Y)| Y \\
  };

  \draw[cof] (b) to (s);

  \draw[tfib] (X) to node[right] {$p$}                                                                   (Y);
  \draw[->]   (b) to node[above] {$\left[ \left[ H(x \face_i) \mid i \in [m] \right], s p x, x \right]$} (X);
  \draw[->]   (s) to node[below] {$p x \pi$}                                                             (Y);

  \draw[->,dashed,shorten >=0.5mm] (s) to node[above left] {$H x$} (X.225);
  \end{tikzeq*}
  (which exists by \cref{cofibration-pushout-product} of~\cref{first-pushout-product}).

  To conclude the proof,
  we note that \cref{cofibrant-replacement-he,category-of-simplices-whe} imply
  that both projections
  $\Simp_\sharp \slice (X^{\simp{1}}) \to \Simp_\sharp \slice X$
  induce \whe{}s on the nerves.
  Therefore, by an argument analogous to \cref{whe-Kan-he} of \cref{whe-Kan-properties}, $\nerve p_*$ is also a \whe{}.
\end{proof}

All properties of the functor $T$ that will be needed later in the paper
are summarised in the following proposition.
(Although we will not use it, the functor $N(\Simp_\sharp \slice \uvar)$ appearing above satisfies the same properties.)

\begin{proposition}\label{cofibrant-replacement}
  The functor $T$ satisfies the following conditions.
  \begin{enumerate}
  \item\label{cofibrant-replacement-colimit}
    It preserves colimits.
  \item\label{cofibrant-replacement-cofibrant}
    It takes values in cofibrant simplicial sets.
  \item\label{cofibrant-replacement-cofibration}
    It preserves cofibrations.
  \item\label{cofibrant-replacement-whe}
    For every simplicial set $X$, $\tau_X \from T X \to X$ is a \whe{}.
  \end{enumerate}
\end{proposition}

\begin{proof}
  \Cref{cofibrant-replacement-colimit} follows from the fact that we can write
  \begin{equation*}
    (T X)_m = \bigcoprod_{[k_0] \to \ldots \to [k_m]} X_{k_m} \text{.}
  \end{equation*}

  \Cref{cofibrant-replacement-cofibrant} is a consequence of
  \cref{nerve-cofibrant} of \cref{nerve-cof}
  and \cref{discrete-fibration-decidable-identities}.

  For \cref{cofibrant-replacement-cofibration}, it is enough to check that
  $T \bdsimp{m} \to T \simp{m}$ is a cofibration.
  By the previous part and \cref{nerve-cofibration} of \cref{nerve-cof},
  it is sufficient to verify that
  $\Simp \slice \bdsimp{m} \to \Simp \slice \simp{m}$ is
  a decidable inclusion.
  That in turn follows from the fact that $\bdsimp{m} \to \simp{m}$ is
  a levelwise decidable inclusion.

  We now prove \cref{cofibrant-replacement-whe}.
  First, observe that $T \simp{m}$ is
  the nerve of a category with a terminal object and hence is contractible.
  Therefore, $\tau_{\simp{m}}$ is a \whe{}.
  Next, we show that $\tau_{\bdsimp{m}}$ is a \whe{} by induction.
  By \cref{simplex-cofibrant}, $\bdsimp{m}$ is
  a cell complex \wrt{} $\set{\bdsimp{k} \to \simp{k}}{k < m}$.
  By the inductive hypothesis and the preceding observation,
  $\tau$ is a \whe{} on the domains and codomains of maps in this set.
  Since  $T$ preserves colimits and cofibrations,
  \cref{cofibration-category-cofibrant} and
  the Gluing Lemma \cite{rb}*{Lemma~1.4.1~(1b)} imply that
  $\tau_{\bdsimp{m}}$ is a \whe{} as well.
  The same argument shows that $\tau_X$ is a \whe{}
  for all $I$-cell complexes $X$ (where $I = \set{ \bdsimp{m} \ito \simp{m} }{ m \ge 0 }$).
  Since \whe{}s are closed under retracts
  by \cref{whe-cof-retract} of \cref{whe-cof-properties},
  the same holds for all cofibrant $X$.
  Finally, the conclusion follows for arbitrary $X$
  by \cref{cofibrant-replacement-trivial-fibration}.
\end{proof}

\subsection{Conclusion of the first proof}
\label{sec:the-model-structure-first-proof}

We are now ready to establish the Kan--Quillen model structure.
We need to show that acyclic fibrations coincide with trivial fibrations.
In \cref{acyclic-fibration-trivial-Kan}, we have already verified this
for maps between Kan complexes.
We follow the argument in \cite{may-ponto}*{Section~17.6},
attributed by the authors to Bousfield,
to extend this result to maps between arbitrary simplicial sets.
This uses the $\Ex^\infty$ functor, but our situation is more subtle due to
the fact that it is a fibrant replacement on cofibrant objects only.

We also need to verify that acyclic cofibrations coincide
with  trivial cofibrations, which will follow by a general retract argument
as soon as we know that trivial cofibrations are \whe{}s.
This is, however, non-trivial for maps between non-cofibrant objects and
relies on good properties of the cofibrant replacement functor $T$.

\begin{lemma}\label{acyclic-fibration-contractible-fibers}
  If $X$ and $Y$ are cofibrant and
  $p \from X \fto Y$ is an acyclic Kan fibration, then
  all fibers of $p$ are contractible.
\end{lemma}

\begin{proof}
  For each $y \in Y_0$, form the diagram
  \begin{tikzeq*}
  \matrix[diagram]
  {
      |(F)|  F_y                 & & |(X)|  X            & \\
    & |(EF)| \Ex^\infty F_y      & & |(EX)| \Ex^\infty X   \\
      |(0)|  \simp{0}            & & |(Y)|  Y            & \\
    & |(E0)| \Ex^\infty \simp{0} & & |(EY)| \Ex^\infty Y   \\
  };

  \draw[->] (F) to (0);
  \draw[->] (X) to node[above right,yshift=8pt] {$p$} (Y);
  \draw[->] (F) to (X);
  \draw[->] (0) to (Y);

  \draw[->,over] (EF) to (E0);
  \draw[->]      (EX) to node[right] {$\Ex^\infty p$} (EY);
  \draw[->,over] (EF) to (EX);
  \draw[->]      (E0) to (EY);

  \draw[->] (F) to node[above right] {$\weq$} (EF);
  \draw[->] (X) to node[above right] {$\weq$} (EX);
  \draw[->] (0) to node[above right] {$\weq$} (E0);
  \draw[->] (Y) to node[above right] {$\weq$} (EY);
  \end{tikzeq*}
  where the back square is a pullback that exhibits
  $F_y$ as the fiber of $p$ over $y$.
  The front square is also a pullback by \cref{Ex-infty-finite-limits} of \cref{Ex-infty-properties} and
  all back-to-front maps are \whe{}s by \cref{Ex-infty-whe}.
  Since $p$ is acyclic, $\Ex^\infty p$ is
  an acyclic Kan fibration by \cref{Ex-infty-fibration} of \cref{Ex-infty-properties} and \cref{whe-cof-2-out-of-6} of \cref{whe-cof-properties}.
  Moreover, all objects in the front face are Kan complexes
  by \cref{Ex-infty-fibrant-replacement} of \cref{Ex-infty-properties},
  and hence $\Ex^\infty F_y \to \Ex^\infty \simp{0}$ is a \whe{}
  by \cref{fibration-category-Kan}.
  Thus, so is $F_y \to \simp{0}$, \ie, $F_y$ is contractible.
\end{proof}

\begin{lemma}\label{contractible-fibers-trivial-fibration}
  A Kan fibration with contractible fibers is trivial.
\end{lemma}

\begin{proof}
  Let $p \from X \to Y$ be a Kan fibration with contractible fibers.
  Consider a lifting problem
  \begin{tikzeq*}
  \matrix[diagram,column sep={between origins,6em}]
  {
    |(b)| \bdsimp{m} & |(X)| X                 \\
    |(s)|   \simp{m} & |(Y)| Y \rlap{\text{.}} \\
  };

  \draw[cof] (b) to (s);

  \draw[->] (X) to node[right] {$p$} (Y);
  \draw[->] (b) to node[above] {$u$} (X);
  \draw[->] (s) to node[below] {$v$} (Y);
  \end{tikzeq*}
  Let $H \from \simp{m} \times \simp{1} \to \simp{m}$ be
  a homotopy from the constant map at $0$ to the identity map.
  It yields another lifting problem
  \begin{tikzeq*}
  \matrix[diagram,column sep={between origins,12em}]
  {
    |(0)| \bdsimp{m} \times \{ 1 \}  & |(X)| X                 \\
    |(1)| \bdsimp{m} \times \simp{1} & |(Y)| Y \rlap{\text{,}} \\
  };

  \draw[ano] (0) to (1);

  \draw[fib] (X) to node[right] {$p$} (Y);
  \draw[->]  (0) to node[above] {$u$} (X);

  \draw[->] (1) to node[below] {$v H | \bdsimp{m} \times \simp{1}$} (Y);
  \end{tikzeq*}
  which has a solution $\tilde G \from \bdsimp{m} \times \simp{1} \to X$
  by \cref{ano-cof-pushout-product} of \cref{second-pushout-product}.
  Such $\tilde G$ is a homotopy from $\tilde u$ to $u$ and
  the former factors through the fiber $F_{v_0}$ of $p$ over $v_0$ to
  yield a diagram
  \begin{tikzeq*}
  \matrix[diagram,column sep={between origins,6em}]
  {
    |(b)| \bdsimp{m} & |(F)| F_{v_0}  & |(X)| X                 \\
    |(s)|   \simp{m} & |(0)| \simp{0} & |(Y)| Y \rlap{\text{.}} \\
  };

  \draw[cof]  (b) to (s);
  \draw[tfib] (F) to (0);

  \draw[->] (b) to (F);
  \draw[->] (F) to (X);
  \draw[->] (s) to (0);

  \draw[->] (X) to node[right] {$p$}   (Y);
  \draw[->] (0) to node[below] {$v_0$} (Y);
  \end{tikzeq*}
  The left square of this diagram has a lift since $F_{v_0}$ is contractible.
  (Since $\simp{0}$ and $F_{v_0}$ are Kan complexes,
  the acyclic fibration $F_{v_0} \to \simp{0}$ is trivial
  by \cref{acyclic-fibration-trivial-Kan}.)
  Let $\tilde w \from \simp{m} \to X$ denote the resulting composite, which
  leads to yet another lifting problem
  \begin{tikzeq*}
  \matrix[diagram,column sep={11em,between origins}]
  {
    |(b)| \simp{m} \times \{ 0 \} \union \bdsimp{m} \times \simp{1} &
    |(X)| X                 \\
    |(s)| \simp{m} \times \simp{1} &
    |(Y)| Y \rlap{\text{.}} \\
  };

  \draw[ano] (b) to (s);

  \draw[fib] (X) to node[right] {$p$}                    (Y);
  \draw[->]  (b) to node[above] {$[\tilde w, \tilde G]$} (X);
  \draw[->]  (s) to node[below] {$v H$}                  (Y);
  \end{tikzeq*}
  Again, the left map is a trivial cofibration
  by \cref{ano-cof-pushout-product} of \cref{second-pushout-product}, so
  there is a solution $G \from \simp{m} \times \simp{1} \to X$, which is
  a homotopy from $\tilde w$ to $w$.
  That $w$ is a solution to the original lifting problem.
  Indeed, $p w = p G \iota_1 = v H \iota_1 = v$ and
  $w | \bdsimp{m} = G \iota_1 | \bdsimp{m} = \tilde G \iota_1 = u$.
\end{proof}

\begin{proposition}\label{acyclic-fibration-trivial}
  A simplicial map is an acyclic Kan fibration \iff{}
  it is a trivial fibration.
\end{proposition}

\begin{proof}
  This follows by the same argument as \cref{acyclic-fibration-trivial-Kan}
  except it uses \cref{acyclic-fibration-contractible-fibers,contractible-fibers-trivial-fibration}
  in the place of \cref{fiberwise-acyclic-fibration-trivial-Kan}
  and \cref{whe-arbitrary-properties} in the place of \cref{whe-Kan-properties}.
\end{proof}

\begin{proposition}\label{acyclic-cofibration-anodyne}
  A simplicial map is an acyclic cofibration \iff{} it is a trivial cofibration.
\end{proposition}

\begin{proof}
  An acyclic cofibration is a cofibration by \cref{triv-cof-cof}.

  First, we check the acyclicity in the subcategory of cofibrant objects.
  Let $f \from X \ato Y$ be a trivial cofibration between cofibrant simplicial sets.
  For any Kan complex $K$, the map $f^* \from K^Y \to K^X$ is a trivial fibration between Kan complexes
  by \cref{ano-fib-pullback-hom,cof-fib-pullback-hom}
  of \cref{{second-pushout-product}}.
  Thus, $f^*$ is a \whe{} in the sense of \ref{whe-Kan-def}
  by \cref{whe-Kan-trivial-fibration} of \cref{whe-Kan-properties},
  and it follows that $f$ is a \whe{} in the sense of \ref{whe-cof-def}.

  In the general case, note that
  the class of cofibrations $i$ \st{} $T i$ is a \whe{} is closed under
  coproducts, pushouts, sequential colimits and retracts
  by \cref{cofibration-category-cofibrant} and
  \cref{whe-retract} of \cref{whe-arbitrary-properties} and since
  $T$ preserves colimits and cofibrations and takes values in cofibrant objects
  (\cref{cofibrant-replacement}).
  By the preceding argument, this class also contains the horn inclusions and
  hence all trivial cofibrations by \cref{thm:wfs-via-soa:fib} of \cref{thm:wfs-via-soa}.

  Conversely, let $f$ be an acyclic cofibration and
  pick a factorisation
  \begin{tikzeq*}
  \matrix[diagram]
  {
    |(X)| X &               & |(Y)| Y \\
            & |(hY)| \hat Y &         \\
  };

  \draw[cof] (X) to node[above] {$f$} node[below] {$\weq$} (Y);

  \draw[ano] (X)  to node[below left]  {$i$} (hY);
  \draw[fib] (hY) to node[below right] {$p$} (Y);
  \end{tikzeq*}
  where $i$ is a trivial cofibration and $p$ is a Kan fibration.
  Then $i$ is a \whe{} by the argument above and thus so is $p$ by 2-out-of-3.
  Hence, $p$ is a trivial fibration by \cref{acyclic-fibration-trivial}.
  This implies that the lifting problem
  \begin{tikzeq*}
  \matrix[diagram]
  {
    |(X)|  X & |(hY)| \hat Y \\
    |(Yl)| Y & |(Yr)|      Y \\
  };

  \draw[cof]  (X)  to node[left]  {$f$}     (Yl);
  \draw[tfib] (hY) to node[right] {$p$}     (Yr);
  \draw[->]   (X)  to node[above] {$i$}     (hY);
  \draw[->]   (Yl) to node[below] {$\id_Y$} (Yr);
  \end{tikzeq*}
  has a solution, which exhibits $f$ as a retract of $i$.
  The former is thus a trivial cofibration.
\end{proof}

\begin{theorem} \label{thm:main-thm-first-proof}
  The cofibrations, Kan fibrations, and \whe{}s form
  a model structure on simplicial sets.
\end{theorem}

\begin{proof}
  \Whe{}s satisfy 2-out-of-6
  by \cref{whe-2-out-of-6} of \cref{whe-arbitrary-properties}.
  By \cref{thm:wfs-via-soa}, we have weak factorisation systems of cofibrations and trivial fibrations and of
  trivial cofibrations and Kan fibrations.
  By \cref{acyclic-fibration-trivial},
  acyclic fibrations coincide with trivial fibrations.
  By \cref{acyclic-cofibration-anodyne}, trivial cofibrations coincide with acyclic cofibrations.
\end{proof}

\begin{remark}\label{thm:first-diff-henry}
  Let us briefly highlight the differences between our first proof and
  Simon Henry's proof in~\cite{Henry-qms}.
  First of all, the definitions of the weak equivalences are different.
  Indeed, Henry first introduces weak equivalences only between objects that
  are either fibrant or cofibrant using the homotopy category,
  as in his work on weak model categories~\cite{Henry-wms}, and then
  extends this definition to all objects
  using the adjunction to the category of semisimplicial sets and
  a weak model structure on it.
  Secondly, the two proofs are organised in very different ways.
  In particular, Henry's proof first establishes the existence of
  a model structure with the required weak equivalences and cofibrations
  (\cite{Henry-qms}*{Theorem~2.2.8}) and then exploits $\Ex^\infty$ to show that
  the fibrations of the model structure are the Kan fibrations.
  This second step is particularly complex since it relies on an intricate auxiliary notion of a P-structure introduced in \cite{Moss}.
  Moreover, it depends on the theory of degeneracy quotients (called collapses by Joyal~\cite{joyal-collapses}), which we replace by the closely related, but less technical, theory of generalised degeneracies.
  (Every degeneracy quotient between finite colimits of simplices is a generalised degeneracy.)
\end{remark}

  \section{The model structure via the equivalence extension property}
  \label{sec:second-proof}

In this section, we present our second proof of the existence of the constructive Kan--Quillen model structure.
It is based on the development in~\cites{Gambino-Sattler,Sattler}.
The strategy of the approach is as follows.
We first build the restriction of the model structure to cofibrant objects.
For this, we follow the approach of~\cite{Sattler} in the setting of cofibrant simplicial sets.
We then obtain the extension to simplicial sets using formal reasoning.

The development of the model structure in cofibrant simplicial sets occupies most of this section.
The cornerstones are the Frobenius property (\cref{frobenius}) in \cref{sec:frobenius-property} and the trivial fibration extension (\cref{trivial-fibration-extension}) and fibration extension (\cref{fibration-extension}) properties in \cref{sec:equivalence-extension-property}.
The latter two statements are consequences of the equivalence extension property (\cref{equivalence-extension}), which is the namesake of \cref{sec:equivalence-extension-property}.
(Although we refer to all these statements as properties, we recall that constructively we ought to think of them as operations.)
In \cref{sec:model-structure-second-proof}, we define notions of weak homotopy equivalences and verify the model structure properties, first in cofibrant objects and then in all of simplicial sets.
This subsection is entirely formal and uses from the earlier two subsections only the three statements referred to above.

\subsection{The Frobenius property}
\label{sec:frobenius-property}

The first step is to prove a restricted version of the Frobenius property for the \wfs{} of trivial cofibrations and fibrations, asserting that pullback along a fibration with cofibrant domain preserves trivial cofibrations.
For this, we follow~\cite{Gambino-Sattler}, but make explicit the cofibrancy assumption on the domain of the fibration (\cf the proof of \cref{she-pullback-square}).

We begin by recalling from~\cite{Gambino-Sattler} the notion of a strong homotopy equivalence in simplicial sets.
Let $f \from A \to B$ be a homotopy equivalence with homotopy inverse $g \from B \to A$.
We say that $f$ is \emph{$0$-oriented} if $g f \htp \id_A$ and $f g \htp \id_B$ are witnessed by a homotopy $u$ from $g f$ to $\id_A$ and a homotopy $v$ from $f g$ to $\id_B$, respectively.
If the homotopies go in the reverse direction, we say that $f$ is \emph{$1$-oriented}.
In either case, we say that $f$ is a \emph{strong homotopy equivalence} if $f u = v f$.
Note that if $u$ is the constant homotopy, $f$ specialises to a map with a strong deformation retraction as in \cref{sec:htpy-and-whe}.
Dually, if $v$ is the constant homotopy, $f$ specialises to a shrinkable map in the sense of \cref{sec:cofibrations} (\cf \cref{she-for-cylinders}).

A key property of the notion of strong homotopy equivalence is that it admits a characterisation in terms of retracts.
For this, we consider the following square in simplicial sets:
\begin{tikzeq*}
\matrix[diagram]
{
  |(e)| \emptyset  & |(0)| \braces{0}               \\
  |(1)| \braces{1} & |(s)| \simp{1} \rlap{\text{.}} \\
};

\draw[->] (e) to node[above] {$!$}        (0);
\draw[->] (e) to node[left]  {$!$}        (1);
\draw[->] (0) to node[right] {$\iota_0$} (s);
\draw[->] (1) to node[below] {$\iota_1$} (s);
\end{tikzeq*}
In the arrow category $\sSet^{[1]}$, this square induces a map $\theta_0 \from ! \to \iota_0$ when read horizontally and a map $\theta_1 \from ! \to \iota_1$ when read vertically.
Note that $!$ is the unit for the pushout product.
We then have the following characterisation.

\begin{lemma}[{\cite{Gambino-Sattler}*{Lemma~4.3}}] \label{she-as-retract}
The following are equivalent for a map $f$ and $k \in \braces{0, 1}$:
\begin{conditions}
\item \label{she-as-retract:def}
$f$ is a $k$-oriented strong homotopy equivalence,
\item \label{she-as-retract:retraction}
$f \hatbin{\times} \theta_k \from f \to f \hatbin{\times} \iota_k$ has a retraction,
\item \label{she-as-retract:section}
$\hat{\hom}(\theta_k, f) \from \hat{\hom}(\iota_k, f) \to f$ has a section.
\end{conditions}
\end{lemma}

\begin{proof}
The equivalence of~\ref{she-as-retract:retraction} and~\ref{she-as-retract:section} follows by adjointness.
For the equivalence of~\ref{she-as-retract:def} and~\ref{she-as-retract:retraction}, we only consider the case $k = 0$ (the case $k = 1$ is dual).
The data of a retraction of $f \hatbin{\times} \theta_0 \from f \to f \hatbin{\times} \iota_0$ is shown in the following diagram:
  \begin{tikzeq*}
  \matrix[diagram,column sep={10em,between origins}]
  {
    |(Al)| A & |(pp)| A \times \simp{1} \coprod_{A \times \braces{0}} B \times \braces{0} & |(Ar)| A \\
    |(Bl)| B & |(p)|  B \times \simp{1}                                                   & |(Br)| B \\
  };

  \draw[->] (Al) to node[left]  {$f$}                         (Bl);
  \draw[->] (pp) to node[left]  {$f \hatbin{\times} \iota_0$} (p);
  \draw[->] (Ar) to node[right] {$f$}                         (Br);

  \draw[->] (Al) to (pp);

  \draw[->] (Bl) to node[below] {$B \times \iota_1$} (p);

  \draw[->,dashed] (pp) to node[above] {$[g, u]$}           (Ar);
  \draw[->,dashed] (p) to node[below] {$v$} (Br);
  \end{tikzeq*}
(we have omitted the identity arrows for $A$ and $B$; the top left map is induced by $A \times \iota_1$).
Commutativity of the diagram precisely amounts to $g$, $u$, $v$ making $f$ into a $0$-oriented strong homotopy equivalence.
\end{proof}

\begin{remark} \label{she-for-cylinders}
The general setting for the theory of strong homotopy equivalences is that of an (adjoint) functorial cylinder.
One may then take the $0$-oriented notion as default and obtain the $1$-oriented notion from the cylinder with mirrored endpoints.
Indeed, this viewpoint allows us to support zig-zags of homotopies (as occur in our notions of homotopy equivalence and shrinkable map).
We merely have to consider the cylinder functor induced by a walking zig-zag $J$, seen as an interval object.
\end{remark}

\begin{lemma} \label{horn-inclusion-is-she}
The horn inclusion $\horn{n,k} \to \simp{n}$ is a 0-oriented strong homotopy equivalence for $k < n$ and a 1-oriented strong homotopy equivalence for $k > 0$.
In particular, every horn inclusion is a strong homotopy equivalence.
\end{lemma}

This is \cite{gz}*{IV.2.1.3} (describing strong homotopy equivalences via \cref{she-as-retract:retraction} of \cref{she-as-retract}).
For self-containedness, we include a proof.

\begin{proof}
We let $k < n$ and show that $h \from \horn{n,k} \to \simp{n}$ is a 0-oriented strong homotopy equivalence (the other case is dual).
The initial segment $s \from [k] \to [n]$ has a unique retraction $r \from [n] \to [k]$.
Because $k < n$, the map $s$ induces a monomorphism $i \from \simp{k} \to \horn{n,k}$.
We define the homotopy inverse of $h$ as $i r \from \simp{n} \to \horn{n,k}$.
The homotopy from $h (i r)$ to $\id_{\simp{n}}$ is given by the nerve of the map $[1] \times [n] \to [n]$ corresponding to the unique map $s r \to \id_{[n]}$.
This restricts on $h$ to a homotopy from $\id_{\horn{n,k}}$ to $(i r) h$.
\end{proof}

\begin{corollary} \label{triv-cof-generated-by-she}
The \wfs{} of trivial cofibrations and fibrations is cofibrantly generated by cofibrations between cofibrant objects that are strong homotopy equivalences.
\end{corollary}

\begin{proof}
Every horn inclusion is a cofibration between cofibrant objects by \cref{triv-cof-cof,horn-cofibrant} and a strong homotopy equivalence by \cref{horn-inclusion-is-she}.
In the other direction, every cofibration $i$ that is a $k$-oriented strong homotopy equivalence is a retract of $\iota_k \hatbin{\times} f$ by \cref{she-as-retract}, which is a trivial cofibration by \cref{ano-cof-pushout-product} of \cref{second-pushout-product}.
\end{proof}

\begin{lemma}[cf.~\cite{Gambino-Sattler}*{Lemma~4.7}] \label{she-pullback-square}
In any pullback square
\begin{tikzeq*}
\matrix[diagram]
{
  |(A)| A & |(B)| B \\
  |(X)| X & |(Y)| Y \\
};

\draw[->] (A) to node[left]  {$f$} (X);
\draw[->] (B) to node[right] {$g$} (Y);

\draw[->]  (A) to (B);
\draw[fib] (X) to (Y);

\pb{A}{Y};
\end{tikzeq*}
with $X$ cofibrant, if $g$ is a $k$-oriented strong homotopy equivalence, then so is $f$.
\end{lemma}

\begin{proof}
We only consider the case $k = 0$ (the case $k = 1$ is dual).
We work with $0$-oriented strong homotopy equivalences as characterised by \cref{she-as-retract:retraction} of \cref{she-as-retract}.
To make $f$ into one, we will construct the dotted map in the following (vertical) map of retracts in the arrow category:
\begin{tikzeq*}
\matrix[diagram,column sep={6em,between origins}]
{
  |(fl)| f & |(df)| f \hatbin{\times} \iota_0 & |(fr)| f                 \\
  |(gl)| g & |(dg)| g \hatbin{\times} \iota_0 & |(gr)| g \rlap{\text{.}} \\
};

\draw[->] (fl) to (gl);
\draw[->] (df) to (dg);
\draw[->] (fr) to (gr);

\draw[->] (fl) to node[above] {$f \hatbin{\times} \theta_0$} (df);
\draw[->] (gl) to node[below] {$g \hatbin{\times} \theta_0$} (dg);

\draw[->,dashed] (df) to (fr);
\draw[->] (dg) to (gr);
\end{tikzeq*}
Since $f \to g$ is a pullback square, it suffices to produce this map on codomains:
\begin{tikzeq*}
\matrix[diagram,column sep={6em,between origins}]
{
  |(Xl)| X & |(sX)| X \times \simp{1} & |(Xr)| X                 \\
  |(Yl)| Y & |(sY)| Y \times \simp{1} & |(Yr)| Y \rlap{\text{.}} \\
};

\draw[->] (Xl) to (Yl);
\draw[->] (sX) to (sY);
\draw[->] (Xr) to (Yr);

\draw[->] (Xl) to node[above] {$X \times \iota_1$} (sX);
\draw[->] (Yl) to node[below] {$Y \times \iota_1$} (sY);

\draw[->,dashed] (sX) to (Xr);

\draw[->] (sY) to (Yr);
\end{tikzeq*}
This is a lifting problem
\begin{tikzeq*}
\matrix[diagram,column sep={6em,between origins}]
{
  |(Xl)| X                 & |(Xr)| X                \\
  |(sX)| X \times \simp{1} & |(Y)| Y \rlap{\text{.}} \\
};

\draw[->]  (Xl) to node[above] {$\id_X$}            (Xr);
\draw[ano] (Xl) to node[left]  {$X \times \iota_1$} (sX);

\draw[->]  (sX) to (Y);
\draw[fib] (Xr) to (Y);

\draw[->,dashed] (sX) to (Xr);
\end{tikzeq*}
Here, the left map is the pushout product of $\emptyset \to X$ with $\braces{1} \to \simp{1}$, hence is a trivial cofibration by \cref{ano-cof-pushout-product} of \cref{second-pushout-product} since $X$ is cofibrant.
\end{proof}

\begin{proposition}[restricted Frobenius property, cf.~\cite{Gambino-Sattler}*{Theorem~3.8}] \label{frobenius}
Let $X \fto Y$ be a fibration with $X$ cofibrant.
Pullback along $X \fto Y$ preserves trivial cofibrations.
\end{proposition}

\begin{proof}
Since pullback along $X \to Y$ is a left adjoint, it suffices to show that it preserves the generators of \cref{triv-cof-generated-by-she}, \ie, cofibrations between cofibrant objects that are strong homotopy equivalences.
For this, we note the following.
\begin{itemize}
\item
Pullback along $X \to Y$ preserves cofibrations (and hence cofibrant objects) by \cref{pullback-of-cofibrations}.
\item
Let $A \to B$ be a map over $Y$ with $B$ cofibrant such that $A \to B$ is a strong homotopy equivalence.
Then its pullback $X \pull_Y A \to X \pull_Y B$ is a strong homotopy equivalence.
This follows from \cref{she-pullback-square} since $X \pull_Y B \to B$ is a fibration (as a pullback of $X \to Y$) with cofibrant domain.
\qedhere
\end{itemize}
\end{proof}

\begin{remark}
Let us explain the name of \cref{frobenius}.
The \emph{Frobenius property}~\cite{van-den-Berg-Garner} of a \wfs{} in a category with finite limits refers to the condition that left maps are closed under pullback along right maps.
In our setting, we do not quite obtain the Frobenius property of the \wfs{} of trivial cofibrations and fibrations, but only a restricted version where the source of the map we pull back along is cofibrant.
This is because cofibrations are not generally closed under pullback.

\Cref{frobenius} in particular encompassed the Frobenius property in cofibrant simplicial sets.
This is what we will use to verify the model structure properties in cofibrant simplicial sets.
However, \Cref{frobenius} is more general: the target of the fibration does not need to be cofibrant.
This will be used to extend the model structure to the entirety of simplicial sets.
\end{remark}

\subsection{The equivalence extension property}
\label{sec:equivalence-extension-property}

In this subsection, we prove the equivalence extension property and derive its corollaries, the (trivial) fibration extension properties.
All of this happens entirely in the cofibrant fragment of simplicial sets.
Before we delve into the proof, we record some basic facts about mapping path space factorisations.

The context of the following definition and lemmas is the slice of cofibrant simplicial sets over an object $M$.
The \emph{mapping path space factorisation} of a map $X \to Y$ is
\begin{tikzeq}{mapping-path-space}
  \matrix[diagram,column sep={4em}]
  {
    |(X)| X & |(M)| X \pull_{Y} \simp{1} \cotensor Y & |(Y)| Y \rlap{\text{.}} \\
  };

  \draw[->] (X) to (M);
  \draw[->] (M) to (Y);
\end{tikzeq}
We adopt the convention that the position of the pullback symbol with respect to $\simp{1} \cotensor Y$ indicates whether the pullback is taken with respect to the left or right endpoint projection.
The first factor is induced by the constant map $Y \to \simp{1} \cotensor Y$ and the second factor is induced by the right endpoint projection.
The middle object is cofibrant by \cref{cofibrant-finite-limits,exponential-finite-cofibrant}.

The following lemma already features implicitly in the proof of \cref{fibration-category-slice-cofibrant}.

\begin{lemma} \label{mapping-path-space-fibration}
  If $X$ and $Y$ are fibrant in~\eqref{mapping-path-space}, then the second factor in~\eqref{mapping-path-space} is a fibration.
\end{lemma}

\begin{proof}
  The second factor decomposes as a pullback of the pullback cotensor of $Y \fto 1$ with $\bd \simp{1} \cto \simp{1}$ followed by a pullback of $X \fto 1$.
  So it is a fibration by \cref{cof-fib-pullback-hom} of \cref{second-pushout-product}.
\end{proof}

\begin{lemma} \label{mapping-path-space-equivalence}
  Assume that $Y$ is fibrant in~\eqref{mapping-path-space}.
  Then the following are equivalent:
  \begin{conditions}
    \item \label{mapping-path-space-equivalence:fib-h-equiv}
    $X$ is fibrant and $X \to Y$ is a fiberwise homotopy equivalence,
    \item \label{mapping-path-space-equivalence:triv-fib}
    the second factor in~\eqref{mapping-path-space} is a trivial fibration.
  \end{conditions}
\end{lemma}

\begin{proof}
  First note that the first factor in~\eqref{mapping-path-space} admits a retraction (given by the first projection).
  Thus, $X$ is a retract of the middle object.
  If the second factor is a fibration, then the middle object is fibrant, and hence so is $X$.
  This shows that~\cref{mapping-path-space-equivalence:triv-fib} makes $X$ fibrant.

  Now assume that $X$ is fibrant.
  Using the fibration category of \cref{fibration-category-slice-cofibrant}, it remains to show that $X \to Y$ is a fiberwise homotopy equivalence if and only if the second factor in~\eqref{mapping-path-space} is a fiberwise homotopy equivalence.
  This holds by 2-out-of-3 after we check that the first factor in~\eqref{mapping-path-space} is a fiberwise homotopy equivalence.
  Indeed, that map is a retraction of a pullback of the first endpoint projection $\simp{1} \cotensor Y \to Y$, a trivial fibration by \cref{ano-fib-pullback-cotensor} of \cref{cof-tensor-cotensor}.
\end{proof}

We are now ready to prove the equivalence extension property.
While in~\cites{Kapulkin,Gambino-Henry} this property is proved in context of a model structure to establish univalence of a classifying fibration, it was observed in~\cite{Sattler} that it can reversely be used to establish that very model structure.
Indeed, the statement of the equivalence extension property does not refer to the weak equivalences of a model structure, but the more elementary notion of fiberwise homotopy equivalence.
Thus, the natural setting for its direct proof is the fibration category of Kan fibrations over a base in cofibrant simplicial sets established in \cref{fibration-category-slice-cofibrant}.
The idea of a proof using elementary means, not making use of an ambient model structure, goes back to~\cite{CCHM}.

\begin{proposition}[Equivalence extension, cf.~\cite{Sattler}*{Proposition~5.1}] \label{equivalence-extension}
  In cofibrant simplicial sets, consider the solid part of the diagram
  \begin{tikzeq}{equivalence-extension:0}
    \matrix[diagram]
    {
      |(X0)| X_0 &            & |(Y0)| Y_0 &            \\
                 & |(X1)| X_1 &            & |(Y1)| Y_1 \\
      |(A)|  A   &            & |(B)|  B   &            \\
    };

    \draw[->]        (X0) to node[above right] {$\heq$} (X1);
    \draw[->,dashed] (Y0) to node[above right] {$\heq$} (Y1);
    \draw[cof]       (A)  to node[below]       {$i$}    (B);

    \draw[fib] (X0) to (A);
    \draw[fib] (X1) to (A);

    \draw[fib,dashed] (Y0) to (B);
    \draw[fib]        (Y1) to (B);

    \draw[->,dashed] (X0) to (Y0);
    \draw[->,over]   (X1) to (Y1);
  \end{tikzeq}
  where $A \to B$ is a cofibration, $X_0 \to A$ and $Y_1 \to B$ are fibrations,
  the lower square is a pullback, and
  the map $X_0 \to X_1$ is a \fhe{} over $A$.
  Then there is $Y_0$ fitting into the diagram as indicated \st{}
  the back square is a pullback, the map $Y_0 \to B$ is a fibration and
  the map $Y_0 \to Y_1$ is a \fhe{} over $B$.
\end{proposition}

\begin{proof}
  In the following, we will make use of the adjunction $i^* \adj \Pi_i$.
  We identify $X_1$ with $i^* Y_1$ over $A$.
  Since $i$ is a monomorphism, the unit of the adjunction $\Sigma_i \adj i^*$ is invertible.
  By adjointness, the counit of the adjunction $i^* \adj \Pi_i$ is invertible.
  That is, $\Pi_i$ is a reflective embedding with reflector $i^*$.
  Note that $\Sigma_i i^*$ is isomorphic to the functor of product with $A$ over $B$.
  By adjointness, this means $\Pi_i i^* \simeq \exp_B(A, -)$.
  By \cref{pullback-of-cofibrations}, the functor $i^*$ preserves cofibrations.
  By adjointness, the functor $\Pi_i$ preserves trivial fibrations.
  In the following, we use freely that $\Pi_i$ preserves cofibrant objects by \cref{cofibrant-pushforward} and that cofibrant objects are closed under finite limits (\cref{cofibrant-finite-limits}).
  With this, all our constructions remain within cofibrant objects.

  In the slice over $B$, we define $Y_0$ and its map to $Y_1$ via the pullback
  \begin{tikzeq}{equivalence-extension:def_Y0}
    \matrix[diagram,column sep={6em,between origins}]
    {
      |(Y0)| Y_0 & |(X0)| \Pi_i X_0     \\
      |(Y1)| Y_1 & |(i)|  \Pi_i i^* Y_1 \\
    };

    \draw[->] (Y0) to (X0);
    \draw[->] (Y1) to (i);
    \draw[->] (Y0) to (Y1);
    \draw[->] (X0) to (i);
    \pb{Y0}{i};
  \end{tikzeq}
  where the right map is the image of $X_0 \to X_1$ under $\Pi_i$ (identifying $X_1$ with $i^* Y_1$ over $A$) and the bottom map is a unit map of $i^* \adj \Pi_i$.
  Since $\Pi_i$ is a reflective embedding, the bottom map in~\eqref{equivalence-extension:def_Y0} becomes invertible after pulling back along $i \from A \to B$:
  \begin{tikzeq}{equivalence-extension:Y0_fits}
    \matrix[diagram,column sep={6em,between origins}]
    {
      |(Y0)| i^* Y_0 & |(X0)| i^* \Pi_i X_0                            \\
      |(Y1)| i^* Y_1 & |(i)|  i^* \Pi_i i^* Y_1 \rlap{\text{.}} \\
    };

    \draw[->] (Y0) to node[above] {$\simeq$} (X0);
    \draw[->] (Y1) to node[above] {$\simeq$} (i);
    \draw[->] (Y0) to (Y1);
    \draw[->] (X0) to (i);
    \pb{Y0}{i};
  \end{tikzeq}
  Furthermore, the right map above is isomorphic to $X_0 \to i^* Y_1$, \ie $X_0 \to X_1$.
  This shows that $Y_0$ over $Y_1$ pulls back along $i$ to $X_0$ over $X_1$ in~\eqref{equivalence-extension:0} as required.

  For the rest of the proof, we work in the slice over $B$ unless stated otherwise.
  It remains to show that $Y_0$ is fibrant and $Y_0 \to Y_1$ is a fiberwise homotopy equivalence.
  For this, we consider the mapping path space factorisation of $Y_0 \to Y_1$:
  \begin{tikzeq}{equivalence-extension:mapping-path-space}
    \matrix[diagram,column sep={4em}]
    {
      |(Y0)| Y_0 & |(M)| Y_0 \pull_{Y_1} \simp{1} \cotensor Y_1 & |(Y1)| Y_1 \rlap{\text{.}} \\
    };

    \draw[->] (Y0) to (M);
    \draw[->] (M) to (Y1);
  \end{tikzeq}
  Using the reverse direction of \cref{mapping-path-space-equivalence}, it remains to show that its second factor is a trivial fibration.

  By pullback pasting with~\eqref{equivalence-extension:def_Y0}, the middle object in~\eqref{equivalence-extension:mapping-path-space} is isomorphic to $\Pi_i X_0 \pull_{\Pi_i i^* Y_1} \simp{1} \cotensor Y_1$.
  Under this isomorphism, the right map in~\eqref{equivalence-extension:mapping-path-space} corresponds to the top map in the below diagram, induced by the right endpoint projection:
  \begin{tikzeq}{equivalence-extension:factorisation}
    \matrix[diagram,column sep={8em,between origins}]
    {
      |(M)| \Pi_i X_0 \pull_{\Pi_i i^* Y_1} \simp{1} \cotensor Y_1
       &  &
      |(Y1)| Y_1 \rlap{\text{.}}
      \\&
      |(N)| \Pi_i X_0 \pull_{\Pi_i i^* Y_1} \Pi_i i^* (\simp{1} \cotensor Y_1) \pull_{\Pi_i i^* Y_1} Y_1
      \rlap{\text{.}} \\
    };

    \draw[->] (M) to (Y1);
    \draw[->] (M) to (N);
    \draw[->] (N) to (Y1);
  \end{tikzeq}
  To show that it is a trivial fibration, we produce a factorisation into two trivial fibrations as indicated above.

  Associating the pullbacks in the middle object of~\eqref{equivalence-extension:factorisation} to the right, the first factor in~\eqref{equivalence-extension:factorisation} is the pullback along $\Pi_i X_0 \to \Pi_i i^* Y_1$ of the map
  \begin{tikzeq*}{equivalence-extension:first-factor}
    \matrix[diagram,column sep=4em]
    {
      |(M')| \simp{1} \cotensor Y_1 & |(N')| \Pi_i i^* (\simp{1} \cotensor Y_1) \pull_{\Pi_i i^* Y_1} Y_1 \rlap{\text{.}} \\
    };

    \draw[->] (M') to (N');
  \end{tikzeq*}
  It suffices to show this is a trivial fibration.
  Rewriting using $\Pi_i i^* \simeq \exp_B(A, -)$, this is the pullback exponential (over $B$) of the right endpoint projection $\simp{1} \cotensor Y_1 \to Y_1$ with $A \to B$.
  Since $Y_1$ is fibrant, the right endpoint projection $\simp{1} \cotensor Y_1 \to Y_1$ is a trivial fibration by \cref{ano-fib-pullback-hom} of \cref{second-pushout-product}.
  Since $A \to B$ is a cofibration, the pullback exponential is a trivial fibration by \cref{cof-tcof-pullback-hom} of \cref{first-pushout-product}.

  Associating the pullbacks in the middle object of~\eqref{equivalence-extension:factorisation} to the left, the second factor in~\eqref{equivalence-extension:factorisation} is the pullback along $Y_1 \to \Pi_i i^* Y_1$ of the following map, again induced by the right endpoint projection:
  \begin{tikzeq*}{equivalence-extension:second-factor}
    \matrix[diagram,column sep=4em]
    {
      |(N'')| \Pi_i X_0 \pull_{\Pi_i i^* Y_1} \Pi_i i^* (\simp{1} \cotensor Y_1) & |(Y1'')| \Pi_i i^* Y_1 \rlap{\text{.}} \\
    };

    \draw[->] (N'') to (Y1'');
  \end{tikzeq*}
  Since the right adjoint $\Pi_i$ preserves pullbacks, this is the image of the map $X_0 \pull_{i^* Y_1} i^* (\simp{1} \cotensor Y_1) \to i^* Y_1$ under $\Pi_i$.
  And since $\Pi_i$ preserves trivial fibrations, it suffices to show this is a trivial fibration.
  Since $i^*$ preserves cotensors and using the identification $X_1 \simeq i^* Y_1$, this is $X_0 \pull_{X_1} \simp{1} \cotensor X_1 \to X_1$, the second factor in the mapping path space factorisation of $X_0 \to X_1$.
  It is a trivial fibration using the forward direction of \cref{mapping-path-space-equivalence}.
\end{proof}

An important example of a fiberwise homotopy equivalence arises in cofibrant simplicial sets from a fibration with target $A \times \simp{1}$.
Note that $A \times \simp{1}$ is cofibrant if $A$ is cofibrant by \cref{simplex-cofibrant,cofibrant-finite-limits}.
For the below statement, recall that $(A \times \iota_k)^*$ denotes pullback along the inclusion $\iota_k \from A \to A \times \simp{1}$ for $k = 0, 1$.

\begin{lemma} \label{weak-equivalence-from-fibration-over-line}
  Let $p \from X \fto A \times \simp{1}$ be a fibration with $A$ and $X$ cofibrant.
  There is a fiberwise homotopy equivalence over $A$ (in either direction) between $(A \times \iota_0)^* X$ and $(A \times \iota_1)^* X$.
\end{lemma}

\begin{proof}
  We take the pullback
  \begin{tikzeq*}{fibration-extension-square}
    \matrix[diagram,column sep={8em,between origins}]
    {
      |(P)| P & |(X')| \exp(\simp{1}, X) \\
      |(A)| A & |(A')| \exp(\simp{1}, A \times \simp{1}) \\
    };

    \draw[->]  (P) to (X');
    \draw[->]  (A) to node[above] {$\eta_A$} (A');
    \draw[fib] (P) to (A);
    \draw[fib] (X') to node[right] {$\exp(\simp{1}, p)$} (A');
    \pb{P}{A'};
  \end{tikzeq*}
  where the bottom map is a unit map of the evident adjunction.
  We record that $P$ is cofibrant by \cref{cofibrant-finite-limits,exponential-finite-cofibrant}, and that $P \fto A$ is a fibration by \cref{second-pushout-product}.

  We will argue that there are trivial fibrations from $P$ to $(A \times \iota_0)^* X$ and $(A \times \iota_1)^* X$ over $A$.
  These trivial fibrations are fiberwise homotopy equivalences over $A$ by \cref{fiberwise-acyclic-fibration-trivial-Kan}.
  Reversing and composing them in the fibration category of \cref{fibration-category-slice-cofibrant} as needed gives the claim.

  We only construct the trivial fibration from $P$ to $(A \times \iota_0)^* X$ (the other case is dual).
  By pullback pasting, we can see $(A \times \iota_0)^* X$ as being obtained by two successive pullbacks:
  \begin{tikzeq*}{fibration-extension-square}
    \matrix[diagram,column sep={7em,between origins}]
    {
      |(X0)| (A \times \iota_0)^* X &          & |(X0')| X \times_{A \times \simp{1}} \exp(\simp{1}, A \times \simp{1}) &           & |(X)| X                                                          \\
      |(A)| A                       & |(X0op)| & |(A')| \exp(\simp{1}, A \times \simp{1})                               & |(X0'op)| & |(A'')| A \times \simp{1} \rlap{\text{.}} \\
    };

    \draw[->]  (X0) to (X0');
    \draw[->]  (A) to node[above] {$\eta_A$} (A');
    \draw[->]  (X0') to (X);
    \draw[->]  (A') to node[above] {$\exp(\iota_0, A \times \simp{1})$} (A'');
    \draw[fib] (X0) to (A);
    \draw[fib] (X0') to (A');
    \draw[fib] (X) to node[right] {$p$} (A'');
    \pb{X0}{X0op};
    \pb{X0'}{X0'op};
  \end{tikzeq*}
  Taking the pullback exponential of $p$ with $\iota_0 \from \braces{0} \to \simp{1}$, we obtain
  \[
    \exp(\simp{1}, X) \to X \times_{A \times \simp{1}} \exp(\simp{1}, A \times \simp{1})
  \]
  over $\exp(\simp{1}, X)$.
  This is a trivial fibration by \cref{second-pushout-product}.
  Pulling back along $\eta_A$ yields the desired trivial fibration.
\end{proof}

In cofibrant simplicial sets, we say that a (trivial) fibration $X \to A$ \emph{extends} along a cofibration $A \to B$ if we can find a (trivial) fibration $Y \to B$ that restricts to $X \to A$ along the given cofibration:
\begin{tikzeq}{fibration-extension-square}
  \matrix[diagram]
  {
    |(X)| X & |(Y)| Y                         \\
    |(A)| A & |(B)| B \rlap{\text{.}} \\
  };

  \draw[->,dashed]  (X) to (Y);
  \draw[cof]        (A) to (B);
  \draw[fib]        (X) to (A);
  \draw[fib,dashed] (Y) to (B);
  \pb{X}{B};
\end{tikzeq}
Our ultimate goal in this subsection is to show that fibrations extend along trivial cofibrations.

\begin{corollary}[Trivial fibration extension] \label{trivial-fibration-extension}
  In cofibrant simplicial sets, trivial fibrations extend along cofibrations.
\end{corollary}

\begin{proof}
  This is the special case of \cref{equivalence-extension} where the maps $X_1 \to A$ and $Y_1 \to B$ are identities.
  To see this, we recall that a map $X \to Y$ is a trivial fibration exactly if it is a fibration and fiberwise homotopy equivalence over $Y$.
  This is recorded in \cref{fibration-category-slice-cofibrant} (with $M$ instantiated to $Y$).
\end{proof}

\begin{lemma} \label{fibration-generating-extension}
In cofibrant simplicial sets, fibrations extend along cofibrations with a retraction up to homotopy.
In particular, they extend along cofibrations that are strong homotopy equivalences.
\end{lemma}

\begin{proof}
Let $i \from A \to B$ be a cofibration with $r \from B \to A$ and a homotopy relating $r i$ and $\id_A$ (in either direction).
We will solve an extension problem
\begin{tikzeq}{fibration-generating-extension:0}
\matrix[diagram,column sep={6em,between origins}]
{
  |(X)| X & |(Y)| Y                 \\
  |(A)| A & |(B)| B \rlap{\text{.}} \\
};

\draw[->,dashed]  (X) to                  (Y);
\draw[cof]        (A) to node[below]{$i$} (B);
\draw[fib]        (X) to                  (A);
\draw[fib,dashed] (Y) to                  (B);
\pb{X}{B};
\end{tikzeq}
We have the solid part of the diagram
\begin{tikzeq*}
\matrix[diagram]
{
  |(X0)| X        &                  & |(Y0)| Y &              \\
                  & |(X1)| (r i)^* X &          & |(Y1)| r^* X \\
  |(A)|  A        &                  & |(B)|  B &              \\
};

\draw[->]        (X0) to node[above right] {$\heq$} (X1);
\draw[->,dashed] (Y0) to node[above right] {$\heq$} (Y1);
\draw[->]        (A)  to node[below]       {$i$}    (B);

\draw[fib] (X0) to (A);
\draw[fib] (X1) to (A);

\draw[fib,dashed]        (Y0) to (B);
\draw[fib]               (Y1) to (B);

\draw[->,dashed] (X0) to (Y0);
\draw[->,over]   (X1) to (Y1);
\end{tikzeq*}
where the squares going from left to right are pullbacks.
For the fiberwise homotopy equivalence on the left, we apply \cref{weak-equivalence-from-fibration-over-line} to the pullback of $X$ along the homotopy $\simp{1} \times A \to A$ relating $r i$ and $\id_A$.
We then complete the diagram using equivalence extension, \cref{equivalence-extension}.
\end{proof}

\begin{corollary}[Fibration extension] \label{fibration-extension}
  In cofibrant simplicial sets, fibrations extend along trivial cofibrations.
\end{corollary}

\begin{proof}
  We work exclusively in cofibrant simplicial sets.
  Consider the class of cofibrations that fibrations extend along.
  By \cref{fibration-generating-extension}, this includes the horn inclusions (they are cofibrations between cofibrant objects by \cref{triv-cof-cof,horn-cofibrant} and strong homotopy equivalences by \cref{horn-inclusion-is-she}).
  Using the presentation of trivial cofibrations given by \cref{thm:wfs-via-soa:fib} of \cref{thm:wfs-via-soa}, it remains to show this class is closed under coproducts, pushouts, sequential colimits, and codomain retracts.
  In the case of a codomain retract $A \to B'$ of $A \to B$, we simply extend along $A \to B$ and then pull back along $B' \to B$ (cofibrancy is assured by \cref{cofibrant-finite-limits}); by pullback pasting, this gives the required extension along $A \to B'$.
  In the other cases, all involved colimits in simplicial sets are van Kampen by \cref{sSet-van-Kampen} (since cofibrations are monomorphisms).

  For a coproduct of cofibrations $A_s \to B_s$, consider a fibration $X \to \bigcoprod_s A_s$.
  We pull it back to each $A_s$, extend separately for each $s$ along $A_s \to B_s$, and then take the coproduct of the resulting fibrations to get a map $Y \to \bigcoprod_s B_s$.
  Since $\bigcoprod_s B_s$ is van Kampen, this map pulls back to the individual fibrations over $B_s$ for each $s$.
  So it is a fibration by \cref{van-kampen-colim-of-structured-fibrations} of \cref{colim-of-structured-fibrations}.
  Using that $\bigcoprod_s A_s$ is van Kampen, the map also pulls back to $X \to \bigcoprod_s A_s$.

  In the remaining two cases, it is convenient to improve fibration extension to \emph{structured fibration extension}.
  Recall that a fibration is a map together with a choice of lifts against horn inclusions; recall also the notion of structure morphism of fibrations from \cref{sec:ssets}.
  In any fibration extension square~\eqref{fibration-extension-square}, also the upper horizontal map is a cofibration by \cref{pullback-of-cofibrations}, hence the horizontal maps are levelwise decidable.
  Using \cref{aligning-structured-fibrations}, we can choose lifts for the right map such that the square becomes a structure morphism of fibrations.
  We note finally that the colimits in question are all van Kampen (since cofibrations are monomorphisms).

  For a pushout $A' \cto B'$ of a cofibration $A \cto B$, consider a fibration $X' \fto A'$.
  We first pull it back to a fibration $X \fto A$ (using \cref{cofibrant-finite-limits} for cofibrancy of $X$).
  Note that the canonically induced choice of lifts for $X \fto A$ makes this pullback square into a structure morphism of fibrations.
  We then use structured fibration extension along $A \to B$ to produce a fibration $Y \fto B$.
  Finally, we take the pushout of all three fibrations.
  Since the given pushout is van Kampen, this map pulls back to the individual fibrations.
  So it is a fibration by \cref{van-kampen-colim-of-structured-fibrations} of \cref{colim-of-structured-fibrations}.

  For a sequential colimit of cofibrations $A_0 \cto A_1 \cto \ldots$, consider a fibration with target $A_0$.
  We recursively use structured fibration extension to produce a fibration with target $A_k$ for each $k$ and then take the sequential colimit of the resulting fibrations.
  Since the given sequential colimit is van Kampen, this map pulls back to the fibration over $A_k$ for each $k$.
  So it is a fibration by \cref{van-kampen-colim-of-structured-fibrations} of \cref{colim-of-structured-fibrations}.

  Note that the domain of the extended fibrations in these three cases is cofibrant by construction.
  This uses that the top map in any fibration extension square~\eqref{fibration-extension-square} is a cofibration.
\end{proof}

\subsection{Conclusion of the second proof}
\label{sec:model-structure-second-proof}

This subsection is entirely formal.
It is divided into two parts.
In the first part, ending with \cref{comparison-sattler-model-structure}, we develop the properties of the model structure on \emph{cofibrant simplicial sets}.
We define its weak equivalences from homotopy equivalences via fibrant replacement.
That trivial (co)fibrations and acyclic (co)fibrations coincide is deduced from the Frobenius and (trivial) fibration extension properties, all in cofibrant simplicial sets.

In the second part, we extend to the desired model structure on \emph{simplicial sets} (\cref{model-structure}).
We define its weak equivalences from the above ones via cofibrant replacement.
That trivial (co)fibrations and acyclic (co)fibrations coincide reduces to the cofibrant case using the restricted Frobenius property and a cancellation property of trivial fibrations.

\medskip

We define the notion of \emph{strong fibrant replacement} dual to the notion of strong cofibrant replacement of \cref{whes}.
A strong fibrant replacement of a simplicial set $X$ is a Kan complex $\bar X$ equipped with a trivial cofibration $X \to \bar X$.
Note that $\bar X$ is cofibrant if $X$ is cofibrant.
A strong fibrant replacement of a map $f \from X \to Y$ is a map $\bar f \from \bar X \to \bar Y$ equipped with a square
\begin{tikzeq*}
\matrix[diagram]
{
  |(X)| X & |(bX)| \bar X \\
  |(Y)| Y & |(bY)| \bar Y \\
};

\draw[->] (X)  to node[left]  {$f$}      (Y);
\draw[->] (bX) to node[right] {$\bar f$} (bY);

\draw[ano] (X) to (bX);
\draw[ano] (Y) to (bY);
\end{tikzeq*}
where $\bar X$ and $\bar Y$ are Kan complexes and both horizontal maps are trivial cofibrations.
Strong fibrant replacements can be constructed using \cref{thm:wfs-via-soa}.

Using strong fibrant replacement, we may lift the notion of homotopy equivalence for cofibrant Kan complexes to a notion of weak equivalence in cofibrant simplicial sets.
We say that a map in cofibrant simplicial sets is a \emph{weak homotopy equivalence} if it has a strong fibrant replacement that is a homotopy equivalence.

\begin{lemma} \label{whe-basics-cof}
In cofibrant simplicial sets:
\begin{enumerate}
\item \label{whe-basics-cof:independent}
the notion of weak homotopy equivalence does not depend on the choice of strong fibrant replacements,
\item \label{whe-basics-cof:compatible}
a map between Kan complexes is a weak homotopy equivalence if and only if it is a homotopy equivalence,
\item \label{whe-basics-cof:2-out-of-6}
weak homotopy equivalences satisfy 2-out-of-6.
\end{enumerate}
\end{lemma}

\begin{proof}
\Cref{whe-basics-cof:independent} holds by the dual of the argument of \cref{whe-Kan-independent} (considering strong fibrant replacement instead of strong cofibrant replacement).
This uses that trivial cofibrations between cofibrant Kan complexes are homotopy equivalences by \cref{trivial-cofibration-sdr} of \cref{trivial-fibration-whe-cof}.
With this, we prove the remaining parts.
For \cref{whe-basics-cof:compatible}, we take the map itself as its strong cofibrant replacement.
For \cref{whe-basics-cof:2-out-of-6}, functorial fibrant replacement (via \cref{thm:wfs-via-soa}) creates weak homotopy equivalences from homotopy equivalences between cofibrant Kan complexes.
The latter satisfy 2-out-of-6 by \cref{he-closure:2-out-of-6} of \cref{he-closure}, hence so do the former.
\end{proof}

\begin{lemma} \label{trivial-fibration-whe-cof}
Let $f \from X \to Y$ be a fibration in cofibrant simplicial sets.
Then $f$ is a trivial fibration if and only if it is a weak homotopy equivalence.
\end{lemma}

\begin{proof}
Assume first that $f \from X \fto Y$ is a trivial fibration.
We take a strong fibrant replacement $Y \ato \bar{Y}$ and extend the given trivial fibration along it using \cref{trivial-fibration-extension}, obtaining a pullback square
\begin{tikzeq*}
\matrix[diagram]
{
  |(X)| X & |(bX)| \bar X                 \\
  |(Y)| Y & |(bY)| \bar Y \rlap{\text{.}} \\
};

\draw[->]   (X)  to (bX);
\draw[ano]  (Y)  to (bY);
\draw[tfib] (X)  to (Y);
\draw[tfib]   (bX) to (bY);
\pb{X}{bY};
\end{tikzeq*}
By the Frobenius property in cofibrant simplicial sets (\cref{frobenius}), $X \to \bar{X}$ is a trivial cofibration, so $X \to \bar{X}$ is a strong fibrant replacement.
As a trivial fibration between cofibrant objects, $\bar{X} \to \bar{Y}$ is a homotopy equivalence by \cref{trivial-fibration-shrinkable} of \cref{trivial-to-she}.
Thus, $X \to Y$ is a weak homotopy equivalence.

Conversely, assume that $f \from X \to Y$ is a weak homotopy equivalence.
We extend $f$ along $Y \to \bar{Y}$ using \cref{fibration-extension}, obtaining a pullback square
\begin{tikzeq*}
\matrix[diagram]
{
  |(X)| X & |(bX)| \bar X                  \\
  |(Y)| Y & |(bY)| \bar{Y} \rlap{\text{.}} \\
};

\draw[->,dashed]  (X)  to (bX);
\draw[ano]        (Y)  to (bY);
\draw[fib]        (X)  to (Y);
\draw[fib,dashed] (bX) to (bY);
\pb{X}{bY};
\end{tikzeq*}
By \cref{frobenius}, $X \to \bar{X}$ is a trivial cofibration.
This makes the above square a strong fibrant replacement.
So $\bar{X} \to \bar{Y}$ is a homotopy equivalence by \cref{whe-basics-cof:independent} of \cref{whe-basics-cof}.
Since it is also a fibration between cofibrant Kan complexes, it is a trivial fibration by \cref{fiberwise-acyclic-fibration-trivial-Kan}.
As its pullback, so is $f$.
\end{proof}

\begin{lemma} \label{trivial-cofibration-whe-cof}
Let $f \from A \cto B$ be a cofibration in cofibrant simplicial sets.
Then $f$ is a trivial cofibration if and only if it is a weak homotopy equivalence.
\end{lemma}

\begin{proof}
Assume that $f$ is a trivial cofibration.
Take a strong fibrant replacement $B \ato \bar{B}$.
The identity on $\bar{B}$ is a strong fibrant replacement of $f$, which is thus a weak homotopy equivalence.

The converse direction is a formal consequence of the forward direction and \cref{trivial-fibration-whe-cof}.
In detail, factor $f$ as a trivial cofibration $j$ followed by a fibration $p$.
Assuming $f$ is a weak homotopy equivalence, so is $p$ by 2-out-of-3.
By \cref{trivial-fibration-whe-cof}, $p$ is a trivial fibration.
Since $f$ lifts against its second factor $p$, it is a retract of its first factor $j$, a trivial cofibration.
\end{proof}

The above statements make weak homotopy equivalences and the two \wfs{}s under consideration a model structure on cofibrant simplicial sets.

\begin{remark} \label{comparison-sattler-model-structure}
The approach of~\cite{Sattler} (in the setting of cofibrant simplicial sets) uses a different definition of weak homotopy equivalences: those maps factoring as a trivial cofibration followed by a trivial fibration.
Then acylic (co)fibrations coincide with trivial (co)fibrations for formal reasons (proved here as \cref{trivial-cofibration-whe-cof,trivial-fibration-whe-cof}) and the goal shifts to showing that weak homotopy equivalences satisfy 2-out-of-3 (holding here essentially by construction).
However, this is only a superficial difference in setup: the overall work ends up being the same.
We have chosen our current definitions to achieve greater similarity with the below extension to arbitrary simplicial sets (via strong cofibrant replacement) and the setup in \cref{whes}.
\end{remark}

\medskip

We now extend our notion of weak homotopy equivalence to arbitrary simplicial sets, not necessarily cofibrant.
Recall the notion of \emph{strong cofibrant replacement} of \cref{whes}, dual to that of strong fibrant replacement.
We say that a map in simplicial sets is a \emph{weak homotopy equivalence} if it has a strong cofibrant replacement that is a weak homotopy equivalence in cofibrant simplicial sets.

\begin{lemma} \label{whe-basics}
In simplicial sets:
\begin{enumerate}
\item \label{whe-basics:independent}
the notion of weak homotopy equivalence does not depend on the choice of strong cofibrant replacements,
\item \label{whe-basics:compatible}
a map between cofibrant simplicial sets is a weak homotopy equivalence (in the above sense) if and only if it is a weak homotopy equivalence in cofibrant simplicial sets,
\item \label{whe-basics:2-out-of-6}
weak homotopy equivalences satisfy 2-out-of-6.
\end{enumerate}
\end{lemma}

\begin{proof}
This is dual to \cref{whe-basics-cof}.
\Cref{whe-basics-cof:independent} holds by the argument of \cref{whe-Kan-independent} (replace Kan complexes by arbitrary simplicial sets, \ref{whe-cof-Kan-def} by weak homotopy equivalence in cofibrant simplicial sets, and \ref{whe-Kan-def} by weak homotopy equivalences).
This uses that trivial fibrations in cofibrant simplicial sets are weak homotopy equivalences by \cref{trivial-fibration-whe-cof}.
With this, \cref{whe-basics:compatible,whe-basics:2-out-of-6} reduce to the respective parts of \cref{whe-basics-cof}.
\end{proof}

In light of \cref{whe-basics:compatible} of \cref{whe-basics}, the meanings of the term ``weak homotopy equivalence'' in cofibrant simplicial sets and simplicial sets are compatible.
We may thus use the term without qualification.

\begin{lemma} \label{trivial-fibration-whe}
Let $f \from X \to Y$ be a fibration in simplicial sets.
Then $f$ is a trivial fibration if and only if it is a weak homotopy equivalence.
\end{lemma}

\begin{proof}
Assume that $f$ is a trivial fibration.
Take a strong cofibrant replacement $\tilde{X} \tfto X$.
The identity on $\bar{X}$ is a strong cofibrant replacement of $f$, which is thus a weak homotopy equivalence.

Conversely, let $f$ be an acyclic fibration.
Take a strong cofibrant replacement $\tilde{Y} \tfto Y$ followed by a strong cofibrant replacement of the pullback of $X$ along this map:
\begin{tikzeq*}
\matrix[diagram]
{
  |(tX)| \tilde X & |(Xp)| X'       & |(X)| X                 \\
                  & |(tY)| \tilde Y & |(Y)| Y \rlap{\text{.}} \\
};

\draw[tfib] (tX) to (Xp);
\draw[fib]  (Xp) to (tY);

\draw[fib] (X) to node[right] {$\weq$} (Y);

\draw[tfib] (Xp) to (X);
\draw[tfib] (tY) to (Y);

\pb{Xp}{Y};
\end{tikzeq*}
Since $X \to Y$ is a weak homotopy equivalence, its strong cofibrant replacement $\tilde{X} \to \tilde{Y}$ is a weak homotopy equivalence between cofibrant objects.
By \cref{trivial-fibration-whe-cof}, it is a trivial fibration.
Applying \cref{trivial-fibration-cancellation} in the above diagram, we obtain that $X \to Y$ is a trivial fibration.
\end{proof}

\begin{lemma} \label{trivial-cofibration-whe}
Let $f \from A \to B$ be a cofibration in simplicial sets.
Then $f$ is a trivial cofibration if and only if it is a weak homotopy equivalence.
\end{lemma}

\begin{proof}
Let $A \ato B$ be a trivial cofibration.
Take a strong cofibrant replacement $\tilde{B} \tfto B$.
Consider the pullback square
\begin{tikzeq*}
\matrix[diagram]
{
  |(tA)| \tilde A & |(A)|         A                 \\
  |(tB)| \tilde B & |(B)|         B \rlap{\text{.}} \\
};

\draw[tfib] (tA) to (A);
\draw[tfib] (tB) to (B);

\draw[->] (tA) to (tB);
\draw[ano] (A)  to (B);

\pb{tA}{B};
\end{tikzeq*}
Note that $\tilde A$ is cofibrant by \cref{pullback-of-cofibrations}.
In particular $\tilde A \to \tilde B$ is a strong cofibrant replacement of $A \to B$.
Since $\tilde B$ is cofibrant, the restricted Frobenius property (\cref{frobenius}) makes $\tilde A \to \tilde B$ a trivial cofibration.
By \cref{trivial-cofibration-whe-cof}, it is a weak homotopy equivalence in cofibrant simplicial sets.
This makes $A \to B$ a weak homotopy equivalence.

The converse direction follows from the forward direction together with \cref{trivial-fibration-whe} using 2-out-of-3 by the retract argument as in \cref{trivial-cofibration-whe-cof}.
\end{proof}

\begin{theorem} \label{model-structure}
The cofibrations, Kan fibrations, and weak homotopy equivalences form a model structure on simplicial sets.
\end{theorem}

\begin{proof}
Weak homotopy equivalences satisfy 2-out-of-6 by \cref{whe-basics:2-out-of-6} of \cref{whe-basics}.
We have established the two \wfs{}s in \cref{thm:wfs-via-soa}.
By \cref{trivial-fibration-whe}, acyclic fibrations coincide with trivial fibrations.
By \cref{trivial-cofibration-whe}, trivial cofibrations coincide with acyclic cofibrations.
\end{proof}

\begin{remark}
Any model structure is determined by its classes of cofibrations and fibrations.
So the model structure of \cref{model-structure} coincides with the one of \cref{thm:main-thm-first-proof}.
It follows, after the fact, that all considered notions of weak homotopy equivalence are equivalent in their respective subcategories of cofibrant and/or fibrant simplicial sets.
\end{remark}

\begin{remark} \label{thm:second-diff-henry}
Following up on \cref{thm:first-diff-henry}, our second construction of the constructive Kan-Quillen model structure differs from~\cite{Henry-qms} in that it avoids the~$\Ex^\infty$ functor and semisimplicial homotopy theory.
Instead, following~\cite{Sattler}, it is based on a restricted version of the Frobenius property~\cite{Gambino-Sattler} and the equivalence extension property in cofibrant simplicial sets.
\end{remark}

  \section{Right and left properness}
  \label{sec:properness}

\begin{proposition-s}\label{Kan--Quillen-right-proper}
  The Kan--Quillen model structure on simplicial sets is right proper.
\end{proposition-s}

\begin{proof}[Proof ($\Ex^\infty$ argument).]
  Consider a fibration $p \from X \fto Y$ and a \whe{} $f \from B \to Y$.
  Form the diagram
  \begin{tikzeq*}
  \matrix[diagram]
  {
        |(A)|  A                   & & & |(X)|  X                   & & & \\
    &   |(tA)| \tilde A            & & & |(tX)| \tilde X            & &   \\
    & & |(EA)| \Ex^\infty \tilde A & & & |(EX)| \Ex^\infty \tilde X &     \\
        |(B)|  B                   & & & |(Y)|  Y                   & & & \\
    &   |(tB)| \tilde B            & & & |(tY)| \tilde Y            & &   \\
    & & |(EB)| \Ex^\infty \tilde B & & & |(EY)| \Ex^\infty \tilde Y &     \\
  };

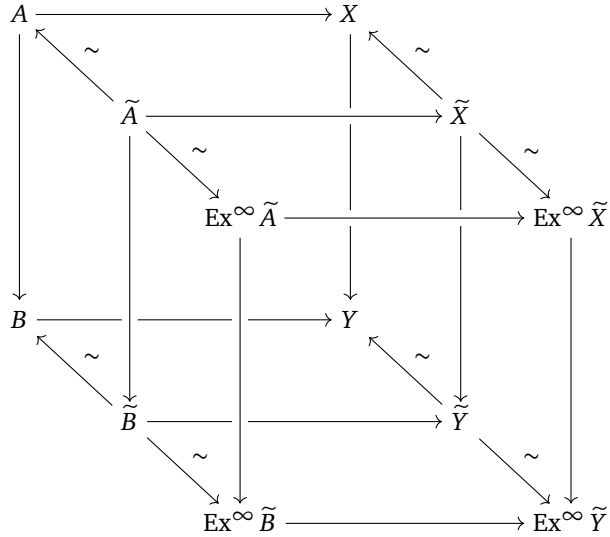
  \draw[->] (A) to (B);
  \draw[->] (X) to (Y);
  \draw[->] (A) to (X);
  \draw[->] (B) to (Y);

  \draw[->,over] (tA) to (tB);
  \draw[->]      (tX) to (tY);
  \draw[->,over] (tA) to (tX);
  \draw[->]      (tB) to (tY);

  \draw[->,over] (EA) to (EB);
  \draw[->]      (EX) to (EY);
  \draw[->,over] (EA) to (EX);
  \draw[->]      (EB) to (EY);

  \draw[->] (tA) to node[above right] {$\weq$} (A);
  \draw[->] (tB) to node[above right] {$\weq$} (B);
  \draw[->] (tX) to node[above right] {$\weq$} (X);
  \draw[->] (tY) to node[above right] {$\weq$} (Y);

  \draw[->] (tA) to node[above right] {$\weq$} (EA);
  \draw[->] (tX) to node[above right] {$\weq$} (EX);
  \draw[->] (tB) to node[above right] {$\weq$} (EB);
  \draw[->] (tY) to node[above right] {$\weq$} (EY);
  \end{tikzeq*}
  as follows.
  \begin{enumerate}
  \item The back square is a pullback.
  \item $\tilde Y \weto Y$ is a strong cofibrant replacement.
  \item $\tilde B \weto B$ and $\tilde X \weto X$ are strong cofibrant replacements,
    the latter chosen so that $\tilde X \to X \pull_{Y} \tilde Y$ is a trivial fibration.
  \item The middle square is a pullback.
  \item The front square is obtained by applying $\Ex^\infty$ to the middle one and the transformation between them
    is $\nu^\infty$ of \cref{sec:subdivision-and-ex}.
  \end{enumerate}
  Denote the maps $\tilde X \to \tilde Y$ and $\tilde B \to \tilde Y$ by $\tilde p$ and $\tilde f$, respectively.
  Note that $\tilde A \to A$ is a trivial fibration as the composite of pullbacks of $\tilde B \to B$ and $\tilde X \to X \pull_{Y} \tilde Y$.
  Thus all middle to back maps are \whe{}s
  by \cref{whe-trivial-fibration} of \cref{whe-arbitrary-properties}.
  Moreover, $\tilde A$ is cofibrant by \cref{cofibrant-finite-limits} so
  \cref{Ex-infty-whe} implies that all middle to front maps are also \whe{}s.
  It follows that $\Ex^\infty \tilde f$ is a \whe{} since $f$ is.
  Moreover, the front square is a pullback by \cref{Ex-infty-finite-limits} of \cref{Ex-infty-properties},
  $\Ex^\infty \tilde p$ is a Kan fibration by \cref{Ex-infty-fibration} of \cref{Ex-infty-properties} and
  all objects in the front face are Kan complexes by \cref{Ex-infty-fibrant-replacement} of \cref{Ex-infty-properties}.
  It follows that $\Ex^\infty A \to \Ex^\infty X$ is a \whe{}
  by \cref{fibration-category-Kan} and \cite{rb}*{Lemma~1.4.2~(2b)}
  and thus so is $A \to X$.
\end{proof}

\begin{proof}[Proof (Frobenius property argument).]
  It suffices to show that the pullback of a trivial cofibration $B \to Y$ along a fibration $X \to Y$ is a weak equivalence.
  Consider a further pullback along a cofibrant replacement $\tilde{X} \to X$:
  \begin{tikzeq*}
  \matrix[diagram]
  {
    |(Ap)| A'       & |(A)| A & |(B)| B                 \\
    |(tX)| \tilde X & |(X)| X & |(Y)| Y \rlap{\text{.}} \\
  };

  \draw[->] (A)  to (X);
  \draw[->] (Ap) to (tX);

  \draw[cof] (B) to node[right] {$\weq$} (Y);

  \draw[->]  (A) to (B);
  \draw[fib] (X) to (Y);

  \draw[fib] (Ap) to node[above] {$\weq$} (A);
  \draw[fib] (tX) to node[below] {$\weq$} (X);
  \end{tikzeq*}
  Since $\tilde{X}$ is cofibrant, $A' \to \tilde{X}$ is a trivial cofibration by the restricted Frobenius property (\cref{frobenius}).
  Then $A \to X$ is a weak equivalence by 2-out-of-3.
\end{proof}

Semisimplicial sets are presheaves over $\Simp_\sharp$,
the subcategory of face operators in $\Simp$.
Let $U$ be the forgetful functor from simplicial to semisimplicial sets and
let $L$ be its left adjoint, with unit $\eta$ and counit $\epsilon$.

Let $\ssimp{m}$ denote the semisimplicial set represented by $[m]$;
its boundary $\bdssimp{m}$ is obtained by removing the top simplex $\id_{[m]}$.
The category of semisimplicial sets carries a \wfs{} generated by
the boundary inclusions $\bdssimp{m} \ito \ssimp{m}$.
A \emph{cofibration} is a morphism of the left class of this \wfs{}.

\begin{lemma-s}\label{semisimplicial-cofibration-levelwise-decidable}
  A semisimplicial map is a cofibration \iff{}
  it is a levelwise decidable inclusion.
\end{lemma-s}

\begin{proof}
  Cofibrations coincide with Reedy decidable inclusions by the same argument as
  in \cref{cofibration-wfs}.
  Since $\Simp_\sharp$ is a direct category,
  the latter are the same as levelwise decidable inclusions.
\end{proof}

\begin{corollary-s} \label{L-cofibration}
The functor $L$ preserves cofibrations.
\end{corollary-s}

\begin{proof}
It suffices to check that $L$ preserves boundary inclusions.
Indeed, we have $\bdssimp{m} = \colim_{[k] \overset{\neq}{\ito} [m]} \ssimp{k}$ and thus $L \bdssimp{m} = \colim L \ssimp{k} = \colim \simp{k} = \bdsimp{m}$.
\end{proof}

\begin{corollary-s} \label{U-cofibration}
The functor $U$ sends levelwise decidable inclusions to cofibrations.
In particular, it preserves cofibrations.
\end{corollary-s}

\begin{proof}
The first claim follows from \cref{semisimplicial-cofibration-levelwise-decidable}.
With this, the second claim follows from \cref{cofibration-levelwise-decidable}.
\end{proof}

\begin{proposition-s}\label{LU-properties}
  The functor $L U$ satisfies the following conditions.
  \begin{enumerate}
  \item\label{LU-colimit}
    It preserves colimits.
  \item\label{LU-cofibrant}
    It takes values in cofibrant simplicial sets.
  \item\label{LU-cofibration}
    It preserves cofibrations.
  \item\label{LU-whe}
    For every simplicial set $X$,
    the counit $\epsilon_X \from L U X \to X$ is a \whe{}.
  \end{enumerate}
\end{proposition-s}

\begin{proof}
  \Cref{LU-colimit} holds since both $L$ and $U$ are left adjoints.
  For \cref{LU-cofibrant}, note that every semisimplicial set is cofibrant
  by \cref{semisimplicial-cofibration-levelwise-decidable} while
  $L$ preserves cofibrant objects by \cref{L-cofibration}.
  \Cref{LU-cofibration} follows directly from \cref{L-cofibration,U-cofibration}.
  We postpone the proof of \cref{LU-whe} until we show an auxiliary lemma.
\end{proof}

\begin{remark-s} \label{cofibrant-replacement-variant}
In \cref{sec:explicit-cofibrant-replacement}, the verification of properties of the cofibrant replacement functor $T$ makes use of another cofibrant replacement functor, sending a simplicial set $X$ to $N (\Delta_\sharp \slice U X)$, the nerve of the category of elements of the semisimplicial set underlying $X$.
This functor factors as $L U$ followed by subdivision.
\end{remark-s}

Semisimplicial sets carry a closed symmetric monoidal product called the \emph{geometric product}.
It is uniquely determined by setting the geometric product of $\ssimp{m}$ and $\ssimp{n}$ to the semisimplicial set consisting of the non-degenerate simplices of $\simp{m} \times \simp{n}$.
More precisely, if we denote that semisimplicial set by $P_{m,n}$, then the geometric product of semisimplicial sets $X$ and $Y$ is
the coend $\coend^{m,n} X_m \times Y_n \times P_{m,n}$.
It follows that the geometric product preserves colimits in each variable.
Since $L$ also preserves colimits and satisfies $L P_{m,n} = \simp{m} \times \simp{n}$, it follows that it is monoidal, \ie,
it carries the geometric product to the cartesian product.

\begin{lemma-s}\label{LU-trivial-fibration}
  The functor $L U$ sends trivial fibrations to weak equivalences.
\end{lemma-s}

\begin{proof}
  Given a trivial fibration $p \from X \to Y$, we will show that
  $L U p \from L U X \to L U Y$ is a homotopy equivalence.
  It is then a weak equivalence \eg by \cref{whe-he} of \cref{whe-arbitrary-properties}.

  Note that $L U X$ and $L U Y$ are cofibrant
  by \cref{LU-cofibrant} of \cref{LU-properties}.
  We take a lift
  \begin{tikzeq*}
  \matrix[diagram]
  {
    |(e)|   \emptyset & |(X)| X                 \\
    |(LUY)| L U Y     & |(Y)| Y \rlap{\text{.}} \\
  };

  \draw[->]  (e) to (X);
  \draw[cof] (e) to (LUY);

  \draw[fib] (X) to node[right] {$p$} node[left] {$\weq$} (Y);
  \draw[->] (LUY) to node[below] {$\epsilon_Y$} (Y);

  \draw[->,dashed] (LUY) to (X);
  \end{tikzeq*}
  Transposing the square using the adjunction $L \adj U$ and
  applying $L$ gives a section $s \from L U Y \to L U X$ of $L U p$.

  Using \cref{ano-cof-pushout-product} of \cref{second-pushout-product}, we take a lift
  \begin{tikzeq*}
  \matrix[diagram,column sep={8em,between origins}]
  {
    |(bLUX)| L U X \times \bdsimp{1} &                                & |(LUX)| L U X & |(X)| X                 \\
    |(sLUX)| L U X \times   \simp{1} & |(sLUY)| L U Y \times \simp{1} & |(LUY)| L U Y & |(Y)| Y \rlap{\text{.}} \\
  };

  \draw[cof] (bLUX) to (sLUX);
  \draw[->] (sLUX) to (sLUY);
  \draw[->] (sLUY) to (LUY);

  \draw[->] (bLUX) to node[above] {$[s (L U p), \id]$}      (LUX);
  \draw[->] (LUX)  to node[above] {$\epsilon_X$}            (X);
  \draw[->] (LUY)  to node[below] {$\epsilon_Y$}            (Y);
  \draw[fib] (X)    to node[right] {$p$} node[left] {$\weq$} (Y);

  \draw[->,dashed,shorten >=0.5mm] (sLUX) to node[above left] {$u'$} (X.215);
  \end{tikzeq*}
  Note that the left arrow lies in the essential image of $L$
  (it is isomorphic to the image of
  the geometric product of $\bdsimp{1} \to \simp{1}$ with $U X$),
  and so does the top left arrow (by construction of $s$).
  Using terminality of $(U X, \epsilon_X)$ in $L \downarrow X$ and applying $L$,
  we obtain the dashed arrow in the commuting triangle
  \begin{tikzeq*}
  \matrix[diagram,column sep={12em,between origins}]
  {
    |(bLUX)| L U X \times \bdsimp{1} & |(LUX)| L U X \rlap{\text{.}} \\
    |(sLUX)| L U X \times   \simp{1} &                               \\
  };

  \draw[cof] (bLUX) to (sLUX);
  \draw[->] (bLUX) to node[above] {$[s (L U p), \id]$} (LUX);
  \draw[->,dashed] (sLUX) to node[below right] {$u$} (LUX);
  \end{tikzeq*}
  We thus have $s (L U p) \htp \id_{L U X}$ and
  $(L U p) s = \id_{L U Y}$ as desired.
\end{proof}

\begin{remark-s}
The proof of \cref{LU-trivial-fibration} is the combination and unfolding of the following elementary observations regarding semisimplicial sets, which carry notions of homotopy and \he{} based on the geometric product:
\begin{itemize}
\item $U$ preserves trivial fibrations (defined also in semisimplicial sets as lifting against boundary inclusions),
\item for a simplicial set $X$, the identity of $UX$ admits an endohomotopy,
\item in semisimplicial sets, every trivial fibration $Y \to X$ with an endohomotopy on the identity on $X$ extends to a homotopy equivalence,
\item $L$ preserves homotopy equivalences.
\end{itemize}
\end{remark-s}

\begin{proof}[Proof of \cref{LU-whe} of \cref{LU-properties}]
  First, note that $\epsilon_{\simp{m}} \from L U \simp{m} \to \simp{m}$ is
  a \whe{} for all $m$.
  Indeed, $L U \simp{m}$ is the nerve of a category $[m]'$ which
  is obtained from $[m]$ by
  adjoining an idempotent endomorphism to every object that
  acts trivially on morphisms of $[m]$.
  This category admits a natural transformation from
  the endofunctor constant at $0$ to the identity endofunctor.
  Thus, its nerve is contractible and the conclusion follows.
  Since we have already verified that $L U$ preserves colimits and cofibrations,
  the argument of \cref{cofibrant-replacement-whe} of \cref{cofibrant-replacement}
  shows that $\epsilon_X$ is a \whe{} for all cofibrant $X$.
  (That argument is an instance of the following general fact for a natural transformation $u$ between endofunctors of $\sSet$ that preserve colimits and cofibrations: if $u$ is a \whe{} on simplices, then $u$ is a \whe{} on all cofibrant objects.)
  The same follows for arbitrary $X$ by \cref{LU-trivial-fibration}.
\end{proof}

\begin{proposition-s}\label{Kan--Quillen-left-proper}
  The Kan-Quillen model structure on simplicial sets is left proper.
\end{proposition-s}

\begin{proof}
  The following argument uses
  the functor $T$ of \cref{sec:explicit-cofibrant-replacement} and
  its properties listed in \cref{cofibrant-replacement}.
  An alternative argument is obtained by substituting $L U$ for $T$ since
  the parallel properties have been verified in \cref{LU-properties}.

  Consider a cofibration $A \cto B$ and a \whe{} $A \to X$.
  Form the diagram
  \begin{tikzeq*}
  \matrix[diagram]
  {
      |(TA)| T A & & |(TX)| T X &\\
    & |(A)|  A   & & |(X)|  X    \\
      |(TB)| T B & & |(TY)| T Y &\\
    & |(B)|  B   & & |(Y)|  Y    \\
  };

  \draw[->,over] (TA) to (TB);
  \draw[->]      (TX) to (TY);
  \draw[->,over] (TA) to (TX);
  \draw[->]      (TB) to (TY);

  \draw[cof,over] (A) to (B);
  \draw[->]      (X) to (Y);
  \draw[->,over] (A) to node[above left,xshift=-8pt] {$\weq$} (X);
  \draw[->]      (B) to (Y);

  \draw[->] (TA) to node[above right] {$\weq$} (A);
  \draw[->] (TX) to node[above right] {$\weq$} (X);
  \draw[->] (TB) to node[above right] {$\weq$} (B);
  \draw[->] (TY) to node[above right] {$\weq$} (Y);
  \end{tikzeq*}
  where the front square is a pushout and the remainder of the cube is given by the copointed functor $T$.
  \Cref{cofibrant-replacement-whe} of \cref{cofibrant-replacement} implies that all back to front maps are \whe{}s.
  It follows that $TA \to TX$ is a \whe{} since $A \to X$ is.
  Moreover, the back square is a pushout by \cref{cofibrant-replacement-colimit} of \cref{cofibrant-replacement}, $TA \to TB$ is a cofibration by \cref{cofibrant-replacement-cofibration} of \cref{cofibrant-replacement} and all objects in the back face are cofibrant by \cref{cofibrant-replacement-cofibrant} of \cref{cofibrant-replacement}.
  It follows that $T B \to T Y$ is a \whe{} by \cref{cofibration-category-cofibrant} and \cite{rb}*{Lemma~1.4.2~(1b)}, and thus so is $B \to Y$.
\end{proof}

Using the cofibrant replacement functor $L U$, we obtain a slightly stronger statement.

\begin{proposition-s}
  In the category of simplicial sets, pushout along a levelwise decidable inclusion preserves weak homotopy equivalences.
  In particular, every pushout along a levelwise decidable inclusions is a homotopy pushout.
\end{proposition-s}

\begin{proof}
  By \cref{U-cofibration,L-cofibration}, the functor $L U$ carries levelwise decidable inclusions to cofibrations.
  With this, the conclusion follows exactly as in the proof of \cref{Kan--Quillen-left-proper}.
  (Here, we used that levelwise decidable inclusions are stable under pushout by \cref{decidable-wfs} to reduce to pushout squares as in that proof.)
\end{proof}

\begin{remark-s}\label{thm:properness-diff-henry}
  We conclude by describing the relationship between our proofs of properness
  and the one given in~\cite{Henry-qms}.
  For left properness, the proof in~\cite{Henry-qms} relies on the existence of
  a weak model structure (in the sense of~\cite{Henry-wms}) on
  the category of semisimplicial sets and uses its interaction with the adjunction $L \adj U$ with simplicial sets.

  Here, in one version of the argument, we use the comonad $L U$ of
  the adjunction $L \adj U$ to model cofibrant replacement,
  but using only the notion of cofibration in semisimplicial sets
  and not any further semisimplicial homotopy theory.
  In particular, we do not need to know that $U$ preserves cylinder objects.
  The other version also circumvents semisimplicial homotopy theory and uses
  a completely different cofibrant replacement functor $T$ in place of $L U$,
  although there is a relation as explained in \cref{cofibrant-replacement-variant}.

  For right properness, our first proof is very similar to
  the one in~\cite{Henry-qms}, but we included it for completeness,
  also because it follows naturally from our discussion of
  the $\Ex^\infty$ functor in \cref{sec:subdivision-and-ex}.
  Our second proof is entirely different and goes via the restricted Frobenius property proved directly in \cref{sec:frobenius-property}.
\end{remark-s}


\begin{bibdiv}
\begin{biblist}

\bib{Aczel}{article}{
  author={Aczel, Peter},
  title={The type theoretic interpretation of constructive set theory},
  conference={
    title={Logic Colloquium '77},
    address={Proc. Conf., Wroc\l aw},
    date={1977},
  },
  book={
    series={Stud. Logic Foundations Math.},
    volume={96},
    publisher={North-Holland, Amsterdam-New York},
  },
  date={1978},
  pages={55--66},
  label={Acz},
}

\bib{Aczel-Rathjen}{article}{
  author={Aczel, P.},
  author={Rathjen, M.},
  eprint={http://www1.maths.leeds.ac.uk/~rathjen/book.pdf},
  title={Notes on constructive set theory},
  date={2010},
}

\bib{awodey-ct}{book}{
  author={Awodey, Steve},
  title={Category theory},
  series={Oxford Logic Guides},
  volume={52},
  edition={2},
  publisher={Oxford University Press, Oxford},
  date={2010},
  pages={xvi+311},
  label={Aw},
}

\bib{van-den-Berg-Garner}{article}{
  author={van den Berg, Benno},
  author={Garner, Richard},
  title={Topological and Simplicial Models of Identity Types},
  journal={ACM Transactions on Computational Logic (TOCL)},
  volume={13},
  number={1},
  pages={1--44},
  date={2012},
  publisher={ACM New York, NY, USA}
}

\bib{bergner-rezk-elegant}{article}{
  author={Bergner, Julia E.},
  author={Rezk, Charles},
  title={Reedy categories and the $\varTheta$-construction},
  journal={Math. Z.},
  volume={274},
  date={2013},
  number={1-2},
  pages={499--514},
}

\bib{bcp}{article}{
   author={Bezem, Marc},
   author={Coquand, Thierry},
   author={Parmann, Erik},
   title={Non-constructivity in Kan simplicial sets},
   conference={
      title={13th International Conference on Typed Lambda Calculi and
      Applications},
   },
   book={
      series={LIPIcs. Leibniz Int. Proc. Inform.},
      volume={38},
      publisher={Schloss Dagstuhl. Leibniz-Zent. Inform., Wadern},
   },
   date={2015},
   pages={92--106},
   review={\MR{3448099}},
}

\bib{br}{article}{
  author={Brown, Kenneth S.},
  title={Abstract homotopy theory and generalized sheaf cohomology},
  journal={Trans. Amer. Math. Soc.},
  volume={186},
  date={1973},
  pages={419--458},
}

\bib{clw-extensive}{article}{
  author={Carboni, Aurelio},
  author={Lack, Stephen},
  author={Walters, R. F. C.},
  title={Introduction to extensive and distributive categories},
  journal={J. Pure Appl. Algebra},
  volume={84},
  date={1993},
  number={2},
  pages={145--158},
}

\bib{Cisinski}{article}{
  author={Cisinski, Denis-Charles},
  title={Les pr\'{e}faisceaux comme mod\`eles des types d'homotopie},
  language={French, with English and French summaries},
  journal={Ast\'{e}risque},
  number={308},
  date={2006},
  pages={xxiv+390},
}

\bib{CCHM}{article}{
  author={Cohen, Cyril},
  author={Coquand, Thierry},
  author={Huber, Simon},
  author={M\"{o}rtberg, Anders},
  title={Cubical type theory: a constructive interpretation of the
  univalence axiom},
  conference={
    title={21st International Conference on Types for Proofs and Programs},
  },
  book={
    series={LIPIcs. Leibniz Int. Proc. Inform.},
    volume={69},
    publisher={Schloss Dagstuhl. Leibniz-Zent. Inform., Wadern},
  },
  date={2018},
  pages={Art. No. 5, 34},
}

\bib{gz}{book}{
  author={Gabriel, P.},
  author={Zisman, M.},
  title={Calculus of fractions and homotopy theory},
  series={Ergebnisse der Mathematik und ihrer Grenzgebiete, Band 35},
  publisher={Springer-Verlag New York, Inc., New York},
  date={1967},
  pages={x+168},
}

\bib{Gambino-Henry}{article}{
  author={Gambino, N.},
  author={Henry, S.},
  eprint={https://arxiv.org/abs/1905.06281},
  title={Towards a constructive simplicial model of {U}nivalent {F}oundations},
  date={2019},
}

\bib{Gambino-Henry-Sattler-Szumilo}{article}{
  author={Gambino, N.},
  author={Henry, S.},
 author={Sattler, C.},
 author={Szumi{\l}o, K.},
  eprint={https://arxiv.org/abs/2102.06146},
  title={The effective model structure and $\infty$-groupoid objects},
  date={2021},
}

\bib{Gambino-Sattler}{article}{
  author={Gambino, N.},
  author={Sattler, C.},
  title={The {F}robenius property, right properness and uniform fibrations},
  journal={J. Pure Appl. Algebra},
  volume={221},
  date={2017},
  pages={No. 12, 3027--3068},
}

\bib{garner-lack-adhesive}{article}{
   author={Garner, Richard},
   author={Lack, Stephen},
   title={On the axioms for adhesive and quasiadhesive categories},
   journal={Theory Appl. Categ.},
   volume={27},
   date={2012},
   pages={No. 3, 27--46},
}

\bib{Goerss-Jardine}{book}{
  author={Goerss, Paul G.},
  author={Jardine, John F.},
  title={Simplicial homotopy theory},
  series={Progress in Mathematics},
  volume={174},
  publisher={Birkh\"{a}user Verlag, Basel},
  date={1999},
  pages={xvi+510},
}

\bib{Henry-wms}{article}{
  author={Henry, S.},
  eprint={https://arxiv.org/abs/1807.02650},
  title={Weak model categories in classical and constructive mathematics},
  date={2018},
}

\bib{Henry-qms}{article}{
  author={Henry, S.},
  eprint={https://arxiv.org/abs/1905.06160},
  title={A constructive account of the {K}an-{Q}uillen model structure and of {K}an's $\mathrm{Ex}^\infty$ functor},
  date={2019},
}

\bib{hovey}{book}{
  author={Hovey, M.},
  title={Model categories},
  series={Mathematical Surveys and Monographs},
  publisher={American Mathematical Society},
  date={2000},
}

\bib{Jardine}{article}{
  author={Jardine, J. F.},
  title={Boolean localization, in practice},
  journal={Doc. Math.},
  volume={1},
  date={1996},
  pages={No. 13, 245--275},
}

\bib{Johnstone}{book}{
  author={Johnstone, P. T.},
  title={Sketches of an Elephant. A {T}opos {T}heory {C}ompendium},
  publisher={Oxford University Press},
  date={2002}
}




\bib{Joyal}{article}{
  author={Joyal, A.},
  eprint={https://webusers.imj-prg.fr/~georges.maltsiniotis/ps/lettreJoyal.pdf},
  title={Letter to {A}. Grothendieck},
  date={1984},
}

\bib{joyal-collapses}{article}{
  author={Joyal, A.},
  title={Factorisation systems},
  eprint={https://ncatlab.org/joyalscatlab/show/Factorisation+systems},
  date={2020},
}

\bib{jt}{article}{ 
  author={Joyal, A.},
  author={Tierney, M.},
  title={Notes on simplicial homotopy theory},
  publisher={Centre de Recerca Mat\`ematica},
  journal={Centre de Recerca Mat\`ematica, Quadern},
  volume={47},
  date={2008},
  eprint={http://www.mat.uab.cat/~kock/crm/hocat/advanced-course/Quadern47.pdf},
}

\bib{Kapulkin}{article}{
  author={Kapulkin, C.},
  author={LeFanu Lumsdaine, P.},
  eprint={https://arxiv.org/abs/1211.2851v5},
  note={to appear in \emph{Journal of the European Mathematical Society}},
  title={The simplicial model of {U}nivalent {F}oundations (after {V}oevodsky)},
  date={2012},
  label={KL},
}

\bib{ltw}{article}{
  author={Latch, Dana May},
  author={Thomason, Robert W.},
  author={Wilson, W. Stephen},
  title={Simplicial sets from categories},
  journal={Math. Z.},
  volume={164},
  date={1979},
  number={3},
  pages={195--214},
}

\bib{Moss}{article}{
  title={Another approach to the {Kan}--{Quillen} model structure},
  author={Moss, S.},
  journal={Journal of Homotopy and Related Structures},
  volume={15},
  number={1},
  pages={143--165},
  year={2020},
}

\bib{oury-duality}{article}{
  author={Oury, David},
  title={On the duality between trees and disks},
  journal={Theory Appl. Categ.},
  volume={24},
  date={2010},
  pages={No. 16, 418--450},
}

\bib{may-ponto}{book}{
  author={May, J. P.},
  author={Ponto, K.},
  title={More concise algebraic topology -- Localization, completion, and model categories},
  series={Chicago Lectures in Mathematics},
  publisher={University of Chicago Press, Chicago, IL},
  date={2012},
  pages={xxviii+514},
}

\bib{Quillen}{book}{
  author={Quillen, Daniel G.},
  title={Homotopical algebra},
  series={Lecture Notes in Mathematics, No. 43},
  publisher={Springer-Verlag, Berlin-New York},
  date={1967},
  pages={iv+156 pp. (not consecutively paged)},
}

\bib{q}{article}{
   author={Quillen, Daniel},
   title={Higher algebraic $K$-theory. I},
   conference={
      title={Algebraic $K$-theory, I: Higher $K$-theories},
      address={Proc. Conf., Battelle Memorial Inst., Seattle, Wash.},
      date={1972},
   },
   book={
      publisher={Springer, Berlin},
   },
   date={1973},
   pages={85--147. Lecture Notes in Math., Vol. 341},
}

\bib{rb}{article}{
  author={R\u{a}dulescu-Banu, Andrei},
  title={Cofibrations in Homotopy Theory},
  date={2006},
  eprint={https://arxiv.org/abs/math/0610009v4},
}

\bib{riehl-verity-reedy}{article}{
   author={Riehl, Emily},
   author={Verity, Dominic},
   title={The theory and practice of Reedy categories},
   journal={Theory Appl. Categ.},
   volume={29},
   date={2014},
   pages={256--301},
}

\bib{Streicher}{article}{
  author={Streicher, T.},
  title={A model of type theory in simplicial sets: a brief introduction to
  Voevodsky's homotopy type theory},
  journal={J. Appl. Log.},
  volume={12},
  date={2014},
  number={1},
  pages={45--49},
  issn={1570-8683},
  label={Str},
}

\bib{Sattler}{article}{
  label={S1},
  author={Sattler, C.},
  eprint={https://arxiv.org/abs/1704.06911},
  title={The equivalence extension property and model categories},
  date={2017},
}

\bib{nlab-exhaustive}{article}{
  label={S2},
  author={Shulman, M.},
  author={others},
  title={Exhaustive category},
  eprint={https://ncatlab.org/nlab/revision/exhaustive+category/7},
  date={2019},
}

\bib{HoTT-book}{book}{
  author={The Univalent Foundations Program},
  title={Homotopy type theory---univalent foundations of mathematics},
  publisher={The Univalent Foundations Program, Princeton, NJ; Institute
  for Advanced Study (IAS), Princeton, NJ},
  date={2013},
  pages={xiv+589},
  label={HoTT},
}

\bib{Voevodsky-MSCS}{article}{
  author={Voevodsky, Vladimir},
  title={An experimental library of formalized mathematics based on the
  univalent foundations},
  journal={Math. Structures Comput. Sci.},
  volume={25},
  date={2015},
  number={5},
  pages={1278--1294},
}

\end{biblist}
\end{bibdiv}
\end{document}